\documentclass[10pt]{amsart}

\usepackage{amstext}
\usepackage{latexsym}
\usepackage{amssymb}
\usepackage{amsmath,amsthm,amscd}
\usepackage{enumitem}
\usepackage[pdftex, bookmarks=true, linkbordercolor={0 1 0}]{hyperref}
\usepackage{a4wide}
\usepackage[all]{xy}

\newcommand{\R}{\mathbb R}
\newcommand{\Q}{\mathbb Q}
\newcommand{\C}{\mathbb C}
\newcommand{\Z}{\mathbb Z}
\newcommand{\N}{\mathbb N}
\newcommand{\A}{\mathbb A}

\newcommand{\eps}{\varepsilon}
\newcommand{\minus}{\backslash}
\newcommand{\Aut}{\mathop{\rm Aut}\nolimits}

\newcommand{\Hom}{\mathop{\rm Hom}\nolimits}

\newcommand{\tr}{\mathop{\rm tr}\nolimits}
\newcommand{\vol}{\mathop{\rm vol}\nolimits}
\newcommand{\GL}{\mathop{\rm GL}\nolimits}

\newcommand{\Orth}{{\rm O}}

\newcommand{\SL}{\mathop{\rm SL}\nolimits}
\newcommand{\id}{\mathop{\rm id}\nolimits}

\newcommand{\res}{\mathop{\rm res}}

\newcommand{\diag}{\mathop{\rm diag}}

\newcommand{\val}{\mathop{\rm val}\nolimits}

\newcommand{\Frac}{\mathop{\rm Frac}\nolimits}

\newcommand{\DIV}{\mathop{\rm Div}\nolimits}

\newcommand{\Ad}{\mathop{{\rm Ad}}\nolimits}

\newcommand{\OOO}{\mathcal{O}}
\newcommand{\NNN}{\mathcal{N}}
\newcommand{\nnn}{\mathfrak{n}}
\newcommand{\ggG}{\mathfrak{g}}
\newcommand{\One}{\mathbf{1}}
\newcommand{\ppp}{\mathfrak{p}}
\newcommand{\cpt}{\mathbf{K}}
\newcommand{\aaa}{\mathfrak{a}}
\newcommand{\lcm}{\mathop{{\rm lcm}}}
\newcommand{\uuu}{\mathfrak{u}}
\newcommand{\SSS}{\mathcal{S}}
\newcommand{\FFF}{\mathcal{F}}
\newcommand{\PPP}{\mathcal{P}}
\newcommand{\mmm}{\mathfrak{m}}
\newcommand{\ooo}{\mathfrak{o}}
\newcommand{\MMM}{\mathcal{M}}
\newcommand{\EEE}{\mathcal{E}}
\newcommand{\RRR}{\mathcal{R}}
\newcommand{\ad}{\mathop{\rm ad}\nolimits}

\newcommand{\BBB}{\mathcal{B}}

\newcommand{\UUU}{\mathcal{U}}
\newcommand{\bbb}{\mathfrak{b}}

\theoremstyle{plain}
\newtheorem{theorem}{Theorem}[section]

\newtheorem{lemma}[theorem]{Lemma}
\newtheorem{cor}[theorem]{Corollary}
\newtheorem{proposition}[theorem]{Proposition}
\theoremstyle{remark}
\newtheorem{rem}[theorem]{Remark}
\newtheorem{example}[theorem]{Example}
\theoremstyle{definition}
\newtheorem{definition}[theorem]{Definition}
\newtheorem{conjecture}[theorem]{Conjecture}

\begin{document}
\title[Zeta Functions for $\GL(n)$]{Zeta Functions for the Adjoint Action of $\GL(n)$ and density of residues of Dedekind zeta functions}
\author{Jasmin Matz}
\address{Mathematisches Institut, Rheinische Friedrich-Wilhelms-Universit\"at Bonn, Endenicher Allee ~60, 53115 Bonn, Germany}
\email{matz@math.uni-bonn.de}
\begin{abstract}
We define zeta functions for the adjoint action of $\GL_n$ on its Lie algebra and study their analytic properties. 
For $n\leq 3$ we are able to fully analyse these functions, and recover the Shintani zeta function for the prehomogeneous vector space of binary quadratic forms for $n=2$. Our construction naturally yields a regularisation, which is necessary for the improvement of the  properties of these zeta function, in particular for the analytic continuation if $n\geq3$.

We further obtain upper and lower bounds on the mean value $X^{-\frac{5}{2}}\sum_{E} \res_{s=1}\zeta_E(s)$ as $X\rightarrow\infty$, where $E$ runs over totally real cubic number fields whose second successive minimum of the trace form on its ring of integers is bounded by $X$. To prove the upper bound we use our new zeta function for $\GL_3$. These asymptotic bounds are a first step towards a generalisation of density results obtained by Datskovsky in case of quadratic field extensions.
\end{abstract}
\maketitle
\tableofcontents 

\section{Introduction}
The first purpose of this paper is to provide another point of view for the construction of the Shintani zeta function $Z(s,\Psi)$ for the binary quadratic forms and to generalise this approach to the action of $\GL_1\times\GL_n$ on the Lie algebra $\mathfrak{gl}_n$. 
To improve its properties, $Z(s,\Psi)$ has to be  ``adjusted'' (cf.\ \cite{Yu92,Da93}), and the advantage of our approach is that a suitable modification (for $Z(s,\Psi)$ as well as the higher dimensional case) naturally emerges. 
The second purpose of this paper is to make a first step towards the generalisation of a result from \cite{Da93}: We prove upper and lower bounds on the density of residues of Dedekind zeta functions for totally real cubic number fields. For the upper bound we use our new zeta function for $n=3$.

There has been a long interest in zeta functions attached to group actions, in particular in the Shintani zeta functions attached to prehomogeneous vector spaces, cf.\ \cite{SaSh74,Sh75,Yu92,Ki03}. 
One basic example of a prehomogeneous vector space is the space of binary quadratic forms on which $\GL_1\times\GL_2$ acts by multiplication by scalars and by changing basis.
There are two natural generalisation of this space to higher dimensions corresponding to different viewpoints:
From the point of view of quadratic forms, the obvious generalisation is to consider $\GL_1\times\GL_n$ acting on quadratic forms in $n$ variables. This is again a prehomogeneous vector space studied, e.g., in \cite{Sh75,Suz79}.

On the other hand, we can equally well identify the space of binary quadratic forms with the Lie algebra $\mathfrak{sl}_2$ of $\SL_2$ so that the action of $\GL_2$ becomes the adjoint representation on $\mathfrak{sl}_2$. From this point of view, it seems more natural to generalise to higher dimensions by considering the action of $\GL_1\times\GL_n$ on $\mathfrak{sl}_n$ by letting $\GL_1$ act by multiplication of scalars and $\GL_n$ by the adjoint action.
This is the point of view we take in this paper.
 The problem is that this is not a prehomogeneous vector space if $n\geq3$ so that the general theory of Shintani zeta functions does not apply. 

Shintani zeta functions often turned out to be useful to obtain information on certain arithmetic quantities encoded in these zeta functions, cf.\ \cite{Sh75,WrYu,Da93}.
In particular, the Shintani zeta function $Z(s, \Psi)$ for the binary quadratic forms can be used to deduce density theorems for class numbers of binary quadratic forms as well as for residues of Dedekind zeta functions for quadratic field extensions, cf.\ \cite{Sh75,Da93}. We will later find that in our zeta function for $n=3$, the residues of the Dedekind zeta functions for cubic number fields are encoded. For general $n\geq 2$, one could find the respective objects for  number fields of degree $n$.

The paper consists of two main parts. The second part is independent of the techniques of the first one, we only use results from the first part.

To describe our results in more detail, let $n\geq2$, $G=\GL_n$ or $G=\SL_n$, and let accordingly $\ggG=\mathfrak{gl}_n$ or $\ggG=\mathfrak{sl}_n$ be the Lie algebra of $G$. Put $D=\dim\ggG$. Then $G$ acts on $\ggG$ by the adjoint action $\Ad$. 
Let $\ggG(\Q)_{\text{er}}$ denote the set of regular elliptic elements in $\ggG(\Q)$, i.e. matrices $X$ having an irreducible characteristic polynomial over $\Q$, and let $\OOO_{\text{er}}$ denote the set of orbits $[X]\subseteq \ggG(\Q)_{\text{er}}$ of regular elliptic elements under $\Ad G(\Q)$.

\subsection*{Part \ref{part_zeta}}
We generalise the zeta function $Z(s,\Phi)$ to higher dimensions by defining the ``main'' (or unregularised) zeta function for $G$ by
\[
 \Xi_{\text{main}}(s, \Phi)=
\int_{\Q^{\times}\backslash \A^{\times}}|\lambda|^{\sqrt{D}(s+\frac{\sqrt{D}-1}{2})}
\int_{G(\Q)\backslash G(\A)^1}\sum_{[X]\in \OOO_{\text{er}}}\Phi(\lambda \Ad x^{-1} X) dx~d^{\times}\lambda
\]
for $s\in\C$, $\Re s\gg0$, and  $\Phi:\ggG(\A)\longrightarrow\C $ a Schwartz-Bruhat function. 
We will see (cf.\ Theorem \ref{holom_intr} below) that this defines a holomorphic function for $\Re s>\frac{\sqrt{D}+1}{2}$.
For $n=2$ the function $\Xi_{\text{main}}(\cdot, \Phi)$ basically coincides with the (unmodified) Shintani zeta function from  \cite{Sh75,Yu92,Da93} (cf.\ \S \ref{connection_arthur_tf} and \cite{diss}).

To study $\Xi_{\text{main}}(\cdot, \Phi)$ one needs to regularise it in  a suitable way.
For $n=2$ a regularisation is needed to obtain a ``nice'' functional equation (cf.\ \cite{Yu92,Da93}), but for higher dimensions, the regularisation appears to be even more essential: 
Already for $n=3$, it seems that $\Xi_{\text{main}}(\cdot, \Phi)$ can not be continued to all of $\C$, cf.\ \cite[IV.iii]{diss}.  
Our method of regularisation is different from the one used for $Z(s,\Phi)$ so far: In  \cite{Yu92,Da93} smoothed Eisenstein series were used to cut off diverging integrals. In contrast to this we use a more geometric truncation process that is analogous to the one employed by Arthur for his trace formula; cf.\ also \cite{Le99} for a similar truncation for the Shintani zeta function for the binary quartic forms.
For this we use Chaudouard's trace formula for $\ggG$ (= truncated summation formula) from \cite{Ch02}: Let $\OOO$ denote the set of equivalence classes on $\ggG(\Q)$. This set corresponds bijectively to orbits of semisimple elements, cf.\ \S \ref{section_eq_classes}. Let $\nnn\in\OOO$ be the nilpotent variety in $\ggG$. 
One can attach to every $\ooo\in\OOO$ and to every truncation parameter $T$ in the coroot space $\aaa$ of $G$ a distribution $J_{\ooo}^T$ on the space of Schwartz-Bruhat functions $\Phi:\ggG(\A)\longrightarrow\C $, cf.\ \S \ref{section_distr}.
They are defined similar to  Arthur's distributions on the space of test functions on a reductive algebraic group appearing in Arthur's trace formula for the group.
We now define the regularised zeta function $\Xi^T(s,\Phi)$ as follows: If $\lambda\in\R$, set $\Phi_{\lambda}(x)=\Phi(\lambda x)$. Then 
\begin{equation}\label{def_zetafct_intr}
 \Xi^T(s, \Phi):=\int_{0}^{\infty}\lambda^{\sqrt{D}(s+\frac{\sqrt{D}-1}{2})}\sum_{\ooo\in\OOO\minus\{\nnn\}} J_{\ooo}^T(\Phi_{\lambda})d^{\times}\lambda
\end{equation}
provided this integral converges.
We need to extend this definition to non-smooth test functions $\Phi\in\SSS^{\nu}(\ggG(\A))$ for $\nu\in\Z$ sufficiently large, where $\SSS^{\nu}(\ggG(\A))$ denotes the space of functions $\ggG(\A)\longrightarrow\C$, which are of rapid decay, but only differentiable up to order $\nu$, cf.\ \S \ref{section_testfcts}. This extension to non-smooth functions is important for later applications in Part \ref{part_meanval}.
In the definition of $\Xi^T(\cdot,\Phi)$, the function $\Xi_{\text{main}}(\cdot, \Phi)$ corresponds to the partial sum over such $\ooo\in\OOO$ which are attached to orbits of regular elliptic elements. The function $\Xi(\cdot, \Phi)$ is also closely connected to Arthur's trace formula for $G$, cf.\ \S \ref{connection_arthur_tf}. 
Our first result is the following:
\begin{theorem}\label{holom_intr}[cf.\ Theorem \ref{convergence_distributions}]
Let $n\geq2$. There exists $\nu>0$ depending only on $n$ such that for every $\Phi\in\SSS^{\nu}(\ggG(\A))$ the following holds:
\begin{enumerate}[label=(\roman{*})]
 \item If $T$ is sufficiently regular, the integral defining $\Xi^T(s,\Phi)$ converges absolutely and locally uniformly for $\Re s>\frac{\sqrt{D}+1}{2}$. In particular, $\Xi^T(s,\Phi)$ is holomorphic in this half plane.

\item $\Xi^T(s,\Phi)$ is a polynomial in $T$ of degree at most $\dim\aaa=n-1$ and can be defined for every $T\in\aaa$. Then for every $T$ the function $\Xi^T(s,\Phi)$ is holomorphic in $\Re s>\frac{\sqrt{D}+1}{2}$.
\end{enumerate}
\end{theorem}
In this way we get a well-defined family $\Xi^T(s,\Phi)$ of zeta functions indexed by the parameter $T\in\aaa$ and varying continuously with $T$. By the nature of our construction this family depends on an initial choice of minimal parabolic subgroup in $G$. 
We can, however, choose a zeta function in this family which is independent of this choice:
Taking $T=0$, the function $\Xi^0(s,\Phi)$ does not depend on the fixed minimal parabolic subgroup anymore (cf.\ \cite[Lemma 1.1]{Ar81}) so that $\Xi^0(s,\Phi)$ can be viewed as ''the`` zeta function associated with $G$ acting on $\ggG$.

One of the standard methods to get the meromorphic continuation and functional equation of zeta functions is to use the Poisson summation formula. In our context, Chaudouard's trace formula takes the place of the  Poisson summation formula, and the main obstruction  to obtain the meromorphic continuation and the functional equation for $\Xi^T(s,\Phi)$ is to understand the nilpotent contribution $J_{\nnn}^T(\Phi_{\lambda})$.
Restricting to  $n\leq 3$, we are able to analyse the nilpotent distribution $J_{\nnn}^T(\Phi_{\lambda})$ completely (see \S \ref{nilp_distr1} and \S \ref{section_main_res}), obtaining our main result of Part \ref{part_zeta}:
\begin{theorem}\label{main_thmintr}[cf.\ Theorems \ref{main_theorem}]
Let $G=\GL_n$ or $G=\SL_n$ with $n\leq 3$, and let $R>n$ be given. Then there exists $\nu<\infty$ such that for every $\Phi\in\SSS^{\nu}(\ggG(\A))$ and  $T\in\aaa$ the following holds.
\begin{enumerate}[label=(\roman{*})]

\item 
$\Xi^T(s, \Phi)$ has a meromorphic continuation to all $s\in\C$ with $ \Re s>-R$, and satisfies for such $s$ the functional equation
\[
\Xi^T(s, \Phi)=\Xi^T(1-s, \hat{\Phi}).
\]

\item
The poles of $\Xi^T(s, \Phi)$ in $\Re s>-R$ are parametrised by the nilpotent orbits $\NNN\subseteq \nnn$. More precisely, its poles occur exactly at the points 
\[
s_{\NNN}^-=\frac{1-\sqrt{D}}{2}+\frac{\dim\NNN}{2\sqrt{D}}
~~~~\text{and }~~~~
s_{\NNN}^+=\frac{1+\sqrt{D}}{2}-\frac{\dim\NNN}{2\sqrt{D}}
\]
and are of order at most $\dim\aaa=n-1$. 
In particular, the furthermost right and furthermost left pole in this region are both simple, correspond to $\NNN=0$, and are located at the points
$s_{0}^+=\frac{1+\sqrt{D}}{2}$
and 
$s_{0}^-=\frac{1-\sqrt{D}}{2}$, respectively.
 The residues at these poles are given by
\begin{align*}
\res_{s=s_{0}^-}\Xi^T(s, \Phi)
&=\vol(A_GG(\Q)\backslash G(\A))\Phi(0),\text{   and }\\
\res_{s=s_{0}^+}\Xi^T(s, \Phi)
&=\vol(A_GG(\Q)\backslash G(\A))\int_{\ggG(\A)}\Phi(X) dX.
\end{align*}
\end{enumerate}
\end{theorem}
Note that if $\nu=\infty$, then $\Phi$ is a Schwartz-Bruhat function and  $\Xi^T(s, \Phi)$ can be meromorphically continued to all of $\C$.

Chaudouard's trace formula is valid for any reductive group. In principle, it is possible to define the zeta function $\Xi^T(s, \Phi)$ as in \eqref{def_zetafct_intr} for $G$ an arbitrary reductive group acting on its Lie algebra. At least  Theorem \ref{holom_intr} should stay true (in fact, our proof should go through as it is without major difficulties; we restricted to $\GL_n$ and $\SL_n$ mainly to make it not more technical as it already is). One can of course also conjecture that the analogue of Theorem \ref{main_thmintr} holds, and the main difficulty then lies in the analysis of the nilpotent contribution $J_{\nnn}^T(\Phi_{\lambda})$.
One could take this approach even further, by considering a general rational representation of the group instead of its adjoint representation. In \cite{Le01}  equivalence classes $\ooo$ and corresponding distributions $J_{\ooo}^T(\Phi)$ are defined for such a representation, and also a kind of trace ''formula'' is proved for this situation. For the Shintani zeta function of binary quartic forms such an approach has already been used in \cite{Le99}.

For $G=\GL_2$ and $G=\GL_3$, we can show that $\Xi_{\text{main}}(s,\Phi)$ is indeed the main part of $\Xi^T(s,\Phi)$ in the following sense:
\begin{proposition}\label{holom_ct_intr}[cf.\ Corollaries \ref{holom_ct_2} and \ref{holom_ct_3}]
 If $G=\GL_2$ or $G=\GL_3$, then $\Xi^T(s, \Phi)-\Xi_{\text{main}}(s, \Phi)$ continues holomorphically at least to $\Re s>\frac{n}{2}$. In particular, the furthermost right pole of $\Xi^T(s,\Phi)$ and $\Xi_{\text{main}}(s, \Phi)$ coincide and have the same residue.
\end{proposition}
This result will become important in Part \ref{part_meanval}, where we will use the analytic properties of $\Xi_{\text{main}}(s,\Phi)$ to apply a Tauberian theorem.

The organisation of Part \ref{part_zeta} is as follows: In \S \ref{section_trace_formula} we will define $\Xi^T(s,\Phi)$ and prove Theorem \ref{holom_intr}. In \S \ref{nilp_distr1} and \S \ref{section_main_res} we study the nilpotent distribution $J_{\nnn}^T(s,\Phi)$ for $n\leq3$ and conclude the proof of Theorem \ref{main_thmintr}. In \S \ref{connection_arthur_tf} we describe the connection of our construction of the zeta function to the Arthur-Selberg trace formula for $\GL_n$, and of $\Xi_{\text{main}}(s,\Phi)$ to the classical Shintani zeta function for $n=2$.
Finally, restricting to $G=\GL_n$, $n\leq 3$, we prove Proposition \ref{holom_ct_intr} in \S \ref{section_poles}.

\subsection*{Part \ref{part_meanval}}
The Shintani zeta function $Z(s,\Psi)$ for the space of binary cubic forms 
was used by Shintani to establish mean values for the class numbers of binary quadratic forms, cf.\ \cite{Sh75}. 
From our point of view, another closely related density result obtained from $Z(s,\Psi)$ is more important: Datskovsky (cf.\ \cite{Da93}) proved that if $S$ is a finite set of prime places of $\Q$ including the archimedean place, and $r_S=(r_v)_{v\in S}$ is a fixed signature for quadratic number fields, then as $X\rightarrow\infty$ one has
\begin{equation}\label{asymp_datskovsky}
 \sum_{L:~D_L\leq X} \res_{s=1} \zeta_L(s) =\alpha(r_S) X,
\end{equation}
where $L$ runs over all quadratic fields of signature $r_S$ and absolute discriminant $D_L$ bounded by $X$, and $\alpha(r_S)$ is a suitable non-zero constant. 
As a first step towards generalising this,  we prove upper and lower bounds for the densities of residues of Dedekind zeta functions of totally real cubic number fields.

Suppose $E$ is a totally real number field of degree $n$  with ring of integers $\OOO_E\subseteq E$. We denote by $Q_E: \OOO_E/\Z\longrightarrow\R$ the positive definite quadratic form $Q_E(\xi)=\tr_{E/\Q}\xi^2-\frac{1}{n}(\tr_{E/\Q}\xi)^2$ for $\xi\in \OOO_E/\Z$, where $\tr_{E/\Q}:E\longrightarrow\Q$ denotes the field trace of $E/\Q$. We denote the successive minima of $Q_E$ on $\OOO_E/\Z$ by $m_1(E)\leq m_2(E)\leq\ldots\leq m_{n-1}(E)$.
If $n=2$, then $m_1(L)=D_L/2$ for every quadratic field $L$ so that the sum in \eqref{asymp_datskovsky} runs over all quadratic fields with $m_1(E)\leq X/2$.
Our main result of Part \ref{part_meanval} is the following:
\begin{theorem}\label{asymp_intr}[cf.\ Theorem \ref{asymptotic_real_fields}]
We have 
\begin{equation}\label{upper_bound_residues}
 \limsup_{X\rightarrow\infty}X^{-\frac{5}{2}}\sum_{E:~m_1(E)\leq X}\res_{s=1}\zeta_E(s)<\infty
\end{equation}
where the sum extends over all totally real cubic number fields $E$ for which the first successive minimum $m_1(E)$ is bounded by $X$. Here $\zeta_E$ denotes the Dedekind zeta function attached to $E$.
\end{theorem}
We complement the above upper bound \eqref{upper_bound_residues} with the following result:
\begin{proposition}\label{lower_bound_intr}[cf.\ Proposition \ref{lower_bound_for_residues}]
 For every $\eps>0$, we have
\[
 \liminf_{X\rightarrow\infty} X^{-\frac{5}{2}+\eps}\sum_{E:~ m_1(E)\leq X}\res_{s=1}\zeta_E(s) =\infty,
\]
where the sum extends over totally real cubic number fields $E$.
\end{proposition}
This is a first step towards a generalisation of \eqref{asymp_datskovsky} to the cubic case and the signature of totally real cubic number fields.
As in the quadratic case, one expects that in fact the limit of the left hand side in \eqref{upper_bound_residues} exists and is non-zero:
\begin{conjecture}\label{conj_intr}
 There exists a constant $\alpha_3>0$ such that as $X\rightarrow\infty$
\[
\sum_{E:~m_1(E)\leq X}\res_{s=1}\zeta_E(s)\sim \alpha_3 X^{\frac{5}{2}},
\]
where the sum extends over all totally real cubic number fields $E$ for which the first successive minimum $m_1(E)$ is bounded by $X$.
\end{conjecture}
Let us make a few remarks on the (quite different) strategies to prove Theorem \ref{asymp_intr} and Proposition \ref{lower_bound_intr}:
First we use a suitable sequence of test functions and apply a Tauberian Theorem to $\Xi_{\text{main}}(s,\Phi)$ to obtain an asymptotic for the density of certain orbital integrals in Proposition \ref{asymptotics_orbital_integrals}. These orbital integrals are basically products of $\res_{s=1}\zeta_E(s)$ and a quantity $c(\xi,\Phi_f)$, $\xi\in E$, obtained from the non-archimedean part $\Phi_f$ of the test function. For an appropriate $\Phi_f$ we have $c(\xi,\Phi_f)\geq1$ for every relevant $\xi$ so that Theorem \ref{asymp_intr} is a direct consequence of Proposition \ref{asymptotics_orbital_integrals}. 
To prove Proposition \ref{lower_bound_intr}, on the other hand, we go a completely different way (independent of our results for $\Xi^T(s,\Phi)$): We basically show that there are sufficiently many irreducible cubic polynomials.

In fact, we would like to deduce all of the conjectured asymptotic from Proposition \ref{asymptotics_orbital_integrals}.
 In Appendix \ref{appendixb} we give a sequence of test functions $(\Phi_f^{\mathfrak{m}})_{\mathfrak{m}}$ for which $c(\xi,\Phi_f^{\mathfrak{m}})\rightarrow 1$.
However, a certain uniformity of the convergence with respect to $Q_E(\xi)$ is needed to prove Conjecture \ref{conj_intr}, which we were not able to show so far.

Our methods can at least heuristically be applied to $\GL_n$ for every $n\geq2$. In particular, the first pole of $\Xi^T(s,\Phi)$ for $\GL_n$ is expected to be at $s=\frac{n+1}{2}$. This suggests:
\begin{conjecture}\label{conj_n}
For every $n\geq3 $ there exists $\alpha_n>0$ such that as $X\rightarrow\infty$
\[
 \sum_{E: ~m_1(E)\leq X}
\res_{s=1}\zeta_E(s)
\sim\alpha_nX^{\frac{n(n+1)-2}{4}},
\]
where the sum extends over all totally $n$-dimensional number fields $E$ for which the first successive minimum $m_1(E)$ is bounded by $X$.
\end{conjecture}
Ordering fields with respect to the first successive minimum of $Q_E$ (in contrast to the discriminant) is also related to a conjecture of Ellenberg-Venkatesh, cf.\ \cite[Remark 3.3]{ElVe06}: Basically they conjecture that 
$X^{-\frac{n(n+1)-2}{4}}\sum_{E: ~ m_1(E)\leq X} 1$ 
has a non-zero limit as $X\rightarrow\infty$ where $E$ runs over $n$-dimensional number fields.
As remarked in \cite{ElVe06}, it is possible to show a ``weak form'' of this asymptotic under a strong hypothesis on the existence of sufficiently many squarefree polynomials. 
If one can prove an $n$-dimensional analogue of Proposition  \ref{asymptotics_orbital_integrals} and make the passage from $c(\xi,\Phi_f)$ to $1$ work (e.g., with a sequence of test function as $(\Phi^{\mathfrak{m}})$), this should lead to another approach to (a slightly weaker form of) the conjecture of Ellenberg-Venkatesh.

This second part of the paper is organised as follows: In \S \ref{section_mean_val_orbint} we first recall and prove some properties of orbital integrals, before stating and proving an asymptotic for the mean value of certain orbital integrals in \S \ref{section_asymptotic}, cf.\ Proposition \ref{asymptotics_orbital_integrals}.
Our main result Theorem \ref{asymp_intr} in \S \ref{bounds_for_residues} will then be an easy consequence of Proposition \ref{asymptotics_orbital_integrals} together with results in \S \ref{section_mean_val_orbint}. Finally, we will prove Proposition \ref{lower_bound_intr} at the end of \S \ref{bounds_for_residues}.

\subsection*{Acknowledgments}
The work on Part \ref{part_zeta} was started while the author stayed at the MPI in Bonn during August/September 2011, and continued during the year 2011/2012 when the author was supported by grant \#964-107.6/2007 from the German-Israeli Foundation for Scientific Research and Development and by a Golda Meir Postdoctoral Fellowship.
Part \ref{part_meanval} was essentially contained in the author's doctoral dissertation \cite{diss} which was written under the direction of T. Finis and supported by grant \#964-107.6/2007 from the German-Israeli Foundation for Scientific Research and Development.
The author would like to thank T. Finis and E. Lapid for communicating their idea for the proof of Proposition \ref{main_approx_res} carried out in Appendix \ref{appendix}.

\newpage
\section{Notation and general conventions}\label{section_notation}
\subsection{General notation}
We fix notation following \cite{Ch02,Ar05}:
\begin{itemize}
\item $\A$ denotes the ring of adeles of $\Q$. If $v$ is a place of $\Q$, $\Q_v$ denotes the completion of $\Q$ at $v$, and if $v$ is non-archimedean, $|\cdot|_v$ is the usual $v$-adic norm on $\Q_v$, i.e. if $q_v\in\Z$ is the prime corresponding to $v$, then $|\cdot|_v$ is normalised by $|q_v|_v=q_v^{-1}$. Then $|\cdot|=|\cdot|_{\A}$ denotes the norm on $\A^{\times}$ given by the product of the $|\cdot|_v$'s.
\item  $n\geq2$ is an integer, and $G$ denotes $\GL_n$ or $\SL_n$ as a group defined over $\Q$ with Lie algebra $\ggG=\mathfrak{gl}_n$ or $\ggG=\mathfrak{sl}_n$. We put $D=\dim\ggG$ ($=n^2$ or $=n^2-1$). $\One_n\in G$ denotes the identity element.

\item $P_0=T_0U_0$ is the minimal parabolic subgroup of upper triangular matrices with $T_0$ the torus of diagonal elements and $U_0$ its unipotent radical of unipotent upper triangular matrices.
If $P\supseteq T_0$ is a $\Q$-defined parabolic subgroup with Levi component $M=M_P\supseteq T_0$, then 
 $\FFF(M)$ denotes  the set of ($\Q$-defined) parabolic subgroups containing $M$, and 
$\PPP(M)\subseteq \FFF(M)$ the subset of parabolic subgroups with Levi component $M$.
For $P\in \FFF(T_0)$ with Levi decomposition $P=M_PU_P$, we denote by $\ppp=\mmm_P+\uuu_P$ the corresponding decomposition of the Lie algebra.
For $P_1, P_2\in \FFF(T_0)$ with $P_1\subseteq P_2$, put
$\uuu_{P_1}^{P_2}:=\uuu_1^2:=\uuu_{P_1}\cap\mmm_{P_2}$
and  
$\overline{\uuu}_{P_1}^{P_2}:=\overline{\uuu}_{P_1}\cap\mmm_{P_2}
:=\uuu_{\bar{P_1}}\cap\mmm_{P_1}$
for $\bar{P_1}\in\PPP(M_{P_1})$ the opposite parabolic subgroup.
$A_M\subseteq M(\R)$ denotes the identity component of the split component of the center in $M(\A)$. 

\item
$P\in\FFF(T_0)$ is called standard if $P_0\subseteq P$ and we write $\FFF_{\text{std}}\subseteq\FFF(T_0)$ for the set of standard parabolic subgroups.
 
 \item $\aaa_{P}^*$ is the root space, i.e. the $\R$-vector space spanned by all rational characters $M_P\longrightarrow\GL_1$, and 
  $\aaa_{P}=\aaa_{M_P}=\Hom_{\R}(\aaa_{P}^*, \R)$ is the coroot-space. 
$\Sigma_P$ denotes the set of reduced roots of the pair $(A_{M_P}, U_P)$ 

 We denote by $\Delta_{P_1}^{P_2}=\Delta_{1}^2$ the set of simple roots  and by $\Sigma_{P_1}^{P_2}=\Sigma_1^2$ the set of all positive roots of the action of $A_1=A_{P_1}$ on $U_1\cap M_2$. If $\alpha\in\Delta_1^2$, then $\alpha^{\vee}$ denotes the corresponding coroot.
Similarly, $\widehat{\Delta}_{P_1}^{P_2}=\widehat{\Delta}_1^2$ is the set of simple weights, and if $\varpi\in\widehat{\Delta}_1^2$, then $\varpi^{\vee}$ denotes the corresponding coweight. If $\alpha\in \Delta_1^2$, we denote by $\varpi_{\alpha}\in\widehat{\Delta}_1^2$ the weight such that $\varpi_{\alpha}(\beta^{\vee})=\delta_{\alpha \beta}$ for all $\beta\in\Delta_1^2$ (here $\delta_{\alpha \beta}$ is the Kronecker $\delta$).

\item 
If $a\in A_P$ and $\lambda\in\aaa_P^*$, write $\lambda(a)=e^{\lambda(H_P(a))}$.
For $P_1\subseteq P_2$, let 
\[ A_{P_1}^{P_2}=A_1^2
 =\{a\in A_{P_1}\mid \forall \alpha\in \Delta_{P_2}:~ \alpha(a)=1\}
 \simeq A_{P_1}/A_{P_2},\] 
 and $\aaa_{P_1}^{P_2}=\log A_1^2\subseteq \aaa_{P_1}$.
For $M\subseteq G$ let $M(\A)^1$ be the intersection of the kernels of all rational characters $M(\A)\longrightarrow\C$.
Let $\aaa_0^+=\{H\in\aaa_0\mid \forall \alpha\in\Delta_0:~ \alpha(H)>0\}$ be the positive chamber in $\aaa_0$ with respect to our fixed minimal parabolic subgroup. Similarly, we define $\big(\aaa_0^G\big)^+$.
 Denote by $\rho_1^2=\rho_{P_1}^{P_2}\in \aaa_0^+$ the unique element in $\aaa_0^+$ such that the modulus function satisfies
$\delta_1^2(m):=\delta_{P_1}^{P_2}(m)
:=|\det \Ad m_{|\uuu_{P_1}^{P_2}(\A)}|
=e^{2\rho_{1}^{2}(H_0(m))}$
for all $m\in M_{P_1}(\A)$ and  write $\rho_1=\rho_{P_1}=\rho_{P_1}^G$ and $\delta_0=\delta_{P_0}^G$.

\item
Let $H_P=H_{M_P}:G(\A)=M(\A)U(\A)\cpt\longrightarrow \aaa_P$ be the map characterised by $H_P(muk)=H_P(m)$ and
$H_P(\exp H)=H$ for all $H\in \aaa_P$.

\item 
We denote by $\Phi(A_0, M_R)$ the set of weights of $A_0$ with respect to $M_R$ so that $\Phi(A_0, M_R)=\Sigma_0^R\cup\{0\}\cup\left(-\Sigma_0^R\right)$.
 Then we have a direct sum decomposition
$\ggG=\bigoplus_{\beta\in\Phi(A_0, M_R)}\ggG_{\beta}$ for $\ggG_{\beta}$ the eigenspace of $\beta$ in $\ggG$.
We take the usual vector norm $\|\cdot\|_{\A}=\|\cdot\|$ on $\ggG(\A)$ obtained by identifying $\ggG(\A)$ with $\A^D$ via the matrix coordinates.
  Then  if $X\in\ggG(\A)$, $X=\sum_{\beta\in\Phi(A_0, M_R)} X_{\beta}$ with $X_{\mu}\in \ggG_{\beta}(\A)$, then $\|X\|=\sum_{\beta\in\Phi(A_0, M_R)}\|X_{\beta}\|$.

 \item  If $M=T_0$, we write $\FFF=\FFF(T_0)$, $H_0=H_{M_0}$, $\aaa_0=\aaa_{M_0}$, etc., and further put $\aaa=\aaa_0^G$ and $\aaa^+=\big(\aaa_0^G\big)^+$.
\end{itemize}

\subsection{Characteristic functions}
Let $P_1, P_2, P\in\FFF$ be parabolic subgroups with $P_1\subseteq P_2$. We define the following functions (cf.\ \cite{Ar78}):
\begin{itemize}
\item
$\hat{\tau}_{P_1}^{P_2}=\hat{\tau}_1^2:\aaa_0\longrightarrow\C$ is the characteristic function of the set 
\[
\{H\in\aaa_0\mid \forall \varpi\in \hat{\Delta}_{1}^{2}:~\varpi(H)>0\}.
\]
If $P_2=G$, we also write $\hat{\tau}_{P_1}=\hat{\tau}_{1}=\hat{\tau}_{1}^G$.

\item
$\tau_{P_1}^{P_2}=\tau_1^2:\aaa_0\longrightarrow\C$ is the characteristic function of the set
\[
\{H\in\aaa_0\mid \forall \alpha\in \Delta_1^2:~ \alpha(H)>0\}.
\]
If $P_2=G$, we also write $\tau_{P_1}=\tau_{1}=\tau_{1}^G$.

\item 
$\sigma_{P_1}^{P_2}=\sigma_1^2:\aaa_0\longrightarrow\C$ is the characteristic function of the set
\[
\{H\in\aaa_0\mid 
\forall \alpha\in\Delta_1^2:~\alpha(H)>0; ~
\forall \alpha\in \Delta_1\minus\Delta_1^2:~ \alpha(H)\leq0;~
\forall \varpi\in \hat{\Delta}_{2}:~\varpi(H)>0\}.
\]
\end{itemize}
\begin{rem}
The function $\sigma_1^2$ is related to  $\tau_1^2$ and $\hat{\tau}_1^2$ by
$\sigma_1^2
=\sum_{R:~P_2\subseteq R} (-1)^{\dim \aaa_{2}^R}\tau_{1}^R\hat{\tau}_R$.
\end{rem}
\begin{itemize}
\item $T\in\aaa^+$ is called \emph{sufficiently regular} if $d(T):=\min_{\alpha\in\Delta_0}\alpha(T)$ is sufficiently large, i.e., if $T$ is sufficiently far away from the walls of the positive Weyl chamber (cf.\ \cite{Ar78}). We fix a small number $\delta>0$ such that the set of sufficiently regular $T\in\aaa$ satisfying  $d(T)>\delta\|T\|$ is a non-empty open cone in $\aaa^+$.

\item
For sufficiently regular $T\in\aaa^+$  the function  $F^P(\cdot, T):G(\A)\longrightarrow\C$ is defined as the characteristic function of all $x=umk\in G(\A)=U(\A)M(\A)\cpt$, $P=MU$, satisfying
\[
 \varpi(H_0(\mu m)-T)\leq0
\]
for all $\mu\in M(\Q)$ and $\varpi\in\widehat{\Delta}_0^M$. 
If $P=G$, we sometimes write $F(\cdot, T)=F^G(\cdot, T)$.

\item
If $T\in\aaa^+$ is sufficiently regular, \cite[Lemma 6.4]{Ar78} gives for every $x\in G(\A)$ the identity
\[
\sum_{R:~P_0\subseteq R\subseteq P}\sum_{\delta\in R(\Q)\backslash P(\Q)}F^R(\delta x, T)\tau_R^P(H_0(\delta x)-T)=1.
\]
\end{itemize}

\subsection{Measures}
We fix the following maximal compact subgroups: If $v$ is a non-archimedean place, then $\cpt_v=G(\Z_v)$, and at the archimedean place, $\cpt_{\infty}=\Orth(n)$. Globally, we take $\cpt=\prod_{v\leq \infty} \cpt_v$.
Up to normalisation there exists a unique Haar measure on $\cpt_v$, and we normalise it by $\vol(\cpt_v)=1$ for all $v\leq \infty$, and then take the product measure on $\cpt$. We further choose measures as follows:
\begin{itemize}
\item  $\Q_v$ and $\Q_v^{\times}$, $v< \infty$: normalized by $\vol(\Z_v)=1=\vol(\Z_v^{\times})$.

\item $\R$,  $\R^{\times}$, $\R_{>0}$, $A_G$, $A_0$: usual Lebesgue measures. 

\item $\C$, $\C^{\times}$:  twice the usual Lebesgue measure.

\item  $\A$ and $\A^{\times}$: product measures.

\item $\A^1=\{a\in\A^{\times}\mid |a|_{\A}=1\}$: measure induced by the exact sequence $1\longrightarrow \A^1\hookrightarrow \A^{\times}\xrightarrow{|\cdot|_{\A}} \R_{>0}\longrightarrow 1$.

\item  $V$ finite dimensional $\Q$-vector space with fixed basis: take the measures induced from $\A$ (resp.\ $\Q_v$) on $V(\A)$ (resp.\ $V(\Q_v)$) via this basis. 
This in particular defines measures on $U_0(\A)$ and $U_0(\Q_v)$ if we take the canonical bases corresponding to the root coordinates. 

\item  $T_0(\A)$ and $T_0(\Q_v)$: measures induced from $\A^{\times}$ and $\Q_v^{\times}$ via the diagonal coordinates.

\item $G(\A)$ and $G(\Q_v)$:  compatible with the Iwasawa decomposition $G(\A)=T_0(\A)U_0(\A)\cpt$ (resp. $G(\Q_v)=T_0(\Q_v)U_0(\Q_v)\cpt_v$) such that for every integrable function $f$ on $G(\A)$ we have 
\[
 \int_{G(\A)} f(g) dg=\int_{T_0(\A)}\int_{U_0(\A)}\int_{\cpt} f(tuk) dk~du~dt
=\int_{T_0(\A)}\int_{U_0(\A)}\int_{\cpt} \delta_0(t)^{-1} f(utk) dk~du~dt
\]
(similarly for the local case).

\item $G(\A)^1$: measure induced by the exact sequence $1\longrightarrow G(\A)^1\hookrightarrow G(\A)\xrightarrow{|\det(\cdot)|_{\A}} \R_{>0}\longrightarrow 1$.

\item Levi and parabolic subgroups: compatible with previous cases.
\end{itemize}

\subsection{Equivalence classes}\label{section_eq_classes}
Let $\ggG(\Q)_{\text{ss}}$ (resp.\ $G(\Q)_{\text{ss}}$) denote the set of semisimple elements in $\ggG(\Q)$ (resp.\ $G(\Q)$).
We define an equivalence relation on $\ggG(\Q)$ as follows: Let $X, Y\in \ggG(\Q)$ and write $X=X_s+X_n$, $Y=Y_s+Y_n$ for the Jordan decomposition with $X_s$, $Y_s\in\ggG(\Q)_{\text{ss}}$ semisimple and $X_n\in\ggG_{X_s}(\Q)$, $Y_n\in\ggG_{Y_s}(\Q)$ nilpotent, where $\ggG_{X_s}=\{Y\in\ggG\mid [X_s, Y]=0\}$ is the centraliser of $X_s$ in $\ggG$. We call $X$ and $Y$ equivalent if and only if there exists $\delta\in G(\Q)$ such that $Y_s=\Ad\delta^{-1} X_s$. We denote the set of equivalence classes in $\ggG(\Q)$ by $\OOO$.

Let $\nnn\subseteq\ggG(\Q)$ denote the set of nilpotent elements. 
Then $\nnn\in \OOO$ constitutes exactly one equivalence class (corresponding to the orbit of $X_s=0$) and decomposes into finitely many nilpotent orbits under the adjoint action of $G(\Q)$. 
On the other hand, if $\ooo\in\OOO$ corresponds to the orbit of a regular semisimple element $X_s$ (i.e., the eigenvalues of $X_s$ (in an algebraic closure of $\Q$) are pairwise different), then $\ooo$ is in fact equal to the orbit of $X_s$.

\subsection{Test functions}\label{section_testfcts}
Let $\bbb$ denote the Lie algebra of one of the standard parabolic subgroups of $G$, of one of their unipotent radicals or of one of their Levi components. We fix the standard vector norm $\|\cdot\|$ on $\bbb(\R)$ by identifying $\bbb(\R)\simeq \R^{\dim \bbb}$ via the usual matrix coordinates. Let $\UUU(\bbb)$ denote the universal enveloping algebra of the complexification $\bbb(\C)$. For every $\nu\in\Z_{\geq0}$ we fix a basis $\BBB_{\nu}=\BBB_{\bbb,\nu}$ of the finite dimensional $\C$-vector space $\UUU(\bbb)_{\leq \nu}$ of elements in $\UUU(\bbb)$ of degree $\leq \nu$.
For a real number $a\geq0$ and a non-negative integer $b\geq0$ we define seminorms $\|\cdot\|_{a,b}$ on the spaces $C^{\nu}(\bbb(\R))$, $\nu\geq b$, by setting for $f\in C^{\nu}(\bbb(\R))$
\[
 \|f\|_{a,b}:=\sup_{x\in \bbb(\R)}\bigg((1+\|x\|)^a \sum_{X \in\BBB_{b}} \big|(X f)(x)\big|\bigg)
\]
with $(X f)(x)=\left[\frac{d}{dt} f(x e^{tX})\right]_{t=0}$.
We put \[\SSS^{\nu}(\bbb(\R)):=\{f\in C^{\nu}(\bbb(\R))\mid \forall a<\infty,~ b<\nu:~\|f\|_{a,b}<\infty\}.\] Then $\SSS(\bbb(\R)):=\SSS^{\infty}(\bbb(\R))$ is the usual space of Schwartz functions on $\bbb(\R)$.   
If $X\in\BBB_{\nu}$, then $X$ operates on $C^{\nu}(\bbb(\A))$ by acting on the archimedean part of the function. 
We then define seminorms $\|\cdot\|_{a,b}$ on $C^{\nu}(\bbb(\A))$ and on the spaces $\SSS^{\nu}(\bbb(\A))$ and $\SSS(\bbb(\A))$ similar as before.
In particular, $\SSS(\bbb(\A))$ is the usual space of Schwartz-Bruhat functions on $\bbb(\A)$. 
Dualy to $\SSS^{\nu}(\bbb(\A))$ we put \[\SSS_{\nu}(\bbb(\A)):=\{f\in C^{\infty}(\bbb(\A))\mid\forall a<\nu,~b<\infty:~\|f\|_{a,b}<\infty\}\] so that $\SSS_{\infty}(\bbb(\A))=\SSS^{\infty}(\bbb(\A))=\SSS(\bbb(\A))$.

If $p$ is a finite prime, we let $\SSS(\ggG(\Q_p))$ denote the set of compactly supported smooth functions $\ggG(\Q_p)\longrightarrow\C$, and define $\SSS(\ggG(\A_f))$ analogously.

The topology induced by the set of seminorms $\|\cdot\|_{a,b}$, $a<\infty$, $b<\nu$ (resp.\ $a<\nu$, $b<\infty$) makes $\SSS^{\nu}(\bbb(\A))$ (resp.\ $\SSS_{\nu}(\bbb(\A))$) into a Frechet space. 
The words ``seminorm'' and ``continuous seminorm'' on one of these spaces will be used synonymously.

We  fix a non-degenerate invariant bilinear form $\langle\cdot, \cdot\rangle:\ggG(\A)\times \ggG(\A)\longrightarrow\A$ by setting $\langle X, Y\rangle =\tr (XY)$ for $X,Y\in\ggG(\A)$.
Let $\psi:\Q\backslash\A\longrightarrow \C$ be the non-trivial character constructed in \cite[XIV, \S 1]{La86}.
We define the Fourier transform 
\[
 \widehat{\;\;}:\SSS^{\nu}(\ggG(\A))\longrightarrow \SSS_{\nu}(\ggG(\A)),
~~~
\widehat{\Phi}(Y):=\int_{\ggG(\A)}\Phi(X) \psi(\langle X,Y\rangle)dX
\]
with respect to this bilinear form.

\subsection{Siegel sets}
If $T\in \aaa$, let $A_0^G(T)$ denote the set of all $a\in A_0^G$ with $\alpha(H_0(a)-T)>0$ for all $\alpha\in\Delta_0$. Reduction theory tells us that there exists $T_1\in -\aaa^+$ such that
\[
 G(\A)^1=G(\Q) P_0(\A)^1  A_0^G(T_1) \cpt.
\]
We fix such a $T_1$ from now on and write 
\[
 \SSS_{T_1}=\{g=pk\in P_0(\A) \cpt\mid \forall \alpha\in \Delta_0:~ \alpha(H_0(a)-T_1)>0\}.
\]
If then $f:G(\Q)\backslash G(\A)^1\longrightarrow\R_{\geq0}$ is measurable, we have
\begin{multline}\label{siegel_est}
 \int_{G(\Q)\backslash G(\A)^1} f(g) dg
\leq \int_{A_G P_0(\Q)\backslash \SSS_{T_1}} f(g) dg\\
= \int_{\cpt}\int_{U_0(\Q)\backslash U_0(\A)}\int_{T_0(\Q)\backslash T_0(\A)^1}\int_{A_0^G} \delta_0(a)^{-1}\tau_0^G(H_0(a)-T_1) f(uatk) da~dt~du~dk.
\end{multline}

\subsection{Distributions associated with equivalence classes}\label{section_distr}
For $\ooo\in\OOO$ and sufficiently regular $T\in\aaa^+$ define for $x\in G(\A)$ (cf.\ \cite{Ch02}) 
\begin{align*}
K_{P, \ooo}(x, \Phi) & =\int_{\uuu_P(\A)}\sum_{X\in \mmm_P(\Q)\cap\ooo} \Phi(\Ad g^{-1}(X+U)) dU \text{    and }\\
k_{\ooo}^T(x, \Phi) & =\sum_{P\in\FFF_{\text{std}}}(-1)^{\dim A_P/A_G}\sum_{\delta\in P(\Q)\backslash G(\Q)}\hat{\tau}_P(H_0(\delta x)-T)K_{P, \ooo}(\delta x, \Phi),
\end{align*}
 and, if $\Phi:\ggG(\A)\longrightarrow\C$ is integrable, we set
\[
J_{\ooo}^T(\Phi)=\int_{A_GG(\Q)\backslash G(\A)} k_{\ooo}^T(x, \Phi)dx
\]
provided the sum-integrals converge.

\part{The zeta function}\label{part_zeta}
\section{The trace formula for Lie algebras and convergence of distributions}\label{section_trace_formula}
Let us recall some of the main results from \cite{Ch02}.
\begin{theorem}[\cite{Ch02}, Th\'{e}ore\`{m}e 3.1, Th\'{e}ore\`{m}e 4.5]\label{thm_ch}
For all $\Phi\in \SSS(\ggG(\A))$ and sufficiently regular $T\in\aaa^+$ we have
\begin{equation}\label{abs_conv_main}
\int_{A_GG(\Q)\backslash G(\A)} \sum_{\ooo\in\OOO}|k_{\ooo}^T(x, \Phi)|dx<\infty.
\end{equation}
and
\begin{equation}\label{summ_formula}
\sum_{\ooo\in\OOO}J_{\ooo}^T(\Phi)=\sum_{\ooo\in\OOO}J_{\ooo}^T(\hat{\Phi}).
\end{equation}
The distributions $J_{\ooo}^T(\Phi)$ and $\sum_{\ooo\in\OOO}J_{\ooo}^T(\Phi)$  are polynomials in $T$ of degree at most $\dim \aaa$.
\end{theorem}

The Poisson summation like identity \eqref{summ_formula}, is what we refer to as Chaudouard's trace formula for the Lie algebra $\ggG$.

\begin{rem}
\begin{enumerate}[label=(\roman{*})]
\item Since the distributions in the theorem are polynomials in $T$ for $T$ varying in a non-empty open cone of $\aaa$, they can be defined at any point $T\in\aaa$, with \eqref{summ_formula} then being valid for all $T\in \aaa$.

\item The results in \cite{Ch02} hold for arbitrary reductive groups $G$.

\item
\eqref{abs_conv_main} holds for every $\Phi\in\SSS^{\nu}(\ggG(\A))\cup\SSS_{\nu}(\ggG(\A))$, and \eqref{summ_formula} holds for every $\Phi\in\SSS^{\nu}(\ggG(\A))$ if $\nu>0$ is sufficiently large in a sense depending only on $n$, cf.\ also the proof of Lemma \ref{fundamental_convergence} below.
\end{enumerate}
\end{rem}

For $\Phi:\ggG(\A)\longrightarrow\C$, $\lambda\in (0, \infty)$, and $x\in \ggG(\A)$ put 
$\Phi_{\lambda}(x):=\Phi(\lambda x)$. For fixed $\lambda$, $\Phi_{\lambda}\in\SSS^{\nu}(\ggG(\A))$ if $\Phi\in\SSS^{\nu}(\ggG(\A))$, and $\Phi_{\lambda}\in\SSS_{\nu}(\ggG(\A))$ if $\Phi\in\SSS_{\nu}(\ggG(\A))$. Hence \eqref{summ_formula} becomes 
\[
 \sum_{\ooo\in\OOO}J_{\ooo}^T(\Phi_{\lambda})=\lambda^{-D}\sum_{\ooo\in\OOO}J_{\ooo}^T(\hat{\Phi}_{\lambda^{-1}}).
\]
Let $\OOO_*:=\OOO\minus\{\nnn\}$, and for sufficiently regular $T\in\aaa^+$ set
$J_*^T=\sum_{\ooo\in\OOO_*}J_{\ooo}^T$. Then
\[
\Xi^T(s, \Phi)
=\int_{0}^{\infty}\lambda^{\sqrt{D}(s+\frac{\sqrt{D}-1}{2})}J_{*}^T(\Phi_{\lambda})d^{\times}\lambda
\]
defines our regularised zeta function provided this last integral converges.

\begin{theorem}\label{convergence_distributions}
There exists $\nu>0$ depending only on $n$ such that for all $\Phi\in\SSS^{\nu}(\ggG(\A))$ the following holds:
\begin{enumerate}[label=(\roman{*})]
\item
If $T$ is sufficiently regular, the function
\[
\Xi^{T, +}(s, \Phi)
=\int_{1}^{\infty}\lambda^{\sqrt{D}(s+\frac{\sqrt{D}-1}{2})}J_{*}^T(\Phi_{\lambda})d^{\times}\lambda
\]
is absolutely and locally uniformly convergent for all $s\in \C$ and hence entire.

\item 
If $T$ is sufficiently regular, the integral defining $\Xi^T(s, \Phi)$ and also
\[
\Xi^T_{\ooo}(s, \Phi)
:=\int_{0}^{\infty}\lambda^{\sqrt{D}(s+\frac{\sqrt{D}-1}{2})}J_{\ooo}^T(\Phi_{\lambda})d^{\times}\lambda,
~~~
\ooo\in\OOO_*,
\]
are well-defined and absolutely and locally uniformly convergent for $s\in\C$ with $\Re s>\frac{\sqrt{D}+1}{2}$ (and hence holomorphic there). Moreover,
\[
 \Xi^T(s, \Phi)
=\sum_{\ooo\in\OOO_*}\Xi^T_{\ooo}(s, \Phi).
\]

\item 
The distributions $\Xi^{T, +}(s, \Phi)$, $\Xi^T_{\ooo}(s, \Phi)$, and $\Xi^T(s, \Phi)$ are polynomials in $T$ of degree at most $\dim\aaa=n-1$. The coefficients of these polynomials are holomorphic functions in $s$ for $s$ ranging in the regions indicated above. 
\end{enumerate}
\end{theorem}

\begin{rem}
 The distributions in the theorem can again be defined at every point $T\in \aaa$ by taking the value of the polynomial at this point. Their analytic properties as stated in the theorem stay valid for every $T$.
\end{rem}

The theorem is an immediate consequence of the following lemma.

\begin{lemma}\label{fundamental_convergence}
Let $T\in\aaa^+$ be sufficiently regular.
\begin{enumerate}[label=(\roman{*})]
 \item 
There exists an integer $\nu>0$ (depending on $n$) such that the following holds. 
\begin{enumerate}
\item
 For every $N\in\N$ there exists  a seminorm $\mu_N$ on the space $\SSS^{\nu}(\ggG(\A))$ such that
\begin{equation}\label{finiteness1}
\int_{A_GG(\Q)\backslash G(\A)}\sum_{\ooo\in\OOO_*}
|k^T_{\ooo}(x, \Phi_{\lambda}) |dx\leq  \mu_N(\Phi)\lambda^{-N}   
\end{equation}
for all $\lambda\in[1, \infty)$ and $\Phi\in\SSS^{\nu}(\ggG(\A))$.

\item 
There exists a seminorm  $\mu$ on $\SSS^{\nu}(\ggG(\A))\cup\SSS_{\nu}(\ggG(\A))$ such that
\begin{equation}\label{finiteness2}
\int_{A_GG(\Q)\backslash G(\A)}\sum_{\ooo\in\OOO_*}
|k^T_{\ooo}(x, \Phi_{\lambda}) |dx\leq  \mu(\Phi)\lambda^{-D}
\end{equation}
for all $\lambda\in (0, 1]$ and $\Phi\in\SSS^{\nu}(\ggG(\A))\cup\SSS_{\nu}(\ggG(\A))$.
\end{enumerate}

\item If $N\in\N$, then there exists an integer $\nu>0$ and a seminorm $\mu_N$ on the space $\SSS_{\nu}(\ggG(\A))$, both depending only on $n$ and $N$, such that 
\begin{equation}\label{finiteness1_lowernu}
\int_{A_GG(\Q)\backslash G(\A)}\sum_{\ooo\in\OOO_*}
|k^T_{\ooo}(x, \Phi_{\lambda}) |dx\leq  \mu_N(\Phi)\lambda^{-N}   
\end{equation}
for all $\lambda\in[1, \infty)$ and  $\Phi\in\SSS_{\nu}(\ggG(\A))$.
\end{enumerate}
\end{lemma}
We will prove the lemma in \S \ref{subsection_prf_of_key_lemma} below, but first deduce the proposition from it.
\begin{proof}[Proof of Theorem \ref{convergence_distributions}]
 \begin{enumerate}[label=(\roman{*})]
  \item 
By Lemma  \ref{fundamental_convergence}
we have for $N$ arbitrarily large and every $\lambda\geq1$,
\[
|\lambda^{\sqrt{D}(s+\frac{\sqrt{D}-1}{2})}J_{*}^T(\Phi_{\lambda})|
\leq \mu_N(\Phi)\lambda^{\sqrt{D}(\Re s+\frac{\sqrt{D}-1}{2})}\lambda^{-N},
\]
which is of course integrable over $\lambda\in[1, \infty)$ if $N$ is chosen sufficiently large.

\item 
We split the integral defining $\Xi^T(s, \Phi)$ into one integral over $\lambda\in (0,1]$ and one over $\lambda\in[1,\infty)$.
By the first part of the proposition the second integral defines a holomorphic function on all of $\C$.
For the first integral we have $|J_{*}^T(\Phi_{\lambda})|\leq \mu(\Phi)\lambda^{-D}$ for all $\lambda\leq 1$ by Lemma \ref{fundamental_convergence} so that 
\[
\int_{0}^{1}|\lambda^{\sqrt{D}(s+\frac{\sqrt{D}-1}{2})}J_{*}^T(\Phi_{\lambda})|d^{\times}\lambda
\leq \mu(\Phi)\int_{0}^{1}\lambda^{\sqrt{D}(\Re s+\frac{\sqrt{D}-1}{2})}\lambda^{-D}d^{\times}\lambda,
\]
which is finite if $\Re s>\frac{\sqrt{D}+1}{2}$, and hence proving the second part of the proposition.

\item 
By Theorem \ref{thm_ch} $J_{\ooo}^T(\Phi)$ and $J_{*}^T(\Phi)$ are polynomials of degree at most $\dim\aaa$ in $T$. The assertion thus follows from the previous parts of the proposition.\qedhere
\end{enumerate}
\end{proof}

\subsection{Auxiliary results}
To prove Lemma \ref{fundamental_convergence}, we need some preparation.
Let $P_1, P_2, R\in\FFF_{\text{std}}$ be standard parabolic subgroups with $P_1\subseteq R\subseteq P_2$, and
write $P_i=M_iU_i$ for their Levi decomposition.
We define
\[
\widetilde{\mmm}_1^2
=\widetilde{\mmm}_{P_1}^{P_2}
=\mmm_{2}\minus\bigg(\bigcup_{\substack{Q\in\FFF: \\ P_1\subseteq Q\subsetneq P_2}} \mmm_{2}\cap\mathfrak{q}\bigg).
\]
Note that $0\not\in \widetilde{\mmm}_{1}^{2}(\Q)$ unless $P_1= P_2$. Moreover,  $\widetilde{\mmm}_{1}^{2}=\mmm_{1}$ if and only if $P_1=P_2$. 
Similarly, put
\[
\overline{\uuu}_{1}^{2\prime}
=\overline{\uuu}_{P_1}^{P_2\prime}
=\overline{\uuu}_{1}^{2}\minus\bigg(\bigcup_{\substack{Q\in\FFF: \\  P_1\subseteq Q\subsetneq P_2}} \overline{\uuu}_{1}^Q\bigg)
=\overline{\uuu}_{1}^{2}\minus\bigg(\bigcup_{\substack{Q\in\FFF: \\  P_1\subseteq Q\subsetneq P_2}} \overline{\uuu}_{P_1}\cap\mmm_Q\bigg),
\]
and define $\uuu_{1}^{2\prime}$ with $\uuu_{1}^Q$ in place of $\overline{\uuu}_{1}^Q$ analogously.
Note that $0\not\in \overline{\uuu}_{1}^{2\prime}(\Q)$ unless $P_1=P_2$.

\begin{definition}
\begin{enumerate}[label=(\roman{*})]
\item
If  $S\subseteq \Sigma_1^2$ is a subset, we say that $S$ has \emph{property $\Pi(P_1,R, P_2)$} if for every $\alpha\in \Delta_{1}^{2}\minus\Delta_{1}^{R}$ there exists $\beta\in S$ such that $\varpi_{\alpha}(\beta^{\vee})>0$. 
In particular, $S=\emptyset$ has property $\Pi(P_1, R, P_2)$.

\item If $S\subseteq \Sigma_1^2$ has property $\Pi(P_1, R, P_2)$, we define $\overline{\uuu}_S^{\prime}\subseteq\overline{\uuu}_R^2$ as the set consisting of all $Y=\sum_{\beta\in\Sigma_{1}^{2}}Y_{-\beta}\in\overline{\uuu}_{R}^{2}$ with $Y_{-\beta}\neq0$ for $\beta\in S$ and $Y_{-\beta}=0$ for $\beta\not\in S$.
Here $Y_{-\beta}$ denotes the component of $Y$ in the $(-\beta)$-eigenspace of the decomposition of $\overline{\uuu}_R^2$ with respect to $-\Sigma_1^2$.
In particular, $\overline{\uuu}_{\emptyset}^{\prime}=\emptyset$ unless $R=P_1$ in which case $\overline{\uuu}_{\emptyset}^{\prime}=\overline{\uuu}_1^{1\prime}=\{0\}$.

\item If $S\subseteq \Sigma_1^R$ has property $\Pi(P_1, P_1, R)$, let $\mmm_{R,S}\subseteq \mmm_{R}$ consist of all $Y\in\mmm_{R}$ such that $Y_{-\beta}\neq0$ for all $\beta\in S$ and $Y_{-\beta}=0$ for all $\beta\in\Sigma_1^R\minus S$.
Here $Y_{-\beta}$ denotes the component of $Y$ in the $(-\beta)$-eigenspace of the decomposition of $\mmm_R$ with respect to $\Phi(A_1,M_R)$.
\end{enumerate}
\end{definition}

\begin{lemma}\label{identifying_nilpotents}
Write $\mmm_2=\bigoplus_{\beta\in\Phi(A_1, M_2)}\mmm_{\beta}$ with $\mmm_{\beta}$ the eigenspace for $\beta$ in $\mmm_2$, and if $X\in\mmm_2(\Q)$, let $X_{\beta}\in\mmm_{\beta}(\Q)$ be its $\beta$-component so that $X=\sum_{\beta\in \Phi(A_1, M_2)}X_{\beta} $.
Then:
\begin{enumerate}[label=(\roman{*})]
\item\label{identifying_nilp1} 
For every $Y\in \overline{\uuu}_{R}^{2\prime}(\Q)$, there exists a subset $S\subseteq \Sigma_{1}^{2}$ with property $\Pi(P_1,R, P_2)$ such that $Y_{-\beta}=0$ for all $\beta\in\Sigma_1^2\minus S$ and $Y_{-\beta}\neq 0$ for all $\beta\in S$.
In particular,
\[
 \overline{\uuu}_R^{2\prime}=\bigoplus_{S\subseteq \Sigma_{1}^{2}}\overline{\uuu}_S^{\prime}
\]
where the sum runs over all subsets $S\subseteq \Sigma_{1}^{2}$ having property $\Pi(P_1, R, P_2)$. 

 \item\label{identifying_nilp2} 
If $P_1\subsetneq R$ and $X\in\widetilde{\mmm}_{P_1}^R(\Q)$,  there exists a non-empty subset $S\subseteq \Sigma_{1}^R$ with property $\Pi(P_1,P_1, R)$ such that $X_{-\beta}\neq 0$ for every $\beta\in S$.
In particular,
\[
 \widetilde{\mmm}_{P_1}^{R}\subseteq\bigoplus_{S\subseteq \Sigma_1^R} \mmm_{R,S}.
\]
where the sum runs over all non-empty subsets $S\subseteq \Sigma_{1}^R$ having property $\Pi(P_1, P_1,R)$. 
\end{enumerate}
\end{lemma}

\begin{proof}
\begin{enumerate}[label=(\roman{*})]
\item 
Let $Y\in\overline{\uuu}_{R}^{2}$. Let the set $S\subseteq\Sigma_1^2$ be defined to consist exactly of those $\beta\in\Sigma_{1}^2$ with $Y_{-\beta}\neq0$. 
$S$ has property $\Pi(P_1, R, P_2)$: For that suppose that instead there exists $\alpha\in \Delta_1^2\minus\Delta_1^R$ such that for all $\beta\in S$ we have $\varpi_{\alpha}(\beta^{\vee})\leq0$. Now every $\beta$  is a non-negative linear combination of elements in  $\Delta_1^2$ so that $\varpi_{\alpha}(\beta^{\vee})\leq0$ implies $\varpi_{\alpha}(\beta^{\vee})=0$. 
But this implies that $\beta\in \Sigma_1^Q$ for some parabolic subgroup $Q\subsetneq  P_2$, $R\subseteq Q$. Hence $Y\in \overline{\uuu}_R^{Q}(\Q)$ in contradiction to $Y\in\overline{\uuu}_R^{2\prime}(\Q)$ so that our set $S$ must have property $\Pi(P_1, R, P_2)$.

\item
 This follows from the definitions. \qedhere
\end{enumerate}
\end{proof}

\begin{lemma}\label{lattice_sum_lemma}
Suppose $R\subsetneq P_2$. If $m>\dim {\uuu}_{R}^{2}$, then there exist constants $c>0$ and $k_{\alpha}\geq 0$ for every $\alpha\in\Delta_1^2$ such that 
\begin{itemize}
\item 
$k_{\alpha}>0$ for all $\alpha\in \Delta_1^2\minus\Delta_1^R$, and 

\item for all $a\in A_1^G=A_{P_1}^G$, we have
\[
 \sum\limits_{{Y}\in{\overline{\uuu}}_{R}^{2\prime}(\frac{1}{N}\Z)}||\Ad a^{-1} {Y}||^{-m}
\leq c \prod\limits_{\alpha\in \Delta_1^2}e^{-k_{\alpha}\alpha(H_0(a))}.
\] 
\end{itemize}
\end{lemma}

\begin{proof}
 This is a slightly refined version of \cite[pp. 946-947]{Ar78} in that we give a  sufficient lower bound for the exponent $m$. 
Suppose first that $m>0$ is sufficiently large. We shall later see that $m>\dim\uuu_R^2$ suffices.
 
Consider non-empty subsets  $S\subseteq\Sigma_R^2$ with property $\Pi(P_1,R, P_2)$.
By Lemma \ref{identifying_nilpotents}\ref{identifying_nilp1} the set $\overline{\uuu}_R^{2\prime}(\frac{1}{N}\Z)$ is the direct sum over such sets $S$ of $\overline{\uuu}_S^{\prime}(\frac{1}{N}\Z)$.
For $\beta\in\Sigma_R^2$ let $\{E_{-\beta, i}\}_{i=1,\ldots, d_{-\beta}}$, $d_{-\beta}:=\dim \uuu_{-\beta}$, be a basis for the eigenspace $\uuu_{-\beta}$ of $-\beta$ in $\overline{\uuu}_R^2$, which is orthogonal with respect to the norm $\|\cdot\|$, i.e. 
$\|\sum_i b_iE_{-\beta, i}\|
=\sum_i|b_i|$ 
for all $b_1, \ldots, b_{d_{-\beta}}\in\R$. 
Thus, if $Y\in\overline{\uuu}_S(\frac{1}{N}\Z)$, we can uniquely write 
$Y=\sum_{\beta\in S}\sum_{i=1}^{d_{-\beta}}Y_{-\beta, i}E_{-\beta, i}$,
and get for every $a\in A_{1}^G$ that
\[
\|\Ad a^{-1} Y\|
=\sum_{\beta\in S}e^{2\beta(H_0(a))}\sum_{i=1}^{d_{-\beta}}\|Y_{-\beta,i}\|.
\]
Let $\RRR=(R_{\beta})_{\beta\in S}$ be a tuple of non-empty subsets $R_{\beta}\subseteq \{1, \ldots, d_{-\beta}\}$, and define
\[
 \overline{\uuu}_{S, \RRR}^{\prime}
=\big\{Y\in\overline{\uuu}_S^{\prime}\mid Y_{-\beta, i}\neq 0\Leftrightarrow \beta\in S \text{ and } i\in R_{\beta}\}. 
\]
Clearly, 
$\overline{\uuu}_S^{\prime}
=\bigoplus_{\RRR=(R_{\beta})_{\beta\in S}}\overline{\uuu}_{S, \RRR}^{\prime}$ with the sum running over all tuples $\RRR$ as before.
As there are only finitely many such tuples $\RRR$, it suffices to consider the sum over $Y\in \overline{\uuu}_{S, \RRR}^{\prime}(\frac{1}{N})$ for one of the tuples $\RRR$.

Then, since $0\not\in\overline{\uuu}_S^{\prime}$ because of $R\subsetneq P_2$,
\begin{align*}\label{bounding_sum_over_nilpotents1}
\sum_{Y\in\overline{\uuu}_{S, \RRR}^{\prime}(\frac{1}{N}\Z)}\|\Ad a^{-1} Y\|^{-m}
& =\sum_{Y\in\overline{\uuu}_{S, \RRR}^{\prime}(\frac{1}{N}\Z)}\bigg(\sum_{\beta\in S}\sum_{i\in R_{\beta}}e^{\beta(H_0(a))}\|Y_{-\beta, i}\|\bigg)^{-m}\\
& \leq \prod_{\beta\in S}\prod_{i\in R_{\beta}}\sum_{Y_{-\beta, i}\in\frac{1}{N}\Z\minus\{0\}}
\bigg(e^{\beta(H_0(a))}\|Y_{-\beta, i}\|\bigg)^{-\frac{m}{r}},
\end{align*}
where $r:=\sum_{\beta\in S} |R_{\beta}|\leq \dim \uuu_R^2$.
This last product equals
\[
 \bigg(\sum_{X\in\frac{1}{N}\Z\minus\{0\}}\|X\|^{-\frac{m}{r}}\bigg)^r
 \prod_{\beta\in S}\prod_{i\in R_{\beta}}e^{-m\beta(H_0(a))/r}
=\bigg(\sum_{X\in\frac{1}{N}\Z\minus\{0\}}\|X\|^{-\frac{m}{r}}\bigg)^r
 \prod_{\beta\in S}e^{-m|R_{\beta}|\beta(H_0(a))/r}.
\]
The sum $\sum_{X\in\frac{1}{N}\Z\minus\{0\}}\|X\|^{-\frac{m}{r}}$  is finite if $m>r$, so it is in particular finite if $m>\dim\uuu_R^2\geq r$, which gives our lower bound on $m$.
Now every $\beta$ is a non-negative linear combination of roots in $\Delta_1^2$ so that the above product equals
\[
 \bigg(\sum_{X\in\frac{1}{N}\Z\minus\{0\}}\|X\|^{-\frac{m}{r}}\bigg)^r
 \prod_{\alpha\in \Delta_1^2}e^{-k_{\alpha, S, \RRR}\alpha(H_0(a))}
\]
for suitable constants $k_{\alpha, S, \RRR}\geq0$.
Since $S$ has property $\Pi(P_1, R, P_2)$, there exists for every $\alpha\in\Delta_1^2\minus\Delta_1^R$ some $\beta\in S$ such that $\alpha$ occurs non-trivially in $\beta$. Hence, since $|R_{\beta}|>0$ for every $\beta\in S$, the corresponding coefficient satisfies $k_{\alpha, S, \RRR}>0$ if $\alpha\in \Delta_1^2\minus\Delta_1^R$, which finishes the proof.
\end{proof}

\begin{lemma}\label{bounding_sum_over_pseudoss}
For $\beta\in \Phi(A_1,\mmm_R)=\Phi(A_1^R, \mmm_R)=:\Phi_1^R$, denote by $\mmm_{\beta}\subseteq \mmm_R$ the eigenspace of $\beta$ in $\mmm_R$ so that $\mmm_R=\bigoplus_{\beta\in \Phi_1^R}\mmm_{\beta}$.
Put $A_1^R(T_1)=\{a\in A_1^R\mid \forall\alpha\in\Delta_1^R:~ \alpha(H_{P_1}(a)-T_1)>0\}$, let $k>1$ be given, and let $\nu>k+D$.

Then for every $\alpha\in\Delta_0^R$ there exists a constant $k_{\alpha}\geq0$, and for every $\beta\in\Phi_0^R$ a seminorm $\mu_{\beta}$ on $\SSS^1(\mmm_{\beta}(\A))\cup\SSS_{\nu}(\mmm_{\beta}(\A))$ such that the following holds:
\begin{itemize}
 \item  $k_{\alpha}>0$ for every $\alpha\in \Delta_0^R\minus\Delta_0^1$, and 

\item 
for all $\lambda>0$, all $\varphi_{\beta}\in\SSS^1(\mmm_{\beta}(\A))\cup\SSS_{\nu}(\mmm_{\beta}(\A))$, and all $a\in A_1^R(T_1)$ we have
\begin{equation}\label{sum_pseudoss}
 \delta_0^R(a)^{-1}\sum\limits_{\substack{X\in\widetilde{\mmm}_{P_1}^R(\frac{1}{N}\Z) \\ X\not\in\nnn}}
 \prod_{\beta\in \Phi_0^R}
\varphi_{\beta}(\lambda \beta(a)^{-1}X_{\beta})
\leq \begin{cases}
      \lambda^{-\dim\mmm_R}\mu(\varphi)\prod\limits_{\alpha\in \Delta_0^R\minus\Delta_0^1}e^{-k_{\alpha}\alpha(H_0(a))}			&\text{if } \lambda\leq1,\\[0.5cm]
      \lambda^{-k}\mu(\varphi)\prod\limits_{\alpha\in \Delta_0^R\minus\Delta_0^1}e^{-k_{\alpha}\alpha(H_0(a))}					&\text{if } \lambda\geq1,
     \end{cases}
\end{equation}
where $\mu(\varphi):=\prod_{\beta\in\Phi_1^R}\mu_{\beta}(\varphi_{\beta})$.
\end{itemize}
\end{lemma}

\begin{proof}
Suppose first that $R\neq P_1$.
The left hand side of \eqref{sum_pseudoss} can by Lemma \ref{identifying_nilpotents}\ref{identifying_nilp2} be bounded by a sum over non-empty subsets $S\subseteq \Sigma_1^R$ with property $\Pi(P_1, P_1, R)$ of the terms
\[
\bigg(\prod_{\beta\in S}\sum_{X_{-\beta}\in\mmm_{-\beta}(\frac{1}{N}\Z)\minus\{0\}}\varphi_{-\beta}(\lambda \beta(a)X_{-\beta})\bigg)
\bigg(\prod_{\beta\in\Phi_0^1\cup\Sigma_1^R}\sum_{X_{\beta}\in\mmm_{\beta}(\frac{1}{N}\Z)}\varphi_{\beta}(\lambda \beta(a)^{-1} X_{\beta})\bigg).
\]
Recall that if $V$ is a finite dimensional vector space, then for every $r>1$ there exists a seminorm $\mu_r$ on $\SSS_{r+\dim V}(V(\A))$ such that for all $s>0$ and all $\Psi\in\SSS_{r+\dim V}(V(\A))\cup \SSS^1(V(\A))$ we have
\[
 \sum_{X=(X_1, \ldots, X_{\dim V})\in V(\Q), X_1\neq0}|\Psi( s X)| \leq \mu_r(\Psi) \sup\{1,s^{-1}\}^{\dim V} \sup\{1,s\}^{-r},
\]
see, e.g., \cite[pp. 510-511]{Wr85}. (Note that in \cite{Wr85} this estimate was only proved for $\Psi\in \SSS(V(\A))$, but it is clear from the proof there that one only needs a polynomial decay of $\Psi$ up to a certain power and no differentiability at all.)
In particular, after possibly changing the seminorm in a way depending only on $\dim V$, we get
\begin{align}\label{est_wright}
 \sum_{X\in V(\Q)}|\Psi( s X)| & \leq \begin{cases}
                                     \mu_r(\Psi)s^{-\dim V}				&\text{if } s\leq 1,\\
				      \mu_r(\Psi)					&\text{if }s\geq1,\;\;\text{ and }
                                    \end{cases}\\\label{est_wright2}
\sum_{X\in V(\Q), X\neq0}|\Psi( s X)| & \leq \begin{cases}
                                     \mu_r(\Psi)s^{-\dim V}					&\text{if } s\leq 1,\\
				      \mu_r(\Psi)s^{-r}						&\text{if }s\geq1.
                                    \end{cases}
\end{align}
From this it follows that for every $\beta\in\{0\}\cup\Sigma_1^R$   there exists a seminorm $\mu_{\beta}$ on $\SSS_{\dim\mmm_{\beta}+ 1}(\mmm_{\beta}(\A))\cup\SSS^{1}(\mmm_{\beta}(\A))$ such that for all $\lambda>0$ and all $a\in A_1^{R}(T_1)$ we have
\[
 \sum_{X_{\beta}\in\mmm_{\beta}(\frac{1}{N}\Z)}\varphi_{\beta}(\lambda \beta(a)^{-1} X_{\beta})
\leq 		\begin{cases}
		    \mu_{\beta}(\varphi_{\beta})\beta(a)^{\dim\mmm_{\beta}} (\lambda^{-1}+1)^{\dim \mmm_{\beta}}		&\text{if }\beta\geq0,\\
		      \mu_{\beta}(\varphi_{\beta}) (\lambda^{-1}+1)^{\dim\mmm_{\beta}}						&\text{if }\beta<0.
		\end{cases}
\]
For this inequality also recall that $a\in A_0^R(T_1)$ implies that $\beta(a)$ is uniformly bounded from below if $\beta>0$.
Hence for all $\lambda>0$ and $a\in A_1^R(T_1)$,
\begin{align*}\label{pseudoss1}
 \prod_{\beta\in\{0\}\cup\Sigma_1^R}\sum_{X_{\beta}\in\mmm_{\beta}(\frac{1}{N}\Z)}\varphi_{\beta}(\lambda \beta(a)^{-1} X_{\beta})
& \leq \delta_0^R(a) (\lambda^{-1}+1)^{\dim \ppp_1} \prod_{\beta\in\{0\}\cup\Sigma_1^R} \mu_{\beta}(\varphi_{\beta})\\
& \leq\begin{cases}
     c\delta_0^R(a)\prod_{\beta\in\Phi_0^1\cup\Sigma_1^R} \mu_{\beta}(\varphi_{\beta})				&\text{if }\lambda\geq 1,\\
      c\delta_0^R(a)\lambda^{-\dim \ppp_1} \prod_{\beta\in\{0\}\cup\Sigma_1^R} \mu_{\beta}(\varphi_{\beta})		&\text{if }\lambda<1,
    \end{cases}
\end{align*}
where $c>0$ is some constant.

Similarly, for every $\beta\in S$ and every $k>1$, there is a seminorm $\mu_{-\beta, k}$ on $\SSS_{k+\dim \mmm_{-\beta}}(\mmm_{-\beta}(\A))\cup\SSS^{1}(\mmm_{-\beta}(\A))$ such that for all $\lambda>0$ and all $a\in A_0^R(T_1)$ we have
\[
\sum_{X_{-\beta}\in\mmm_{-\beta}(\frac{1}{N}\Z)\minus\{0\}}\varphi_{-\beta}(\lambda \beta(a)X_{-\beta})
\leq \begin{cases}
      \mu_{-\beta, k}(\varphi_{-\beta}) \lambda^{-k} \beta(a)^{-k} 					&\text{if } \lambda\geq 1,\\
      \mu_{-\beta, k}(\varphi_{-\beta}) (\lambda\beta(a))^{-\dim\mmm_{-\beta}} 				&\text{if } \lambda< 1.
     \end{cases}
\]
Hence,
\begin{equation*}\label{pseudoss2}
\prod_{\beta\in S}\sum_{X_{-\beta}\in\mmm_{-\beta}(\frac{1}{N}\Z)\minus\{0\}}\varphi_{-\beta}(\lambda \beta(a)X_{-\beta})
\leq \begin{cases}
      \lambda^{-k\sum_{\beta\in S}\dim\mmm_{-\beta}}\mu_{S,k}(\varphi)\prod_{\beta\in S}\beta(a)^{-k}				&\text{if }\lambda\geq1,\\
      \lambda^{-\sum_{\beta\in S}\dim\mmm_{-\beta}}\mu_{S, k}(\varphi)\prod_{\beta\in S}\beta(a)^{-\dim\mmm_{-\beta}}				&\text{if }\lambda<1,
     \end{cases}
\end{equation*}
where $\mu_{S,k}(\varphi):=\prod_{\beta\in S}\mu_{-\beta,k}(\varphi_{-\beta})$.
Now every $\beta\in S$ can be written as $\beta=\sum_{\alpha\in \Delta_0^R} b_{\beta, \alpha} \alpha$ for $b_{\beta, \alpha}\geq0$ suitable constants so that
$\sum_{\beta\in S}\beta=\sum_{\alpha\in\Delta_0^R} B_{\alpha}\alpha$ with $B_{\alpha}:=\sum_{\beta\in S}b_{\beta, \alpha}$. Since $S$ has property $\Pi(P_1, P_1, R)$, we have $B_{\alpha}>0$ if $\alpha\in \Delta_0^R\minus\Delta_0^1$ so that
\[
 \prod_{\beta\in S}\beta(a)^{-k}
\leq c \prod_{\alpha\in\Delta_0^R\minus\Delta_0^1} e^{-k B_{\alpha}\alpha(H_0(a))}
\]
 for a suitable constant $c>0$.
Multiplying the above estimates gives the assertion if $R\neq P_1$.
If $R=P_1$, we simply use the estimate for the sum over $X\in V(\Q)$, $X\neq0$, given in \eqref{est_wright} and \eqref{est_wright2}.
\end{proof}

\begin{rem}\label{remark_on_nilpotent}
 If $G=\GL_n$, then under the same assumptions and with the same notation as in the previous lemma, it follows that for a suitable seminorm $\mu$, we have for every $\lambda\in(0,1]$
\begin{equation}\label{sum_pseudoss_nilp}
 \delta_0^R(a)^{-1}\sum_{X\in\widetilde{\mmm}_{P_1}^R(\frac{1}{N}\Z) \cap\nnn}
 \prod_{\beta\in \Phi_1^R}
\varphi_{\beta}(\lambda \beta(a)^{-1}X_{\beta})
\leq  \lambda^{-\dim\mmm_R+1}\mu(\varphi)\prod_{\alpha\in \Delta_0^R\minus\Delta_0^1}e^{-k_{\alpha}\alpha(H_0(a))},
\end{equation}
since if $X$ is nilpotent, $\tr X=0$. Hence in the proof the sum over $X_0\in\mmm_0(\frac{1}{N}\Z)$ can be restricted to the vector subspace of traceless matrices which has codimension $1$.
Of course, similar versions of this inequality hold if we intersect $\mmm_0$ with other vector subspaces of positive codimension.
\end{rem}

\begin{lemma}\label{lemma_integral_torus}
Suppose we are given positive numbers $m_{\alpha}>0$ for each $\alpha\in\Delta_1^2$. 
Then for every sufficiently regular $T\in\aaa^+$ we have 
\begin{equation}\label{integral_over_torus}
 \int_{A_{1}^G}\sigma_1^2(H_0(a)-T)
\prod\limits_{\alpha\in \Delta_1^2}e^{-m_{\alpha}\alpha(H_0(a))}da
<\infty.
\end{equation}
\end{lemma}

\begin{proof}
This is essentially contained in \cite[p. 365]{Ch02} (cf.\ also \cite[p. 947]{Ar78}), but we need to find a sufficient lower bound for the $m_{\alpha}$.
We can write the integral in \eqref{integral_over_torus} as
\[
 \int_{\aaa_1^G}\sigma_1^2(H-T)
\prod\limits_{\alpha\in \Delta_1^2}e^{-m_{\alpha}\alpha(H)}dH.
\]
If $H\in \aaa_{P_1}^G$, we decompose it as $H=H_1+H_2$ with uniquely determined $H_1\in\aaa_1^2$ and $H_2\in \aaa_2^G$. Then $\sigma_1^2(H-T)\neq0$ implies  $t_{\alpha}:=\alpha(H)=\alpha(H_1)>\alpha(T)$ for all $\alpha\in\Delta_1^2$, and also the existence of a constant $c>0$ (independent of $H$) such that
\[
\|H_2\|\leq c (1+\sum_{\alpha\in \Delta_1^2}t_{\alpha})
\leq c\prod_{\alpha\in \Delta_1^2}(1+t_{\alpha})
\]
(cf.\ \cite[Corollary 6.2]{Ar78}).
Hence the volume in $\aaa_2^G$ of all contributing $H_2$ is bounded by a polynomial in the $t_{\alpha}$ for $\alpha\in\Delta_1^2$ so that there exists some $c>0$ such that the above integral is bounded by
\[
c \prod_{\alpha\in \Delta_1^2}\int_{\alpha(T)}^{\infty}(1+t_{\alpha})^k e^{-m_{\alpha}t_{\alpha}}dt_{\alpha}.
\]
Since $m_{\alpha}>0$ for all $\alpha\in \Delta_1^2$, this implies the assertion.
\end{proof}

Let $\bbb\subseteq \ggG$ be a subspace as in \S \ref{section_testfcts}, and let $S$ be a set of roots acting on $\bbb$ such that we have a direct decomposition $\bbb=\bigoplus_{\beta\in S} \bbb_{\beta}$. 
Let $\|\cdot\|$ denote a norm on $\bbb(\A)$ compatible with this direct sum decomposition (i.e., if $B=\sum_{\beta} B_{\beta}\in \bbb(\A)$, $B_{\beta}\in \bbb_{\beta}(\A)$, then $\|B\|=\sum_{\beta}\|B_{\beta}\|$). 

\begin{lemma}\label{std_estimates_for_SB}
Let $\nu>0$ be an integer. Then for every $Y\in\UUU(\bbb)_{\leq\nu}$, there exists a constant $c_Y>0$ such that the following holds:
For every $\Phi\in\SSS^{\nu}(\bbb(\A))$ (resp.\ $\Phi\in\SSS_{\nu}(\bbb(\A))$) there are functions $\varphi_{\beta}\in\SSS^{\infty}(\bbb_{\beta}(\A)))=\SSS(\bbb_{\beta}(\A))$ (resp.\ $\varphi_{\beta}\in \SSS_{\nu}(\bbb_{\beta}(\A))$), $\beta\in S$, such that
\begin{enumerate}[label=(\roman{*})]
\item 
$\varphi_{\beta}\geq0$ for all $\beta$.

\item 
$|\Phi(B)|\leq \prod_{\beta\in S}\varphi_{\beta}(B_{\beta})$ for all $B=\sum_{\beta\in S}B_{\beta}\in \bbb(\A)$.

\item  
For every tuple $(Y_{\beta})_{\beta\in S}\subseteq \bigoplus_{\beta\in S} \UUU(\bbb_{\beta})$ of degree $\sum_{\beta\in S} \deg Y_{\beta}\leq \nu$ we have
\[
\prod_{\beta\in S}\|Y_{\beta}*\varphi_{\beta}\|_{L^1(\bbb_{\beta}(\A))}\leq \mu(\Phi) \prod_{\beta\in S}c_{Y_{\beta}} 
\]
for $\mu(\Phi):=\sum_{X\in\BBB_{\bbb, \nu}}\|X\Phi\|_{L^1(\bbb(\A))}$, and $\mu$ is a seminorm on $\SSS(\bbb(\A))$ (resp.\ on $\SSS_{\nu}(\bbb(\A))$).
\end{enumerate} 
\end{lemma}
\begin{proof}
This follows from combining \cite[Lemme 2.1]{Ch02} with \cite[pp. 791-792]{FiLa10}: The main idea is to take the convolution of absolute values of certain derivatives of $\Phi$ with a non-negative function $\varphi\in\SSS(\bbb(\A))$, and note that convolution with $\varphi$ maps $\SSS^{\nu}(\bbb(\A))$ to $\SSS(\bbb(\A))$ and $\SSS_{\nu}(\bbb(\A))$ to $\SSS_{\nu}(\bbb(\A))$.
\end{proof}

\subsection{Proof of Lemma \ref{fundamental_convergence}}\label{subsection_prf_of_key_lemma}
\begin{proof}[Proof of Lemma \ref{fundamental_convergence}]
We basically follow the proof of \cite[Th\'{e}or\`{e}me 3.1]{Ch02}, but we need to keep track of the central variable $\lambda$ the whole time.

\begin{enumerate}[label=(\roman{*})]
\item 
Let $\lambda\geq1$. For $\ooo\in\OOO_*$, the truncated kernel $k_{\ooo}^T(x, \Phi)$ can be written as a sum over standard  parabolic subgroups $P_1, P_2$ with $P_1\subseteq P_2$ of 
\begin{equation}\label{descent_kernel}
k_{\ooo}^T(x, \Phi)
=\sum_{\substack{P_1, P, P_2: \\ P_1\subseteq P\subseteq P_2}}
\sum_{\delta\in P_1(\Q)\backslash G(\Q)}(-1)^{\dim A_P/A_G}
F^{P_1}(\delta x, T)\sigma_{P_1}^{P_2}(H_0(\delta x)-T)K_{P, \ooo}(\delta x, \Phi),
\end{equation}
 provided the right hand side converges, cf.\ \cite[Lemma 2.8]{Ch02}. 
Hence  the left hand side of \eqref{finiteness1} can be bounded from above by a sum over parabolic subgroups $P_1, P_2$ with  $P_1\subseteq P_2$, and over $\ooo\in \OOO_*$ of
\begin{multline*}
 \int_{A_GP_1(\Q)\backslash G(\A)}F^{P_1}(x, T)\sigma_{P_1}^{P_2}(H_0(x)-T)\cdot\\
 \left|\sum\limits_{P:~P_1\subseteq P\subseteq P_2}(-1)^{\dim A_P/A_G}
\sum\limits_{X\in\mmm_{P}(\Q)\cap \ooo}
\int_{\uuu_{P}(\A)}\Phi(\lambda \Ad x^{-1}(X+U))dU\right|dx,
\end{multline*} 
cf.\ \cite[pp. 360-361]{Ch02}.
This can be replaced by the sum over $P_1, R, P_2$ with $P_1\subseteq R\subseteq P_2$, and over $\ooo\in\OOO_*$ of 
\begin{multline}\label{temp11b}
 \int_{A_G P_1(\Q)\backslash G(\A)} F^{P_1}(x, T)\sigma_{P_1}^{P_2}(H_0(x)-T)\cdot\\
\sum_{X\in\widetilde{m}_{P_1}^R(\Q)\cap\ooo}
\left|\sum\limits_{P:~R\subseteq P\subseteq P_2}(-1)^{\dim A_P/A_G}
\sum_{Y\in\uuu_R^{P}(\Q)}
\int_{\uuu_{P}(\A)}\Phi(\lambda \Ad x^{-1}(X+Y+U))dU\right|dx.
\end{multline}
We can decompose 
\[
 A_G P_1(\Q)\backslash G(\A)
=U_1(\Q)\backslash U_1(\A) \times M_1(\Q)\backslash M_1(\A)^1\times A_1^G\times\cpt
\]
and write $x\in A_G P_1(\Q)\backslash G(\A)$ accordingly as $x=umak$. Then $F^{P_1}(x, T)=F^{P_1}(m, T)$.
Following the arguments on \cite[p. 361]{Ch02}, we can replace $\Phi$ by $\int_{\Gamma}\Phi(\Ad g^{-1} \cdot) dg\in\SSS^{\nu}(\ggG(\A))$ for a suitable compact subset $\Gamma\subseteq G(\A)^1$ (depending on $T$), and consider instead of the integral above the sum over $P_1, R, P_2$ with $P_1\subseteq R\subseteq P_2$, and $\ooo\in\OOO_*$ of
\begin{multline}\label{temp12}
\int_{A_1^G}e^{-2\rho_0(H_0(a))}\sigma_{P_1}^{P_2}(H_0(a)-T)\cdot \\
\sum_{X\in\widetilde{m}_{P_1}^R(\Q)\cap\ooo}
\left|\sum_{P:~R\subseteq P\subseteq P_2}(-1)^{\dim A_P/A_G}
\sum_{Y\in\uuu_R^{P}(\Q)}
\int_{\uuu_{P}(\A)}\Phi(\lambda \Ad a^{-1}(X+Y+U))dU\right|da.
\end{multline}

We now distinguish the cases $R=P_2$ and $R\subsetneq P_2$. For $R=P_2$, \eqref{temp12} equals the sum over $P_1\subseteq P_2$ of 
\begin{multline}\label{temp13}
\int_{A_{1}^G}e^{-2\rho_0(H_0(a))}\sigma_{P_1}^{P_2}(H_0(a)-T)
\sum_{\substack{X\in\widetilde{\mmm}_{P_1}^{P_2}(\Q) \\ X\not\in \nnn}}
\left|\int_{\uuu_{P_2}(\A)}\Phi(\lambda \Ad a^{-1}(X+U))dU\right|da\\
=\lambda^{-\dim\uuu_{P_2}}\int_{A_{1}^G}e^{-2(\rho_0-\rho_2)(H_0(a))}\sigma_{P_1}^{P_2}(H_0(a)-T)
\sum_{\substack{X\in\widetilde{\mmm}_{P_1}^{P_2}(\Q) \\ X\not\in \nnn}}
\left|\Psi_{P_2}(\lambda \Ad a^{-1}X)\right|da,
\end{multline}
where $\Psi_{P_2}(Y):=\int_{\uuu_{P_2}(\A)}\Phi(Y+U)dU \in\SSS^{\nu}(\mmm_{2}(\A))$.

For $R\subsetneq P_2$, we apply Poisson summation with respect to the sum over $Y$. In the resulting alternating sum many terms cancel out as explained in  \cite[pp. 362-363]{Ch02}. So the sum over $R\subsetneq P_2$ of \eqref{temp12} can be bounded by the sum over $P_1, R, P_2$, $P_1\subseteq R\subsetneq P_2$, of
\begin{multline}\label{temporary1}
\int_{A_{1}^G}e^{-2\rho_0(H_0(a))}\sigma_{P_1}^{P_2}(H_0(a)-T)\cdot\\
\sum_{\substack{X\in\widetilde{\mmm}^{R}_{P_1}(\Q) \\ X\not\in \nnn}}
\left|\sum_{\bar{Y}\in\overline{\uuu}_{R}^{P_2\prime}(\Q)}
\int_{\uuu_{R}(\A)}\Phi(\lambda\Ad a^{-1}(X+U))\psi(\langle U, \bar{Y}\rangle ) dU\right|da.
\end{multline}
For our purposes,  we can replace $\Phi$ by Lemma \ref{std_estimates_for_SB} by the product $\Psi_{\mmm_R}\Psi_{\uuu_R}$
 with  $\Psi_{\mmm_R}\in\SSS(\mmm_R(\A))$, $\Psi_{\uuu_R}\in\SSS(\uuu_R(\A))$, $\Psi_{\mmm_{R}}, \Psi_{\uuu_R}\geq0$,
satisfying the inequalities of Lemma \ref{std_estimates_for_SB}.

Changing variables, we may consider instead of \eqref{temporary1}  the integral
\begin{multline}\label{temp14}
\lambda^{-\dim\uuu_R}\int_{A_{1}^G}e^{-2(\rho_0-\rho_R)(H_0(a))}\sigma_{P_1}^{P_2}(H_0(a)-T)\cdot\\
\sum_{\substack{X\in\widetilde{\mmm}^{R}_{P_1}(\Q) \\ X\not\in \nnn}}\Psi_{\mmm_R}(\lambda \Ad a^{-1}X)
\cdot\sum_{\bar{Y}\in\overline{\uuu}_{R}^{P_2\prime}(\Q)}
\int_{\uuu_{R}(\A)}\Psi_{\uuu_R}(U)\psi(\langle U, \lambda^{-1} \Ad a^{-1}\bar{Y}\rangle ) dU~da.
\end{multline}
The compact support of $\Phi$ at the finite places implies the existence of $N\in\N$ such that all contributing $\bar{Y}$ and $X$ must have coordinates in $\frac{1}{N}\Z$.
Let $m\geq 0$ be a sufficiently large even integer. By standard estimates for Schwartz-Bruhat functions, 
\begin{align*}
\left|\int_{\uuu_{R}(\A)}\Psi_{\uuu_R}(U)\psi(\langle U, \lambda^{-1} \Ad a^{-1}\bar{Y}\rangle )\right| dU
& \leq \|\lambda^{-1}\Ad a^{-1}\bar{Y}\|^{-m}\sum_{D\in\BBB_{m/2}}\int_{\uuu_R(\A)}\left|(D\Psi_{\uuu_R})(U)\right|dU\\
& =:\|\lambda^{-1}\Ad a^{-1}\bar{Y}\|^{-m}\mu_{\uuu_R}^m(\Psi_{\uuu_R}).
\end{align*}
This last sum over the set of differential operators  defines the seminorm $\mu_{\uuu_R}^m:\SSS(\uuu_R(\A))\longrightarrow\R_{\geq0}$.
Hence \eqref{temp14} is bounded by
\begin{multline}\label{temp15}
\lambda^{m-\dim\uuu_R}\mu_{\uuu_R^m}(\Psi_{\uuu_R})
\int_{A_{1}^G}e^{-2(\rho_0-\rho_R)(H_0(a'a))}\sigma_{P_1}^{P_2}(H_0(a)-T)\cdot\\
\bigg(\sum_{\bar{Y}\in\overline{\uuu}_{R}^{P_2\prime}(\frac{1}{N}\Z)}\|\Ad a^{-1}\bar{Y}\|^{-m}\bigg)
\bigg(\sum_{\substack{X\in\widetilde{\mmm}^{R}_{P_1}(\frac{1}{N}\Z) \\ X\not\in \nnn}}
\left|\Psi_{\mmm_R}(\lambda  \Ad a^{-1}X)\right|\bigg)da.
\end{multline}
Write $\mmm_R=\bigoplus_{\beta\in\Phi_1^R} \mmm_{R,\beta}$ for the eigenspace decomposition of $\mmm_R$ with respect to $\Phi_1^R=\Phi(A_1, M_R)$.
In particular, $\mmm_{R, 0}=\mmm_{1}$.
By Lemma \ref{std_estimates_for_SB} there are $\varphi_{\beta}\in\SSS(\mmm_{R,\beta}(\A))$, $\varphi_{\beta}\geq0$, such that
$|\Psi_{\mmm_R}(Z )|
\leq
\prod_{\beta\in \Phi(A_0, M_R)}\varphi_{\beta}(Z_{\beta})$
for all $Z=\sum_{\beta}Z_{\beta}\in \mmm_R(\A)=\bigoplus_{\beta}\mmm_{R, \beta}$, and such that they satisfy the  estimates of Lemma \ref{std_estimates_for_SB}.
With this, \eqref{temp15} is bounded by
\begin{multline}\label{temp16}
\lambda^{m-\dim \uuu_R}
\mu_{\uuu_R}(\Psi_{\uuu_R})
\int_{A_{1}^G}e^{-2(\rho_0-\rho_R)(H_0(a))}\sigma_{P_1}^{P_2}(H_0(a)-T)\cdot\\
\bigg(\sum_{\bar{Y}\in\overline{\uuu}_{R}^{P_2\prime}(\frac{1}{N}\Z)}\|\Ad a^{-1}\bar{Y}\|^{-m}\bigg)
\bigg(\sum_{\substack{X\in\widetilde{\mmm}^{R}_{P_1}(\frac{1}{N}\Z) \\ X\not\in \nnn}}
\prod_{\beta\in \Phi_1^R}
\varphi_{\beta}(\lambda \beta(a)^{-1}X_{\beta})\bigg)da.
\end{multline}
If $m>\dim\uuu_R^{P_2}$, then by Lemma \ref{lattice_sum_lemma} there are $c_1>0$, and real numbers $k_{\alpha}\geq0$ for $\alpha\in \Delta_0^2$ with $k_{\alpha}>0$ whenever $\alpha\in\Delta_0^2\minus\Delta_0^R$, such that
\begin{equation}\label{temp19}
\sum_{\bar{Y}\in\overline{\uuu}_{R}^{P_2\prime}(\frac{1}{N}\Z)}\|\Ad a^{-1}\bar{Y}\|^{-m}
\leq c_1
\prod_{\alpha\in\Delta_0^2}e^{-k_{\alpha}\alpha(H_0(a))}.
\end{equation}
Setting $\sum_{\bar{Y}\in\overline{\uuu}_{R}^{P_2\prime}(\frac{1}{N}\Z)}\|\Ad a^{-1}\bar{Y}\|^{-m}:=1$
and $m=0$ in the case $P_2=R$, we can consider the cases $P_2=R$ and $R\subsetneq P_2$ together.

To the second product in \eqref{temp16} we apply Lemma \ref{bounding_sum_over_pseudoss}. This we are allowed to, since  $\sigma_{P_1}^{P_2}(H_0(a)-T)\neq0$ implies that $a\in A_1^2(T_1)$.
Thus \eqref{temp15} is bounded by a finite sum of terms of the form
\begin{equation}\label{int_torus}
c' \lambda^{-N+m-\dim\uuu_{R}}
\int_{A_{1}^G}\sigma_{P_1}^{P_2}(H_1(a)-T)
\prod_{\alpha\in\Delta_0^2\minus \Delta_0^1}e^{-l_{\alpha}\alpha(H_0(a))}da
\end{equation}
for all $\lambda\geq 1$ and all $N>0$, where $l_{\alpha}>0$ and $c'>0$ are constants depending only on $N$.
By Lemma \ref{lemma_integral_torus} the second integral is finite.  Thus \eqref{finiteness1} is proven.

\item
Now assume that $\lambda\in (0, 1]$. We essentially argue as above, but have to change the upper bounds for the two products occurring in the integral \eqref{temp16}. 
We apply Lemma \ref{lattice_sum_lemma} to bound the left hand side of \eqref{temp19} again by the same quantity as before.
To bound the last term in the integral in  \eqref{temp16}, we use Lemma \ref{bounding_sum_over_pseudoss} giving for this term an upper bound of
\[
\lambda^{-\dim\mmm_R}\delta_0^R(a)\prod_{\alpha\in \Delta_0^R\minus\Delta_0^1}e^{-k_{\alpha}\alpha(H_0(a))}
\]
times the value of some seminorm applied to the $\varphi_{\mu}$'s.
Hence  \eqref{temp15} is bounded by  the product of the value of a seminorm (depending on $m$) applied to $\Phi$ with
\[
\lambda^{m-\dim\uuu_{R}-\dim\mmm_{R}}
\]
with $m>\dim\uuu_{R}^{P_2}$ arbitrary if $R\neq P_2$ and $m=0$ if $R=P_2$,
and
\[
\int_{A_{1}^G}\sigma_{P_1}^{P_2}(H_0(a)-T)
\prod_{\tilde{\alpha}\in\Delta_1^2}e^{-l_{\alpha}'\tilde{\alpha}(H_0(a))}da
\]
for suitable $l_{\alpha}'>0$.
Since (for $P_2=R$ as well as $R\neq P_2$)
\begin{equation}\label{inequality_of_dimensions}
\dim\uuu_{R}-m+\dim\mmm_R
\leq \dim\ggG=D
\end{equation}
the assertion \eqref{finiteness2} follows again from Lemma \ref{lemma_integral_torus}.

\item It is clear from the proof of the first part of the lemma that if $\nu$ is sufficiently large with respect to $N$, then the analogue assertion holds for $\Phi\in\SSS_{\nu}(\ggG(\A))$ instead of $\Phi\in\SSS^{\nu}(\ggG(\A))$.
\end{enumerate}
\end{proof}

\begin{rem}\label{remark}
In \eqref{inequality_of_dimensions} we have  $\dim\uuu_R+\dim\mmm_R\leq\dim\ggG-1$ unless $R=P_2=G$, and if $R\subsetneq G$ we have
$\dim\uuu_R+\dim\mmm_R\leq \dim\ggG-2$ unless $n=2$.
\end{rem}

\section{Nilpotent auxiliary distributions}\label{nilp_distr1}
{\bf From now on we assume $n\leq 3$.}

Recall that $\nnn\subseteq \ggG(\Q)$ denotes the set of nilpotent elements. Note that $\nnn$ is the same for $G=\SL_n$ and $G=\GL_n$.
Under the adjoint action of $G(\Q)$ the set $\nnn\subseteq \ggG(\Q)$ of nilpotent elements decomposes into finitely many nilpotent orbits $\NNN\subseteq \nnn$. 
If $\NNN\neq\{0\}$ and $X_0\in \NNN$, $X_0$ can be embedded into an $\mathfrak{sl}_2$-triple $\{X_0, Y_{X_0}, H_{X_0}\}\subseteq \ggG$ with $H_{X_0}$ semisimple and $Y_{X_0}$ nilpotent. 
The element $H_{X_0}$ defines a grading on $\ggG$, $\ggG=\bigoplus_{i\in\Z}\ggG_i$ with
$\ggG_i=\{X\in\ggG\mid [H_{X_0}, X]=iX\}$
and $X_0\in \ggG_2$.
We set 
$\ppp_{X_0}=\bigoplus_{i\geq 0}\ggG_i$, which is the associated Jacobson-Morozov parabolic subalgebra,
$\uuu_{X_0}=\bigoplus_{i>0}\ggG_i$, 
$\uuu_{X_0}^{>j}=\bigoplus_{i>j}\ggG_i$,
 and 
 $\uuu_{X_0}^{\geq j}=\bigoplus_{i\geq j}\ggG_i$.
Correspondingly, let 
$P_{X_0}=M_{X_0}U_{X_0}\subseteq G$
be the Jacobson-Morozov parabolic subgroup with Lie algebra $\ppp_{X_0}$ and Levi part
 $M_{X_0}$ with Lie algebra $\mmm_{X_0}=\ggG_0$, 
 and unipotent radical
 $U_{X_0}$ with Lie algebra $\uuu_{X_0}$.
The representative $X_0$  of $\NNN$ can be chosen such that $P_{X_0}$ is a standard parabolic subgroup.
If $\NNN=\{0\}$, then $X_0=0$, and we set $H_{X_0}=0$, $P_{X_0}=G$.
The grading of $\uuu_{X_0}$ induces a descending subgroup filtration $\{U_{X_0}^{\geq j}\}_{j>0}$ on $U_{X_0}$.
  Then 
$\NNN=\bigcup_{\delta\in P_{X_0}(\Q)\backslash G(\Q)}\Ad \delta^{-1} \cdot \uuu_{X_0}^{\geq 2 }(\Q)$
(disjoint union), and the action of $M_{X_0}$ on $\uuu^2_{X_0}=\ggG_2$ defines a prehomogeneous vector space, i.e. the orbit $V_0:=\Ad M_{X_0} X_0\subseteq \ggG_2$ is open and dense.
By definition, $H_{X_0}\in \aaa_{M_{X_0}}\subseteq \ggG(\R)$, and for every $\lambda\in \R_{>0}$ we have 
\[
\Ad ( \eta_{X_0,\lambda} )X_0=\lambda X_0,
~~~
\text{ where }
~~~~
\eta_{X_0,\lambda}:=e^{\frac{\log\lambda}{2}H_{X_0}}\in Z^{M_{X_0}}(\A)
\]
($Z^{M_{X_0}}$=center of $M_{X_0}$).
Let $C_{M_{X_0}}(X_0)=\{m\in M_{X_0}\mid \Ad m^{-1} X_0=X_0\}$ be the stabiliser of $X_0$ under the action of $M_{X_0}$, and  $C_{U_{X_0}}(X_0)=\{u\in U_{X_0}\mid \Ad u^{-1} X_0=X_0\}$ the stabiliser of $X_0$ in $U_{X_0}$.
If there is no danger of confusion, we drop the subscript $X_0$ and write $H=H_{X_0}$, $P=P_{X_0}$, etc.
\begin{example}\label{example_nilp}
We choose the following representatives for our nilpotent orbits in the cases $n=2,3$:
\begin{itemize}
 \item $n=2$. There are two nilpotent orbits, the trivial and the regular one:
\begin{center}
 \begin{tabular}{lccccc}
  $\NNN$		&	$X_0$	&	$H_{X_0}$	&	$P_{X_0}$	&	$C_U(X_0)$		\\
\noalign{\smallskip}\hline\noalign{\smallskip}
$\NNN_{\text{triv}}$	&	$0$	&	$0$	&	$G$	&	$\{\One_2\}$		\\
\noalign{\smallskip}
$\NNN_{\text{reg}}$	&	$\left(\begin{smallmatrix}0&1\\0&0\end{smallmatrix}\right)$	&	$\left(\begin{smallmatrix}1&\,\\\,&-1\end{smallmatrix}\right)$	&	$P_0$	&	$U_0$\\
 \noalign{\smallskip}\hline
\end{tabular}
\end{center}

\item $n=3$. There are three nilpotent orbits, the trivial, the minimal (=subregular), and the regular one:
\begin{center}
\begin{tabular}{lccccc}
  $\NNN$		&	$X_0$	&	$H_{X_0}$	&	$P_{X_0}$	&	$C_U(X_0)$		\\
\noalign{\smallskip}\hline\noalign{\smallskip}
$\NNN_{\text{triv}}$	&	$0$	&	$0$	&	$G$	&	$\{\One_2\}$		\\
\noalign{\smallskip}
$\NNN_{\text{min}}$	&	$\left(\begin{smallmatrix} 0&0&1\\0&0&0\\0&0&0\end{smallmatrix}\right)$	&	$\left(\begin{smallmatrix}1&\,&\,\\\,&0&\,\\\,&\,&-1\end{smallmatrix}\right)$&	$P_0$	&	$U_0$	\\
\noalign{\smallskip}
$\NNN_{\text{reg}}$	&	$\left(\begin{smallmatrix}0&1&0\\0&0&1\\0&0&0\end{smallmatrix}\right)$	&	$\left(\begin{smallmatrix}2&\,&\,\\\,&0&\,\\\,&\,&-2\end{smallmatrix}\right)$	&	$P_0$	&	$\{\left(\begin{smallmatrix}1&x_1&x_2\\0&1&x_3\\0&0&1\end{smallmatrix}\right)\mid x_1=x_3\}$\\
\noalign{\smallskip}\hline
 \end{tabular}
\end{center}
\end{itemize}
In all of these examples  we fix measures on $C_U(X_0,\A)$ and $C_M(X_0,\A)$ in the obvious way.
\end{example}

The following is a slight variant of  \cite[Theorem 1]{Ra72}.
\begin{lemma}\label{changing_variables}
There exists a constant $c>0$ such that for every function $f:\overline{V_0(\A)}=\ggG_2(\A)\longrightarrow\C$, which is  integrable and for which all occurring integrals are finite, we have
\[
 \int_{C_M(X_0,\A) \backslash M(\A)}f(\Ad m^{-1} X_0)\delta_{U^{\leq 2}}(m)^{-1}da
=c\int_{V_0(\A)}\varphi(X)f(X)dX,
\]
where  $\varphi:\ggG_2(\A)\longrightarrow\C$ is defined as follows: Let $Z_1, \ldots, Z_r$ be a basis of $\ggG_1$, and $Z_1', \ldots, Z_r'$ a basis of $\ggG_{-1}$, which are dual to each other with respect to the Killing form. For $X\in\ggG_2$ write $[X, Z_i']=\sum c_{ji}(X)Z_j$, and set $\varphi(X)=|\det (c_{ij}(X))_{i,j})|^{\frac{1}{2}}$.
\end{lemma}

\begin{example}
For  $\NNN$ the trivial or regular orbit from Example \ref{example_nilp}, we have $\ggG_1=0=\ggG_{-1}$ so that $\varphi(X)\equiv1$.
If $n=3$ and $\NNN=\NNN_{\text{min}}$, then $\ggG_1=\{\left(\begin{smallmatrix} 0&*&0\\0&0&*\\0&0&0\end{smallmatrix}\right)\}$ and $\varphi(\left(\begin{smallmatrix} 0&0&x\\0&0&0\\0&0&0\end{smallmatrix}\right))=|x|$.  
\end{example}

\begin{proof}[Proof of Lemma \ref{changing_variables}]
Let $X\in\ggG_2(\A)$ and $m\in M(\A)$. Then $\varphi$ transforms according to \cite[Lemma 2]{Ra72} via
$\varphi(\Ad m X)=|\det \Ad m_{|\ggG_1}|\varphi(X)
=\delta_{\ggG_1}(m)\varphi(X)$.
 Let
\begin{align*}
\Lambda_1(f) & =\int_{C_M(X_0,\A)\backslash M(\A)}f(\Ad m^{-1} X_0)\delta_{U^{\leq 2}}(m)^{-1}dm,\;\;\;\text{ and }\\
\Lambda_2(f) & =\int_{V_0(\A)}\varphi(X)f(X)dX.
\end{align*}
Let $m_0\in M(\A)$ and put $f^{m_0}(X)=f(\Ad m_0^{-1} X)$. Then
\begin{align*}
\Lambda_1(f^{m_0}) & = \int_{C_M(X_0,\A)\backslash M(\A)}f(\Ad {m_0}^{-1} \Ad m^{-1} X_0)\delta_{U^{\leq 2}}(m)^{-1}dm\\
& =\int_{C_M(X_0,\A)\backslash M(\A)}f( \Ad(mm_0)^{-1}X_0)\delta_{U^{\leq 2}}(m)^{-1}dm
=\delta_{U^{\leq 2}}(m_0)\Lambda_1(f),
\end{align*}
and, using the above transformation property of $\varphi$, 
\begin{align*}
\Lambda_2(f^{m_0})&=\int_{V_0(\A)}\varphi( X)f( \Ad m_0^{-1} X)dX
=\delta_{\ggG_2}(m_0)\int_{V_0(\A)}\varphi(\Ad m_0 X)f( X) dX\\
&=\delta_{\ggG_2}(m_0)\delta_{\ggG_1}(m_0)\int_{V_0(\A)}\varphi( X)f( X)dX
=\delta_{U^{\leq2}}(m_0)\int_{V_0(\A)}\varphi(X)f( X)dX.\qedhere
\end{align*}
\end{proof}

We need to introduce a certain auxiliary distribution $\tilde{j}_{\NNN}^T$ on $\SSS^{\nu}(\ggG(\A))\cup\SSS_{\nu}(\ggG(\A))$:

\begin{definition}
If $T\in\aaa^+$ is sufficiently regular, we set
\[
 \tilde{j}_{\NNN}^T(\Phi)=\int_{A_GG(\Q)\backslash G(\A)}
 \tilde{F}^{M}(x, T)\sum\limits_{\gamma\in \NNN}\Phi(\Ad x^{-1}\gamma )dx
\]
where the truncation function $\tilde{F}^{M}(\cdot, T): G(\A)\longrightarrow\C$ is defined as the characteristic function of the set of all $x\in G(\A)$ of the form $x=umk$, $m\in M(\A)$, $u\in U(\A)$, $k\in\cpt$, satisfying
\[
\forall \varpi\in \widehat{\Delta}_0~ \forall \gamma\in M(\Q):~ \varpi(H_0(\gamma m)-T)\leq 0.
\]
\end{definition}

Note that $\tilde{F}^{M}(umk, T)=\tilde{F}^{M}(m, T)=F^{M}(m, T)\hat{\tau}_{P}(T-H_0(m))$.

\begin{lemma}\label{abs_conv_of_aux_distr}
Let $T\in\aaa^+$ be sufficiently regular.
For every $\nu\geq1$, there exists a seminorm $\mu$ on $\SSS^{\nu}(\ggG(\A))\cup\SSS_{\nu}(\ggG(\A))$ such that 
\[
  \int_{A_GG(\Q)\backslash G(\A)}
 \tilde{F}^{M}(x, T)\sum_{\gamma\in \NNN}|\Phi(\lambda \Ad x^{-1}\gamma )|dx
\leq \lambda^{-\dim\ggG} \mu(\Phi)
\]
for all $\Phi\in \SSS^{\nu}(\ggG(\A))\cup\SSS_{\nu}(\ggG(\A))$.
\end{lemma}

\begin{proof}
 First note that for every $x=utk\in U_0(\A)T_0(\A)\cpt$ we have $0\leq \tilde{F}^{M}(x, T)\leq \hat{\tau}_0^G(T-H_0(t))$  (As $n\leq 3$, $M=T_0$ for $\NNN\neq \{0\}$ so that $\tilde{F}^{M}(x,T)$ in fact then equals $\hat{\tau}_0^G(T-H_0(t))$.)
Using the standard estimates for integration over Siegel sets as in \eqref{siegel_est}, we get
\begin{multline*}
  \int_{A_GG(\Q)\backslash G(\A)}
 \tilde{F}^{M}(x, T)\sum_{\gamma\in \NNN}|\Phi(\lambda \Ad x^{-1}\gamma )|dx\\
\leq \int_{A_0^G(T_1)} \delta_0(a)^{-1} \hat{\tau}_0^G(T-H_0(a)) \sum_{\gamma\in \NNN}|\Phi'(\lambda \Ad a^{-1}\gamma )|da
\end{multline*}
for $\Phi'$ obtained from $\Phi$ by integration over a suitable compact domain in $G(\A)$.
Now there exists a seminorm $\mu$ on $\SSS^{\nu}(\ggG(\A))\cup\SSS_{\nu}(\ggG(\A))$ such that 
\[
 \delta_0(a)^{-1}\sum_{\gamma\in \NNN}|\Phi'(\lambda \Ad a^{-1}\gamma )| 
\leq  \delta_0(a)^{-1}\sum_{\gamma\in \ggG(\Q)}|\Phi'(\lambda \Ad a^{-1}\gamma )|  
\leq \lambda^{-\dim \ggG} \mu(\Phi)
\]
for every $\lambda\in(0,1]$ and $a\in A_0^G(T_1)$. Hence the original sum-integral is bounded by
\[
 \lambda^{-\dim \ggG} \mu(\Phi)\int_{A_0^G(T_1)}\hat{\tau}_0^G(T-H_0(a))da
=\lambda^{-\dim \ggG} \mu(\Phi)\int_{\aaa}\tau_0^G(A-T_1)\hat{\tau}_0^G(T-A)dA
<\infty,
\]
which proves the lemma.
\end{proof}

The distribution $\tilde{j}_{\NNN}^T(\Phi_{\lambda})$ has the nice property that as a function of $\lambda$ it is almost homogeneous in the following sense:

\begin{lemma}\label{lemma_chang_central_var}
Let $I\subseteq \R_{>0}$ be a compact interval. If $T\in\aaa^+$ is sufficiently regular such that $T+\frac{\log\lambda}{2} H$ is also sufficiently regular for every $\lambda\in I$, then
\[
 \tilde{j}_{\NNN}^{T}(\Phi_{\lambda})
=\lambda^{-\dim \NNN/2}
\tilde{j}_{\NNN}^{T+\frac{\log\lambda}{2} H}(\Phi).
\]
\end{lemma}

\begin{proof}
Replacing $\Phi$ by $\int_{\cpt}\Phi(\Ad k^{-1} \cdot)dk$ if necessary, we may assume that $\Phi$ is invariant under conjugation by $\cpt$. 
For $\NNN=\{0\}$ there is nothing to show so that we may assume $\NNN\neq\{0\}$. Then $M=T_0$ and $U=U_0$.
The integral defining $\tilde{j}_{\NNN}^T(\Phi_{\lambda})$ can be written as
\begin{multline*}
\int_{A_G M(\Q)\backslash M(\A)}\int_{U(\Q)\backslash U(\A)}
\delta_P(m)^{-1}\tilde{F}^{M}(m, T)
 \sum_{\gamma\in \uuu^{\geq 2}(\Q)\cap\NNN}\Phi(\lambda \Ad (um)^{-1}\gamma )du~dm\\
=\int_{A_G M(\Q)\backslash M(\A)}\sum_{\gamma\in\ggG_2(\Q)\cap\NNN} \int_{C_U(\gamma,\Q)\backslash U(\A)}
\delta_P(m)^{-1} \tilde{F}^{M}(m, T)
\Phi(\lambda \Ad (um)^{-1}\gamma )du~dm.
\end{multline*}
By \cite[Lemma 1]{Ra72},  the map $ C_U(\gamma,\A)\backslash U(\A)\ni u\mapsto \Ad u^{-1}\gamma\in \gamma+\uuu^{> 2}(\A)$  is a diffeomorphism with trivial Jacobian. We denote its inverse by $ \gamma+\uuu^{>2}(\A)\ni\gamma+U\mapsto u(\gamma, U)\in C_U(\gamma,\A)\backslash U(\A)$.
Thus the above integral equals
\begin{multline*}
\int_{A_G M(\Q)\backslash M(\A)}\delta_{U^{\leq 2}}(m)^{-1}\tilde{F}^{M}(m, T)\cdot\\
\int_{\uuu^{>2}(\A)}
\sum_{\gamma\in\uuu^2(\Q)\cap\NNN}
 \bigg(\int_{C_U(\gamma,\Q)\backslash C_U(\gamma,\A)} dv\bigg)\Phi(\lambda (\Ad m^{-1}\gamma +U ))dU~dm.
\end{multline*}
Now for $n\leq 3$ it is easily seen that  $\vol(C_U(\gamma, \Q)\backslash C_U(\gamma,\A))=1$ for all occurring $\gamma$.
Hence the integral equals
\[
\lambda^{-\dim\uuu^{>2}}\int_{A_G C_M(X_0,\Q)\backslash M(\A)} 
\delta_{U^{\leq 2}}(m)^{-1} \tilde{F}^{M}(m, T)
 \int_{\uuu^{>2}(\A)}
\Phi(\lambda \Ad m^{-1}X_0 +U )dU~dm.
\]
By Lemma \ref{changing_variables} this equals the product of $c\lambda^{-\dim\uuu^{>2}}$ with
\[
\int_{V_0(\A)} \varphi(X)
 \int_{\uuu^{>2}(\A)}
 \bigg(\int_{A_GC_M(X_0,\Q)\backslash C_{M}(X_0,\A)} 
\tilde{F}^{M}(m'm(X_0,X), T)dm'\bigg)
\Phi(\lambda X +U )dU~dX,
\]
where we used the map 
$C_M(X_0,\A)\backslash M(\A)\ni m\mapsto \Ad m^{-1} X_0
=:X\in V_0(\A)$.
We denote its inverse  by
 $ V_0(\A)\ni X\mapsto  m(X_0, X)=:m(X)\in C_M(X_0,\A)\backslash M(\A)$. 
Changing $X$ to $\lambda^{-1}X$ we obtain 
\[
c\lambda^{-\delta(\NNN)}\int_{V_0(\A)} \varphi(X)
 \int_{\uuu^{>2}(\A)}
 \bigg(\int_{A_GC_M(X_0,\Q)\backslash C_M(X_0,\A)} \tilde{F}^M(m'm\left(\lambda^{-1}X\right), T)dm'\bigg)
\Phi(X +U )dU~dX
\]
with
\[
 \delta(\NNN)=\dim\uuu^{>2}+\dim V_0+\frac{1}{2}\dim\uuu^1
=\dim\NNN/2
\]
where the last equality follows from \cite[Lemma 4.1.3]{CoMc93}.
Now
\[
m\left( \lambda^{-1}X\right)= \eta_{X_0, \lambda}^{-1} m\left( X\right),
\]
since there is $m_0\in M(\A)$ with $X=\Ad m_0 X_0$ so that
\[
\lambda^{-1}X
=\Ad m_0 (\lambda^{-1}X_0)
=\Ad m_0(\Ad\eta_{X_0, \lambda^{-1}}X_0)
=\Ad \eta_{X_0, \lambda^{-1}}(\Ad m_0 X_0)
=\Ad \eta_{X_0, \lambda}^{-1} X.
\]
Hence the above integral equals
\[
c\lambda^{-\delta(\nu)}\int_{V_0(\A)} \varphi(X)
 \int_{\uuu^{>2}(\A)}
 \bigg(\int_{A_GC_M(X_0,\Q)\backslash C_M(X_0,\A)}
\tilde{F}^{M}(\eta_{X_0, \lambda }^{-1}m'm\left(X\right), T)dm'\bigg)
 \Phi(X +U )dU~dX.
\]
By definition of $\tilde{F}^M(\cdot, T)$ we have
\[
\tilde{F}^M(\eta_{X_0,\lambda}^{-1}m'm(X), T)
=\tilde{F}^M(m'm(X), T+\frac{\log\lambda}{2}H).
\]
Inserting this and reversing the changes of variables (except for those involving $\lambda$), we obtain
\[
\lambda^{-\delta(\NNN)}\tilde{j}^{T+\frac{\log\lambda}{2}H}(\Phi).
\]
Since $I\subseteq\R_{>0}$ is compact, $\log I\subseteq \R$ is compact as well and choosing $T\in\aaa^+$ very large, we can ensure that $T$ as well as $T+\frac{\log\lambda}{2}H$ are sufficiently regular for all $\lambda\in I$. In this case, all occurring integrals are well-defined and absolutely convergent so that all changes of variables are justified.
\end{proof}

\section{Nilpotent distributions, continuation of $\Xi^T(s, \Phi)$ and  functional equation}\label{section_main_res}
 We need to attach a further auxiliary distribution to the nilpotent orbit $\NNN$, 
namely, $j_{\NNN}^T:\SSS^{\nu}(\ggG(\A))\cup \SSS_{\nu}(\ggG(\A))\longrightarrow\C$ ($\nu$ sufficiently large as in Lemma \ref{fundamental_convergence}) defined by
\[
 j_{\NNN}^T(\Phi)=\int_{A_GG(\Q)\backslash G(\A)}
 F(x, T)\sum\limits_{\gamma\in \NNN}\Phi(\Ad x^{-1}\gamma )dx.
\]
Since $0\leq F(x, T)\leq \hat{\tau}_0^G(T-H_0(t))$ for every $x=utk\in U_0(\A)T_0(\A)\cpt$, this sum-integral converges absolutely for the same reasons as Lemma \ref{abs_conv_of_aux_distr} holds. 

\begin{proposition}\label{prop_on_sep_nilp_distr}
There exists $\nu>0$ depending only on $n$ such that the following holds.
For every nilpotent orbit $\NNN\subseteq \nnn$ there is a distribution $J_{\NNN}^T:\SSS^{\nu}(\ggG(\A))\cup\SSS_{\nu}(\ggG(\A))\longrightarrow\C$ such that
\[
J_{\nnn}^T(\Phi)=\sum_{\NNN\subseteq\nnn}J_{\NNN}^T(\Phi).
\]
Moreover, $J_{\NNN}^T(\Phi)$ is a polynomial in $T$ of degree at most $\dim\aaa$, and there exist $c>0$ and a seminorm $\mu:\SSS^{\nu}(\ggG(\A))\cup\SSS_{\nu}(\ggG(\A))\longrightarrow\C$ such that
\[
\left|J_{\NNN}^T(\Phi)-j^T_{\NNN}(\Phi)\right|
=
\left|J_{\NNN}^T(\Phi)
-\int_{G(\Q)\backslash G(\A)^1}F(x, T)\sum_{\gamma\in\NNN}\Phi(\Ad x^{-1} \gamma)dx\right|
\leq \mu(\Phi)e^{-c\|T\|}
\]
for all sufficiently regular $T\in\aaa^+$ with $d(T)\geq \delta\|T\|$.
\end{proposition}

\begin{proof}
The assertion is the analogue to \cite[Theorem 4.2]{Ar85} where it is stated for smooth compactly supported functions on the group $G(\A)$.
Large parts of the proof of \cite[Theorem 4.2]{Ar85} carry over to our situation, we have, however, to take into account that our test function is not compactly supported anymore.
We define an auxiliary function similar as in \cite{Ar85}: Let $\NNN\subseteq \nnn$ be a nilpotent orbit and let $\eps>0$ be given. Let $q_1, \ldots, q_r:\ggG(\Q)\longrightarrow\Q$ be polynomials such that $\overline{\NNN}=\{X\in\ggG(\Q)\mid q_1(X)=\ldots=q_r(X)=0\}$. 
Let $\rho_{\infty}:\R\longrightarrow\R$ be a non-negative smooth function with support in $[-1, 1]$  which identically equals $1$ on $[-1/2, 1/2]$ and such that $0\leq \rho_{\infty}\leq1$. Define
\[
 \Phi^{\eps}_{\NNN}(X)=\Phi(X)\rho_{\infty}(\eps^{-1}|q_1(X)|_{\infty})\cdot\ldots\cdot \rho_{\infty}(\eps^{-1} |q_r(X)|_{\infty})
\]
so that $\Phi^{\eps}_{\NNN}=\Phi$ in a neighbourhood of $\overline{\NNN}$.
It follows from the proof of \cite[Theorem 4.2]{Ar85} that it suffices to show the analogue of \cite[Lemma 4.1]{Ar85}, namely that
\begin{equation}\label{key_lemma_eq}
 \int_{G(\Q)\backslash G(\A)^1} F(x, T) \sum_{X\in \nnn\minus\overline{\NNN}} |\Phi^{\eps}_{\NNN} (\Ad x^{-1} X)|dx
\leq \mu(\Phi)\eps^a (1+\|T\|)^{\dim \aaa}
\end{equation}
for a suitable seminorm $\mu$, and a suitable number $a>0$. Hence, using \eqref{siegel_est}, we need  to bound (after integrating $\Phi$ over a compact subset) 
\[
 \int_{A_0^G(T_1)} \delta_0(a)^{-1} F(a, T) \sum_{X\in \nnn\minus\overline{\NNN}} |\Phi^{\eps}_{\NNN} (\Ad a^{-1} X)|da.
\]
It suffices to take the sum over $X\in\ggG(\Q)\minus\overline{\NNN}$. Moreover, since $\Phi$ is compactly supported at the non-archimedean places, there exists $N>0$ such that we can take the sum instead over points with entries in $\frac{1}{N}\Z$. 
For $R>0$ define a function $\Phi_R(X):=\Phi(X) \rho_{\infty}(R^{-1} \|X\|)$ so that the support of $\Phi_R$ is compact and contained in $\{X\in\ggG(\A)\mid \|X\|\leq R\}$, and $\Phi_R(X)=\Phi(X)$ if $\|X\|\leq R/2$. Moreover, if $D\in\UUU(\ggG)$ denotes an element of degree $k\leq \nu $, then there exists a constant $c_D>0$ depending only on $D$ and $\rho_{\infty}$ such that 
\[
 \|D\Phi_R\|_{L^1(\ggG(\A))} \leq c_D\sum_{Y\in\BBB_{\ggG, \nu}} \|Y \Phi\|_{L^1(\ggG(\A))}=: c_D \mu_k(\Phi),
\]
and this last expression is  a seminorm on $\SSS^{\nu}(\ggG(\A))\cup\SSS_{\nu}(\ggG(\A))$.

It follows from the proof of \cite[Lemma 4.1]{Ar85} that there exist constants $r,a_0, k_0, c>0$ depending only on $n$ such that
\[
  \int_{G(\Q)\backslash G(\A)^1} F(x, T) \sum_{X\in \ggG(\Q)\minus\overline{\NNN}} |(\Phi_R)^{\eps}_{\NNN} (\Ad x^{-1} X)|dx
\leq  c R^{a_0}\mu_{k_0}(\Phi)\eps^r (1+\|T\|)^{\dim \aaa}
\]
for every $R\geq1$, since the support of $\Phi_R$ is compact and contained in the ball of radius $R$ around $0\in\ggG(\A)$.
In particular, if $1\leq R_1\leq R_2$, we get
\[
  \int_{G(\Q)\backslash G(\A)^1} F(x, T) \sum_{X\in \ggG(\Q)\minus\overline{\NNN}} 
|\big(\Phi_{R_1}-\Phi_{R_2}\big)^{\eps}_{\NNN} (\Ad x^{-1} X)|dx
\leq  c R_2^{a_0}\mu_{k_0}^{R_1}(\Phi)\eps^r (1+\|T\|)^{\dim \aaa},
\]
where
\[
 \mu_{k_0}^{R_1}(\Phi)
:= \sum_{Y\in\BBB_k} \int_{\ggG(\A)\minus B_{R_1}}|(Y\Phi)(X)| dX
\]
for $B_{R_1}:=\{X\in\ggG(\A)\mid \|X\|< R_1\}$.
Let $N\in\Z_{>0}$ and suppose $\nu>N$. Since $\Phi\in\SSS^{\nu}(\ggG(\A))\cup \SSS_{\nu}(\ggG(\A))$, there exists $C_N>0$ and $k_N\geq k_0$ such that
\[
 \mu_{k_0}^{R_1}(\Phi)
\leq C_N R_1^{-N} \mu_{k_N}(\Phi).
\]
Fix $N>a_0$. By definition $|\Phi-\Phi_{2^i}|\leq 2\sum_{j\geq i-1} |\Phi_{2^{j+2}}-\Phi_{2^j}|$ so that for every $i>0$, we get
\begin{multline*}
  \int_{G(\Q)\backslash G(\A)^1} F(x, T) \sum_{X\in \ggG(\Q)\minus\overline{\NNN}} 
|\big(\Phi-\Phi_{2^i}\big)^{\eps}_{\NNN, v} (\Ad x^{-1} X)|dx\\
\leq  c_N \sum_{j\geq i-1} 2^{(a_0-N )j} \mu_{k_N}(\Phi)\eps^r (1+\|T\|)^{\dim \aaa} 
= c_N' 2^{(a_0-N)(i-1)}  \mu_{k_N}(\Phi)\eps^r (1+\|T\|)^{\dim \aaa}
\end{multline*}
for $c_N, c_N'>0$ suitable constants. Hence if we fix an arbitrary integer $i>0$, we get
\begin{multline*}
  \int_{G(\Q)\backslash G(\A)^1} F(x, T) \sum_{X\in \ggG(\Q)\minus\overline{\NNN}} 
|\Phi^{\eps}_{\NNN, v} (\Ad x^{-1} X)|dx\\
\leq c \big(2^{(a_0-N)(i-1)}  \mu_{k_N}(\Phi)+ 2^{ia_0} \mu_{k_0}(\Phi)\big)\eps^r (1+\|T\|)^{\dim\aaa}
\end{multline*}
for a suitable constant $c>0$ proving the inequality \eqref{key_lemma_eq}. Taking $N=a_0+1$ (which only depends on $n$) and $\nu>a_0+1$, also proves the assertion about the existence of $\nu$.
 \end{proof}

The last proposition implies that to understand the nilpotent distribution $J_{\nnn}^T$ we need to study the distributions $J_{\NNN}^T$ or $j_{\NNN }^T$. 
However, for our purposes the distributions  $\tilde{j}_{\NNN}^T$ are better suited because of the homogenity property from Lemma \ref{lemma_chang_central_var}, and it will in fact suffice to understand them:
\begin{proposition}\label{main_approx_res}
Let $\nu>0$ be as in Lemma \ref{fundamental_convergence}. There exists a continuous seminorm $\mu$ on $\SSS^{\nu}(\ggG(\A))\cup\SSS_{\nu}(\ggG(\A))$, and  $\eps>0$ such that for all $\Phi\in\SSS^{\nu}(\ggG(\A))\cup\SSS_{\nu}(\ggG(\A))$,
\begin{equation}\label{asymp2}
 \bigg|j_{\NNN}^T(\Phi)-\tilde{j}_{\NNN}^T(\Phi)\bigg|
\leq \mu(\Phi) e^{-\eps\|T\|}
\end{equation}
for all sufficiently regular $T\in\aaa^+$ with $d(T)\geq \delta\|T\|$.
\end{proposition}
We postpone the proof of this proposition to Appendix \ref{appendix}.
\begin{cor}
Let $I\subseteq \R_{>0}$ be a compact interval and let $\nu$ be as before. Then:
 \begin{enumerate}[label=(\roman{*})]
\item There exists a continuous seminorm $\mu$ on $\SSS^{\nu}(\ggG(\A))\cup\SSS_{\nu}(\ggG(\A))$ and a constant $\eps>0$ such that for all $\Phi\in\SSS^{\nu}(\ggG(\A))\cup\SSS_{\nu}(\ggG(\A))$ and  all sufficiently regular $T\in\aaa^+$ with $d(T)\geq \delta\|T\|$ we have
\begin{equation*}
 \left|J_{\NNN}^T(\Phi)-\tilde{j}^{T}_{\NNN}(\Phi)\right|
\leq \mu(\Phi) e^{-\eps\|T\|}
\end{equation*}
for every nilpotent orbit $\NNN\subseteq \nnn$.

\item For every $T\in\aaa^+$ such that $T$ and $T+\frac{\log\lambda}{2}H_{X_0}$ are sufficiently regular for all $\lambda\in I$, we have
\begin{equation}\label{identity_nilp_distr_central_var}
 J_{\NNN}^T(\Phi_{\lambda})=\lambda^{-\delta(\NNN)}   
J^{T+\frac{\log\lambda}{2} H_{X_0}}_{\NNN}(\Phi)
\end{equation}
for every nilpotent orbit $\NNN\subseteq \nnn$, all $\lambda\in I$, and $\Phi\in \SSS^{\nu}(\ggG(\A))\cup\SSS_{\nu}(\ggG(\A))$.

\item As a polynomial, $J_{\NNN}^T(\Phi_{\lambda})$ can be defined at every point $T\in\aaa$, and \eqref{identity_nilp_distr_central_var} holds for all $T\in\aaa$ and $\lambda\in\R_{>0}$.
\end{enumerate}
\end{cor}
\begin{proof}
\begin{enumerate}[label=(\roman{*})]
\item  This is a direct consequence of Proposition \ref{main_approx_res} and Proposition \ref{prop_on_sep_nilp_distr}.

\item By the first part  we have for every $\Phi\in\SSS^{\nu}(\ggG(\A))\cup\SSS_{\nu}(\ggG(\A))$ and $\lambda\in I$ we have
\[
| J_{\NNN}^T(\Phi_{\lambda})-\tilde{j}_{\NNN}^{T}(\Phi_{\lambda})|
\leq
\mu(\Phi_{\lambda}) e^{-\eps \|T\|}
\]
for every sufficiently regular $T\in\aaa^+$ with $d(T)\geq \delta \|T\|$. Since $I$ is compact, and $\mu(\Phi_{\lambda})$ varies continuously in $\lambda$,  $C_{I}:=\max_{\lambda\in I}\mu(\Phi_{\lambda})$ exists and is finite.
Similarly, we have
\[
 |\lambda^{-\delta(\NNN)}J_{\NNN}^{T+\frac{\log\lambda}{2} H}(\Phi)
- \lambda^{-\delta(\NNN)}\tilde{j}_{\NNN}^{T+\frac{\log\lambda}{2} H}(\Phi)|
\leq
\lambda^{-\delta(\NNN)}\mu(\Phi) e^{-\eps \|T+\frac{\log\lambda}{2} H\|}
\]
for all $T\in \aaa$ with $d(T+\frac{\log\lambda}{2} H)\geq \delta \|T+\frac{\log\lambda}{2} H\|$ and $T+\frac{\log\lambda}{2}H$ sufficiently regular.
As 
 \[
\tilde{j}_{\NNN}^{T}(\Phi_{\lambda})
=\lambda^{-\delta(\NNN)}\tilde{j}_{\NNN}^{T+\frac{\log\lambda}{2} H}(\Phi),
\]
we therefore get with $\lambda_I:=\min_{\lambda\in I}\lambda$ that
\begin{equation}\label{nilpotent_distr_est}
 | J_{\NNN}^T(\Phi_{\lambda})-\lambda^{-\delta(\NNN)}J_{\NNN}^{T+\frac{\log\lambda}{2} H}(\Phi)|
\leq\max\{C_I, \lambda_I^{-\delta(\NNN)}\mu(\Phi)\}e^{-\eps \|T+\frac{\log\lambda}{2} H\|}
\end{equation}
for all $T\in \aaa$ with $d(T+\frac{\log\lambda}{2} H)\geq \delta \|T+\frac{\log\lambda}{2} H\|$ and $d(T)\geq \delta \|T\|$ if both $T$ as well as $T+\frac{\log\lambda}{2} H$ are sufficiently regular. 

The set of $T\in\aaa^+$ satisfying both inequalities is an open cone in $\aaa^+$ so that $J_{\NNN}^T(\Phi_{\lambda})$ - being a polnyomial in $T$ - is uniquely determined by this estimate.
Thus the left hand side of \eqref{nilpotent_distr_est} must identically vanish and the second part of the corollary follows.

\item As a polynomial, $J_{\NNN}^T(\Phi_{\lambda})$ can be defined at every point $T\in\aaa$ with \eqref{identity_nilp_distr_central_var} holding for all $\lambda\in I$. Since $I\subseteq \R_{>0}$ is arbitrary, \eqref{identity_nilp_distr_central_var} holds for all $\lambda\in \R_{>0}$.
\end{enumerate}
\end{proof}

The next two corollaries are obvious from our previous results so that we omit their proofs.

\begin{cor}\label{nilpotent_distributions_basic_properties}
Let $T\in\aaa$ be arbitrary, and let $\NNN\subseteq \nnn$ be a nilpotent orbit.
Let $\nu>0$ be as before, and let $\Phi\in \SSS^{\nu}(\ggG(\A))$.
\begin{enumerate}[label=(\roman{*})]
\item The function $J_{\NNN}^{T, -}(s, \Phi)$ defined by
\[
 J_{\NNN}^{T, -}(s, \Phi)
=\int_{0}^1 \lambda^{\sqrt{D}(s+\frac{\sqrt{D}-1}{2})}J_{\NNN}^T(\Phi_{\lambda})d^{\times}\lambda
\]
converges absolutely and locally uniformly for $\Re s>\frac{1-\sqrt{D}}{2}+\frac{1}{\sqrt{D}}\delta(\NNN)$. It defines a holomorphic function in this half plane and has a meromorphic continuation to all $s\in\C$ with only pole at 
$
\frac{1-\sqrt{D}}{2}+\frac{1}{\sqrt{D}}\delta(\NNN)
=
\frac{1-\sqrt{D}}{2}+\frac{\dim\NNN}{2\sqrt{D}},
$
which is of order at most $\dim\aaa$.

\item The function $J_{\NNN}^{T, +}(1-s, \Phi)$ defined by 
\[
 J_{\NNN}^{T, +}(1-s, \hat{\Phi})
=\int_{0}^1 \lambda^{\sqrt{D}(s+\frac{\sqrt{D}-1}{2})}\lambda^{-D}J_{\NNN}^T(\hat{\Phi}_{\lambda^{-1}})d^{\times}\lambda
\]
converges absolutely and locally uniformly for $\Re s>\frac{\sqrt{D}+1}{2}-\frac{1}{\sqrt{D}}\delta(\NNN)$. It defines a holomorphic function in this half plane and has a meromorphic continuation to all $s\in\C$ with only pole at 
$
\frac{\sqrt{D}+1}{2}-\frac{1}{\sqrt{D}}\delta(\NNN)
=\frac{\sqrt{D}+1}{2}-\frac{\dim\NNN}{2\sqrt{D}},
$
which is of order at most $\dim\aaa$.
\end{enumerate}
\end{cor}

\begin{cor}\label{properties_nilpotent_distr}
Let $\Phi\in \SSS^{\nu}(\ggG(\A))$ and put
\[
I_{\NNN}^T(s, \Phi)=J_{\NNN}^{T, +}(1-s, \hat{\Phi})-J_{\NNN}^{T, -}(s, \Phi).
\]
Then for every $T\in \aaa$, $I_{\NNN}^T(s, \Phi)$ has a meromorphic continuation to all $s\in \C$ and satisfies the functional equation 
 \[
 I_{\NNN}^{T}(s, \Phi)=I_{\NNN}^{T}(1-s, \hat{\Phi}).
\]
Its only poles are at 
\[
 \frac{1-\sqrt{D}}{2}+\frac{\dim\NNN}{2\sqrt{D}}
\;\;
\text{ and }
\frac{\sqrt{D}+1}{2}-\frac{\dim\NNN}{2\sqrt{D}},
\]
which are both of order at most $\dim\aaa$.
\end{cor}

Our main theorem is now an easy consequence of the previous results.
\begin{theorem}\label{main_theorem}
Let $G=\GL_n$ or $G=\SL_n$ with $n\leq 3$, and let $R>n$ be given. Then there exists $\nu<\infty$ such that for every $\Phi\in\SSS^{\nu}(\ggG(\A))$ and  $T\in\aaa$ the following holds.
\begin{enumerate}[label=(\roman{*})]
\item
$\Xi^T(s,\Phi)$ is holomorphic for all $s\in\C$ with $\Re s>\frac{\sqrt{D}+1}{2}$. It equals a polynomial in $T$ of degree at most $\dim\aaa=n-1$.

\item 
$\Xi^T(s, \Phi)$ has a meromorphic continuation to all $s\in\C$ with $ \Re s>-R$, and satisfies for such $s$ the functional equation
\[
\Xi^T(s, \Phi)=\Xi^T(1-s, \hat{\Phi}).
\]

\item
The poles of $\Xi^T(s, \Phi)$ in $\Re s>-R$ are parametrised by the nilpotent orbits $\NNN\subseteq \nnn$. More precisely, its poles occur exactly at the points 
\[
s_{\NNN}^-=\frac{1-\sqrt{D}}{2}+\frac{\dim\NNN}{2\sqrt{D}}
~~~~\text{and }~~~~
s_{\NNN}^+=\frac{1+\sqrt{D}}{2}-\frac{\dim\NNN}{2\sqrt{D}}
\]
and are of order at most $\dim\aaa=n-1$. 
In particular, the furthermost right and furthermost left pole in this region are both simple, correspond to $\NNN=0$, and are located at the points
$s_{0}^+=\frac{1+\sqrt{D}}{2}$
and 
$s_{0}^-=\frac{1-\sqrt{D}}{2}$, respectively.
 The residues at these poles are given by
\begin{align*}
\res_{s=s_{0}^-}\Xi^T(s, \Phi)
&=\vol(A_GG(\Q)\backslash G(\A))\Phi(0),\text{   and }\\
\res_{s=s_{0}^+}\Xi^T(s, \Phi)
&=\vol(A_GG(\Q)\backslash G(\A))\hat{\Phi}(0).
\end{align*}
\end{enumerate}
\end{theorem}
\begin{rem}
 If we take $\nu=\infty$ and accordingly $\Phi\in \SSS(\ggG(\A))$, then $\Xi^T(s,\Phi)$ continues meromorphically to all of $\C$.
\end{rem}

\begin{proof}
We only prove the theorem for $\nu=\infty$. The other case works similar by using the analogue results from the previous sections for $\nu<\infty$ instead and we omit the details.
For every $\lambda\in (0, \infty)$ and every $T\in \aaa$ Chaudouard's trace formula gives 
\[
 J_{*}^T(\Phi_{\lambda})
=\lambda^{-D}J_{*}^T(\hat{\Phi}_{\lambda^{-1}})
+\lambda^{-D}J_{\nnn}^T(\hat{\Phi}_{\lambda^{-1}})
-J_{\nnn}^T(\Phi_{\lambda}).
\]
Define
\[
 I_{\nnn}^T(s, \Phi)
= \int_0^1 \lambda^{\sqrt{D}(s+\frac{\sqrt{D}-1}{2})}\left(\lambda^{-D}J_{\nnn}^T(\hat{\Phi}_{\lambda^{-1}})
-J_{\nnn}^T(\Phi_{\lambda}) \right) d^{\times}\lambda
\]
which converges for $\Re s>\frac{\sqrt{D}+1}{2}$ and defines a holomorphic function there.
By Corollary \ref{nilpotent_distributions_basic_properties}, we may split $I_{\nnn}^T(s, \Phi)$ into a sum
$\sum_{\NNN}I_{\NNN}^T(s, \Phi)$.
Hence for $s\in \C$ with $\Re s>\frac{\sqrt{D}+1}{2}$ we get
\[
\Xi^T(s, \Phi)
=\Xi^{T, +}(s, \Phi) + \Xi^{T, +}(1-s, \hat{\Phi})
+ I_{\nnn}^T(s, \Phi).
\]
By Theorem \ref{convergence_distributions} and Corollary \ref{properties_nilpotent_distr}, all assertions except for the location of first and last pole follow. 

By Corollary \ref{properties_nilpotent_distr} the poles of $I_{0}^T(s,\Phi)$ are at $s_0^{\pm}$ (with residues as asserted). We only need to show that for $\NNN\neq0$ the poles of $I_{\NNN}^T(s,\Phi)$ are contained in the open intervall $(s_0^-, s_0^+)$, but this follows from the explicit expression of $s_{\NNN}^{\pm}$ in terms of the dimension of $\NNN$.
\end{proof}

\section{Connections to Arthur's trace formula and Shintani zeta function}\label{connection_arthur_tf}
\subsection{The main part of the zeta function}
Let $n\geq 2$ be arbitrary. Recall that $X\in\ggG(\Q)_{\text{ss}}$ (resp., $\gamma\in G(\Q)_{\text{ss}}$) is called  \emph{regular} if its eigenvalues (over some algebraic closure of $\Q$) are pairwise different, and that $X\in\ggG(\Q)_{\text{ss}}$ (resp.\ $\gamma\in G(\Q)_{\text{ss}}$) is called \emph{regular elliptic} if $X$ (resp.\ $\gamma$) is regular and if the commutator subgroup $G_{X}$ (resp., $G_{\gamma}$) is not contained in any proper parabolic subgroup of $G$. 
Note that an element $X\in\ggG(\Q)$ (resp.\ $\gamma\in G(\Q)$) is regular elliptic if and only if its eigenvalues are pairwise distinct and some (and hence any) of them generates an $n$-dimensional field extension over $\Q$.

Let $\OOO_{\text{reg}}$ denote the set of equivalence classes attached to the orbits of regular elements in $\ggG(\Q)$, and $\OOO_{\text{er}}$ the set of classes attached to orbits of elliptic regular elements in $\ggG(\Q)$. Further, write $\OOO_{\text{reg}}'=\OOO_{\text{reg}}\minus\OOO_{\text{er}}$. 
We define the ``main part'' of $\Xi^T$ as
\[
 \Xi_{\text{main}}^{T}(s, \Phi)=\int_{0}^{\infty} \lambda^{\sqrt{D}(s+\frac{\sqrt{D}-1}{2})} \sum_{\ooo\in \OOO_{\text{er}}} J_{\ooo}^T(\Phi_{\lambda}) d^{\times}\lambda.
\]
By Theorem \ref{convergence_distributions} this defines a holomorphic function for $\Re s>\frac{\sqrt{D}+1}{2}$. In the next section we will see that, at least for $G=\GL_n$ and  $n\leq 3$, this function is indeed the main part of $\Xi^T(s, \Phi)$ in the sense that it is responsible for the rightmost pole.
\begin{rem}
 If $\ooo\in\OOO_{\text{er}}$, then $J_{\ooo}^T(\Phi)$ is independent of $T$ and in fact equals an orbital integral: If $\ooo$ corresponds to the orbit of $X\in\ggG(\Q)_{\text{er}}$, then the centraliser $G_X$ of $X$ in $G$ is reductive and we may fix a Haar measure on $G_X(\A)$. Taking the quotient measure on $G_X(\A)\backslash G(\A)$, we then get
\[
 J_{\ooo}(\Phi)= J_{\ooo}^T(\Phi)=
\vol(G_X(\Q)\backslash G_X(\A)^1) \int_{G_X(\A)\backslash G(\A)} \Phi(\Ad g^{-1} X) dg
\]
(cf.\ \cite[\S 5]{Ch02}).
In particular, the main part of the zeta function is independent of $T$ and we also write $\Xi_{\text{main}}(s,\Phi)=\Xi^T_{\text{main}}(s,\Phi)$.
\end{rem}

\subsection{Connection to Arthur's trace formula}
Let $G=\GL_n$, and let $\OOO^G$ denote the set of geometric equivalence classes in the group $G(\Q)$ as defined by Arthur (usually denoted by $\OOO$). To distinguish them from the equivalence classes we defined here on the set $\ggG(\Q)$,  we shall write $\OOO^{\ggG}=\OOO$ if necessary. Let $\OOO_{\text{er}}^{\ggG}$ (resp.\ $\OOO_{\text{er}}^G$) denote the set of equivalence classes attached to orbits of elliptic regular elements $X\in \ggG(\Q)$ (resp.\ $\gamma\in G(\Q)$).
We have a canonical inclusion $G=\GL_n\hookrightarrow \ggG$ of $G$-varieties with $G(\Q)_{\text{ss}}\hookrightarrow \ggG(\Q)_{\text{ss}}$. This is of course a special feature of $\GL_n$ and does not apply to a general reductive group.
If $\gamma_s\in G(\Q)_{\text{ss}}$ and $\ooo^G\in \OOO^G$ is the equivalence class attached to the conjugacy class of $\gamma_s$, it is straightforward that $\ooo^G\in \OOO^{\ggG}$ is also the equivalence class attached to the orbit of $\gamma_s$ viewed as an element in $\ggG(\Q)_{\text{ss}}$.
This gives an inclusion  $\OOO^G\hookrightarrow \OOO^{\ggG}$ and we view $\OOO^G$ as a subset of $\OOO^{\ggG}$. 
Note that $\OOO_{\text{er}}^{\ggG}= \OOO^G_{\text{er}}$.

Arthur's trace formula is an identity of the so-called geometric and spectral distribution on a space of suitable test functions on $G(\A)^1$. The geometric side allows a coarse geometric expansion given by $J_{\text{geom}}^{G,T}(f_s)=\sum_{\ooo\in\OOO^G} J_{\ooo}^{G,T}(f)$ for $T\in\aaa$ and $J_{\ooo}^{G,T}$ a certain distribution attached to $\ooo$, cf.\ \cite{Ar05} (usually $J_{\ooo}^{G,T}$ is denoted by $J_{\ooo}^T$). If $\Phi\in\SSS(\ggG(\A))$ is such that the restriction $\Phi_{\rvert G(\A)^1}$ is admissible as a test function for Arthur's trace formula, then $J_{\ooo}^T(\Phi)=J_{\ooo}^{G,T}(\Phi_{\rvert G(\A)^1})$.

Now suppose that $n\leq 3$ and $\Phi\in \SSS(\ggG(\A))$. For $s\in\C$ with $\Re s>\frac{n+1}{2}$ we define a smooth function $f_s:G(\A)\longrightarrow\C$ by
\[
 f_s(g)=\int_0^{ \infty} \lambda^{n(s+\frac{n-1}{2})} \Phi(\lambda g) d\lambda.
\]
By results of \cite{FiLa09,diss}, $f_s$ is an admissible test function for Arthur's trace formula for $\GL_n$, $n\leq 3$, and for $\Re s>\frac{n+1}{2}$, 
\[
 \Xi_{\text{main}}(s,\Phi)=\sum_{\ooo\in\OOO_{\text{er}}^G} J_{\ooo}^{G,T}(f_s)
\]
equals the regular elliptic part of Arthur's trace formula. 
One could try to use the geometric side $J_{\text{geom}}^{G,T}(f_s)$ as a regularisation for $\Xi^T(s,\Phi)$ and Arthur's trace formula as an replacement for the Poisson summation formula.
However, this leads to problems already for $n=3$: The functions arising from the continuous spectrum on the spectral side might have no meromorphic continuation to all of $\C$. It is quite possible that $J_{\text{geom}}^{G, T}(f_s)$ (and also $\Xi_{\text{main}}(s,\Phi)$) can not be meromorphically continued to all of $\C$, cf.\ \cite{diss}. This is one reason why it seems more natural to study $\Xi_{\text{main}}(s,\Phi)$ in the context of Chaudouard's trace formula.

\subsection{Connection to the Shintani zeta function in the quadratic case}
The purpose of this section is to explain the connection between the Shintani zeta function (in its classical formulation by Shintani \cite{Sh75}) and the main part of the zeta function $\Xi_{\text{main}}(s,\Phi)$ for $\GL_2$, or, equivalently, the regular elliptic part of Arthur's trace formula for $\GL_2$, cf.\ also \cite{La02}.
Let $G=\GL_2$ and let $\Phi(x)=e^{-\pi\tr x_{\infty}^t x_{\infty}} \One_{\ggG(\hat{\Z})}(x_f)$ for $x=x_{\infty} x_f\in G(\A)$. Here $\One_{\ggG(\hat{\Z})}=\prod_{p<\infty} \One_{\ggG(\Z_p)}$ denotes the characteristic function of $\ggG(\hat{\Z})\subseteq \ggG(\A_f)$.
If $\gamma\in G(\Q)_{\text{er}}$ we denote by $\Delta(\gamma)$ its discriminant.
 Let $E$ be a quadratic number field, $d_E$ its discriminant, and $D_E$ the squarefree part of $d_E$ so that $E=\Q(\sqrt{D_E})$. 
The ring of integers of $E$ equals $\OOO_E=\Z[\theta]$ for a suitable $\theta\in \OOO_E$.
We have a two-sheeted surjective map from pairs $(E,\xi)$ of quadratic number fields $E$ and $\xi\in E\minus\Q$ to the conjugacy classes in $G(\Q)_{\text{er}}$ by mapping $(E,\xi)$ to the conjugacy class of the companion matrix of the characteristic polynomial of $\xi$.
With respect to the basis $\{1,\theta\}$, we get a surjection from $\OOO_E\minus\Z$ onto the set of conjugacy classes in  $G(\Q)_{\text{er}}$ whose characteristic polynomials have integer coefficients and whose eigenvalues generate $E$.
This surjection sends $a+b\theta\in \OOO_E$ to the conjugacy class of $\gamma\in G(\Q)_{\text{er}}$ having eigenvalues $\gamma_1=a+b\theta, \gamma_2=a+b\bar{\theta}$ with $\bar{\theta}$ denoting the image of $\theta $ under the action of the non-trivial element of the Galois group of $E/\Q$.  For such $\gamma$, $\Z[\gamma_1]=\Z[b\theta]$ and $\Delta(\gamma)=b^2D_E$.
If $p$ is a prime, one can compute the local orbital integral at $p$ to be
\[
\int_{G_{\gamma}(\Q_p)\backslash G(\Q_p)}\One_{\ggG(\Z_p)}(g^{-1}\gamma g)dg
=p^{\kappa}(1+(1-\left(\frac{D_E}{p}\right))\frac{1-p^{-\kappa}}{p-1}),
\]
where $\kappa=\frac{1}{2}(\val_p(\Delta(\gamma))-\val_p(D_E))=\val_p(b)$.
Here we choose measures on $G_{\gamma}(\Q_p)$ as follows: If the eigenvalues of $\gamma$ generate a quadratic field extension $E_p$ over $\Q_p$, we normalise the measures on $E_p$ and $E_p^{\times}$ such that $\vol(\OOO_{E_p})=1=\vol(\OOO_{E_p}^{\times})$ for $\OOO_{E_p}\subseteq E_p$ the ring of integers. Using the ismorphism $G_{\gamma}(\Q_p)\simeq E_p^{\times}$  (with respect to the basis $\{1,\theta\}$) we fix a measure on $G_{\gamma}(\Q_p)$. We then take the unnormalised product measure on $G_{\gamma}(\A)$ and  $\A_E^{\times}.$
If $E$ is totally real, we have
\[
\int_0^{\infty}\lambda^{2s+1} \int_{G_{\gamma}(\R)\backslash G(\R)}\Phi_{\infty}(\lambda g^{-1}\gamma g)d^{\times} \lambda~ dg
=\frac{\Gamma(s)}{2\pi^{s}\sqrt{\Delta(\gamma)}}(\tr\gamma^2)^{-s}.
\]
 The sum over conjugacy classes of $\gamma$ generating totally real quadratic extensions is
\[
\sum_{\substack{[\gamma]\subseteq G(\Q)_{\text{er}}: \\ \Q(\gamma)\text{ tot. real} }}
\vol(G_{\gamma}(\Q)\backslash G_{\gamma}(\A)^1) 
\int_{0}^{\infty} \lambda^{2s+1} \int_{G_{\gamma}(\A)\backslash G(\A)}\Phi(\lambda g^{-1}\gamma g) d^{\times}\lambda ~dg
\]
Since 
$\vol(G_{\gamma}(\Q)\backslash G_{\gamma}(\A)^1)
=D_{E}^{\frac{1}{2}}\res_{s=1}\zeta_{E}(s)=2h_E\log\eps_E$
 with $\eps_{E}$ a positive fundamental unit in $\OOO_{E}$, this equals
\[
\sum_{\substack{E/\Q, [E:\Q]=2:\\ \text{ tot. real} }} \frac{h_E\log\eps_E\Gamma(s)}{2\pi^s\sqrt{D_E}}
\sum_{a+b\theta\in \OOO_E\minus\Z}\mathcal{N}_{E/\Q}(a+b\theta)^{-s}\prod_{p|b}(1+(1-\left(\frac{D_E}{p}\right))\frac{1-|b|_p}{p-1}),
\]
and an application of Poisson summation to the sum over $a$ yields as the main term
\[
\sum_{\substack{ E/\Q, [E:\Q]=2:  \\ \text{ tot. real}}}\frac{h_E\log\eps_E\Gamma(s-\frac{1}{2})}{\pi^{s-\frac{1}{2}}}D_E^{-s}
\sum_{b\in\N}b^{-2s+1}\prod_{p|b}(1+(1-\left(\frac{D_E}{p}\right))\frac{1-|b|_p}{p-1}).
\]
The sum over $b$ can be computed to equal
$\frac{\zeta(2s-1)\zeta(2s)}{\zeta_E(2s-1)}$
so that  by \cite[Theorem 0.2]{Da93} the above equals
\[
\sum_{\substack{E/\Q, [E:\Q]=2 :\\  \text{ tot. real}}}\frac{h_E\log\eps_E\Gamma(s-\frac{1}{2})}{\pi^{s-\frac{1}{2}}}D_E^{-s}\frac{\zeta(2s-1)\zeta(2s)}{\zeta_E(2s-1)}
=\Gamma(s-\frac{1}{2})\pi^{-s+\frac{1}{2}}Z_{\text{Shin},+}(s),
\]
where
$Z_{\text{Shin},+}=\sum\limits_{d=1}^{\infty}h_d\log\eps_d d^{-s}$
is the Shintani zeta function associated to the positive definite binary forms introduced by Shintani in \cite{Sh75}.  Here $h_d$ is the class number of positive definite binary quadratic forms of discriminant $d$,  $h_d\log\eps_d$ is defined to be $0$ if $d$ is a square, and otherwise $\eps_d=t+u\sqrt{d}$ is the minimal solution of $(t,u)\in \N^2$ of $t^2-u^2d=4$.
For the imaginary quadratic number fields we obtain similarly
\begin{multline*}
\sum_{\substack{\gamma\in G(\Q)_{\text{er}}: \\ \Q(\gamma)\text{ complex} }}
\vol(G_{\gamma}(\Q)\backslash G_{\gamma}(\A)^1) 
\int_{0}^{\infty} \lambda^{2s+1} \int_{G_{\gamma}(\A)\backslash G(\A)}\Phi(\lambda g^{-1}\gamma g) d^{\times}\lambda ~dg\\
=8\sqrt{2}\pi^{-s+1}\Gamma(s)I(s)Z_{\text{Shin}, -}(s)+\text{entire fct.}
\end{multline*}
Here $I(s)=\int_{1}^{\infty}(-\frac{1}{2}+\tau^2)^{-s}d\tau$ is a holomorphic function for $\Re s>1$, and 
$Z_{\text{Shin}, -}(s)=\sum\limits_{-d=1}^{\infty}\frac{h_d}{w_d}(-d)^{-s}$
the Shintani zeta function associated with indefinite binary quadratic forms. Here again $h_d$ is the class number of indefinite quadratic forms of discriminant $d$, and $w_d$ is the order of $\OOO_{\Q(\sqrt{d})}^{\times}$.

Putting both parts together, we see that the main zeta function $\Xi_{\text{main}}(s, \Phi)$ now equals up to an entire function
\[
\Gamma(s-\frac{1}{2})\pi^{-s+\frac{1}{2}}Z_{\text{Shin},+}(s)
+ 8\sqrt{2}\pi^{-s+1}\Gamma(s)I(s)Z_{\text{Shin}, -}(s).
\]
By varying the test function $\Phi_{\infty}$ at the archimedean place it should be possible to filter out only the part belonging to the positive definite or to the indefinite forms.

\section{Poles of $\Xi_{\text{main}}(s, \Phi)$ for $G=\GL_n$, $n\leq 3$.}\label{section_poles}
In this section let $G=\GL_n$ with $n\leq 3$. We assume throughout that $\nu>0$ is sufficiently large as in Lemma \ref{fundamental_convergence}.
The purpose of this section is to show that $\Xi_{\text{main}}(s,\Phi)$ is indeed the main part of $\Xi^T(s,\Phi)$ in the sense that it is responsible for the furthermost right pole of $\Xi(s,\Phi)$.

We group the equivalence classes in $\OOO_*$ into subsets of different type:
Let $\OOO_{c}\subseteq \OOO$ denote the set of equivalence classes attached to the orbits of central elements.
Hence $\nnn\in\OOO_c$ and for every $\ooo\in \OOO_c$ there exists $a\in\Q$ such that $\ooo=a\One_n+\nnn$. Write $\OOO_{c,*}=\OOO_c\minus\nnn$.
Then if $n=2$, we get a disjoint union
\[
 \OOO_*^{\mathfrak{gl}_2}=\OOO_{c,*}\cup \OOO_{\text{reg}}'\cup\OOO_{\text{er}}.
\]
If $n=3$, there is one type of equivalence classes missing:
Let $\OOO_{(2,1)}$ denote the set of $\ooo\in\OOO=\OOO^{\mathfrak{gl}_3}$ for which there are $a,b\in\Q$, $a\neq b$, such that every element $X\in\ooo$ has $a$ as an eigenvalue with multiplicity $2$ and $b$ as an eigenvalue with multiplicity $1$. We denote the equivalence class corresponding to $a,b$ by $\ooo_{(a,b)}$. 
 Then
\[
 \OOO_*^{\mathfrak{gl}_3}=\OOO_{c,*}\cup \OOO_{\text{reg}}'\cup\OOO_{\text{er}}\cup \OOO_{(2,1)}. 
\]
If convenient, we will assume without further notice that $\Phi$ is invariant under $\Ad \cpt$.

\subsection{Contribution from $\OOO_{c,*}$}
We first deal with the contribution from the classes in $\OOO_{c,*}$.
\begin{proposition}
Let $T\in\aaa^+$ be sufficiently regular.
 Then there exists a seminorm $\mu$ on $\SSS^{\nu}(\ggG(\A))$  such that for all $\Phi\in\SSS^{\nu}(\ggG(\A))$ we have
\begin{equation}\label{central_classes}
 \bigg|\sum_{\ooo\in\OOO_{c,*}} J_{\ooo}^T(\Phi_{\lambda})\bigg|
\leq \lambda^{-n^2+1} \mu(\Phi)
\end{equation}
for all $\lambda\in(0,1]$.
\end{proposition}
\begin{proof}
By the proof of Lemma \ref{fundamental_convergence}, it suffices to estimate the sum over $\ooo\in\OOO_c'$  and standard parabolic subgroups $P_1\subseteq R\subseteq P_2$ of \eqref{temp11b}. It further follows from the proof of that lemma and Remark \ref{remark}, that it suffices to find a bound for the case that $R=P_2=G$  if $G=\GL_3$, and $R=P_2$ if $G=\GL_2$. However, if $G=\GL_2$ and $R=P_2\subsetneq G$, we can use the estimate given in Remark \ref{remark_on_nilpotent} (recall that $\ooo=a\One_n+\nnn$ for some $a\in\Q\minus\{0\}$) in the proof of Lemma \ref{fundamental_convergence} to get the stated upper bound. Hence we are left with $R=P_2=G$ for $n=2$ as well as $n=3$. 
\begin{itemize}
 \item[{\bf $G=\GL_2$:}] We need to estimate the sum-integrals
\begin{align}\label{nilpgl21}
 & \int_{A_GP_0(\Q)\backslash G(\A)} \tau_0^G(H_0(x)-T)\sum_{X\in\Q\One_2, X\neq0}\sum_{Y\in\nnn\cap \tilde{\mmm}_0^G(\Q)} \big|\Phi(\lambda (X+\Ad x^{-1} Y))\big| dx,
\;\;\text{ and }\\\label{nilpgl22}
 & \int_{A_G G(\Q)\backslash G(\A)} F(x, T)\sum_{X\in\Q\One_2, X\neq0}\sum_{Y\in\nnn} \big|\Phi(\lambda (X+\Ad x^{-1} Y))\big| dx.
\end{align}
We can replace $|\Phi|$ without loss of generality by a product $\Phi_1\Phi_2$ with  $\Phi_1\in \SSS(\A)$ and $\cpt$-conjugation invariant $\Phi_2\in\SSS(\mathfrak{sl}_2(\A))$ such that $|\Phi(X)|\leq \Phi_1(\tr X) \Phi_2(X-\frac{1}{2}\tr X\id)$ for all $X\in\ggG(\A)$ and such that the relevant seminorms of $\Phi_1$ and $\Phi_2$ are bounded from above by seminorms of $\Phi$ in the sense of Lemma \ref{std_estimates_for_SB}.
If $Y=(Y_{ij})_{i,j=1,2}\in\nnn$, then $Y_{22}=-Y_{11}$, and either $Y_{11}=Y_{21}=Y_{22}=0$ (such elements do not occur in the sum \eqref{nilpgl21}), or $Y_{21}\neq0$ and $Y_{12}=-Y_{11}^2/Y_{21}$ so that 
$Y=\left(\begin{smallmatrix}1&Y_{11}/Y_{21}\\0&1\end{smallmatrix}\right)\left(\begin{smallmatrix} 0&0\\Y_{21}&0\end{smallmatrix}\right)\left(\begin{smallmatrix}1&-Y_{11}/Y_{21}\\0&1\end{smallmatrix}\right)$. 
Hence \eqref{nilpgl21} can be bounded from above by
\begin{align*}
& \int_{A_G T_0(\Q)\backslash G(\A)}  \tau_0^G(H_0(x)-T) \sum_{a\in\Q\minus\{0\}} \Phi_1(\lambda a) \sum_{Y_0\in\Q\minus\{0\}} \Phi_2\big(\lambda \Ad x^{-1}\begin{pmatrix} 0&0\\Y_0&0\end{pmatrix}\big) dx\\
& \leq \mu_1(\Phi_1)\lambda^{-1} \int_{A_0^G} \delta_0(a)^{-1} \tau_0^G(H_0(a)-T) \sum_{Y_0\in\Q\minus\{0\}}\int_{U_0(\A)}\Phi_2\big(\lambda \Ad (ua)^{-1}\begin{pmatrix} 0&0\\Y_0&0\end{pmatrix}\big) du~da,
\end{align*}
where $\mu_1$ denotes a suitable seminorm on $\SSS(\A)$.
Now if we write $a=\diag(a, a^{-1})\in A_0^G$, $a\in\R_{>0}$,
\begin{align*}
 \int_{U_0(\A)}\Phi_2(\lambda \Ad (ua)^{-1}\begin{pmatrix} 0&0\\Y_0&0\end{pmatrix}) du
& =\int_{\A} \Phi_2\big(\lambda \begin{pmatrix} -uY_0 &-u^2 a^{-2} Y_0 \\ a^2Y_0 & uY_0\end{pmatrix}\big) du\\
& \leq \varphi(\lambda a^{2} Y_0)\int_{\A} \varphi(\lambda uY_0)\varphi(-\lambda u^2a^{-2} Y_0) du,
\end{align*}
where $ \varphi\in\SSS(\A)$ is a suitable function related to $\Phi_2$ by Lemma \ref{std_estimates_for_SB}. We can moreover assume that $\varphi$ is monotonically decreasing in the sense that if $x,y\in\A$ with $|x|\leq |y|$, then $\varphi(x)\geq\varphi(y)$.
If $\tau_0^G(H_0(a)-T)=1$, i.e., $2\log a\geq \alpha(T)$ for $\alpha$ the unique simple root, we distinguish the cases $|u|\leq 1$ and $|u|\geq1$. With this we can bound the last integral by $\varphi(\lambda a^{2} Y_0)a^2\lambda^{-1} \mu_2(\varphi)$ for $\mu_2$ a suitable seminorm. Hence \eqref{nilpgl21} is bounded by
\[
 \mu_3(\Phi)\lambda^{-3} \int_{A_0^G} \delta_0(a)^{-1} \tau_0^G(H_0(a)-T) da
= \mu_3(\Phi)\lambda^{-3}e^{-\alpha(T)}/2
\]
for a suitable seminorm $\mu_3$ on $\SSS^{\nu}(\ggG(\A))$.

Now for \eqref{nilpgl22} note that $\nnn$ is the disjoint union of $\uuu_0$ and $\nnn\cap\tilde{\mmm}_0^G(\Q)$.
Then  \eqref{nilpgl22} is bounded from above by
\begin{align*}
\lambda^{-1}\mu_1(\Phi_1)\bigg( & \int_{A_G P_0(\Q)\backslash\SSS_{T_1}} F(x, T) \sum_{Y\in \uuu_0(\Q)} |\Phi_2(\lambda \Ad x^{-1} Y)| ~dx\\
+ & \int_{A_G P_0(\Q)\backslash\SSS_{T_1}} F(x, T)\sum_{Y\in\nnn\cap\tilde{\mmm}_0^G(\Q)} |\Phi_2(\lambda \Ad x^{-1} Y)| ~dx\bigg)
\end{align*}
for which the first sum is bounded by 
\begin{align*}
 \lambda^{-1}\mu_1(\Phi_1)\varphi(0)^2  & \int_{A_{0}^G(T_1)} \delta_0(a)^{-1} 
\bigg(\int_{U_0(\Q)\backslash U_0(\A)}F(ua, T)~du \bigg) \sum_{Y\in \Q} \varphi(\lambda a^{-2}Y) ~da\\
& \leq \lambda^{-1}\mu_1(\Phi_1)\varphi(0)^2 \int_{e^{\alpha(T_1)/2}}^{e^{\alpha(T)/2}} a^{-2} \sum_{Y\in \Q} \varphi(\lambda a^{-2}Y) ~d^{\times}a.
\end{align*}
This is bounded by the product of $\lambda^{-2} \mu_4(\Phi)$ and a linear polynomial in $T$ for some suitable seminorm $\mu_4$ on $\SSS^{\nu}(\ggG(\A))$.
For the second integral recall that $F(uak, T)\leq \hat{\tau}_0(T-H_0(a))=\tau_0^G(T-H_0(a))$ for all $a\in A_{0}^G(T_1)$. 
Using similar manipulations as for \eqref{nilpgl21}, the second integral is therefore bounded by 
\[
 \mu_5(\Phi)\lambda^{-3}\int_{e^{\alpha(T_1)/2}}^{e^{\alpha(T)/2}} a^{-2} d^{\times} a,
\]
which equals a constant multiple of $\mu_5(\Phi)\lambda^{-3}e^{-\alpha(T)}$ for some seminorm $\mu_5$ on $\SSS^{\nu}(\ggG(\A))$.
Hence the assertion of the proposition is proven for $G=\GL_2$.

\item[{\bf $G=\GL_3$:}] For every standard parabolic in $P_1\subseteq G$ we need to estimate the sum-integral
\begin{equation}\label{nilpgl31}
 \int_{A_GP_1(\Q)\backslash G(\A)} F^{P_1}(x, T) \tau_{P_1}^G(H_0(x)-T)\sum_{X\in \Q\One_3, X\neq0}\sum_{Y\in\nnn\cap \tilde{\mmm}_{P_1}^G(\Q)} 
\big|\Phi(\lambda (X+\Ad x^{-1} Y))\big| dx,
\end{equation}
or rather, using the same notation and arguments as in the previous case,
\begin{equation*}\label{nilpgl32}
 \int_{A_GP_1(\Q)\backslash G(\A)} F^{P_1}(x, T) \tau_{P_1}^G(H_0(x)-T)\sum_{Y\in\nnn\cap \tilde{\mmm}_{P_1}^G(\Q)} 
\big|\Phi_2(\lambda\Ad x^{-1} Y)\big| dx,
\end{equation*}
since again $\sum_{X\in \Q, X\neq0}\Phi_1(\lambda X)\leq\mu_1(\Phi_1) \lambda^{-1}$ for some seminorm $\mu_1$ on $\SSS(\A)$.
First, suppose $P_1=P_0$ is the minimal parabolic subgroup. 
Then $\tilde{\mmm}_{P_0}^G(\Q)\cap \nnn$ is the disjoint union of the set of those nilpotent $Y=(Y_{ij})_{i,j=1,2,3}$ with $Y_{31}\neq0$ and those with $Y_{31}=0$, but $Y_{21}\neq0\neq Y_{32}$. The elements $Y$ satisfying the second property are contained in the codimension one vector subspace $\{Y\in \nnn \mid Y_{31}=0\}$ of $\nnn$ so that by similar arguments as before, an upper bound as asserted holds for this sum. 
Hence we are left to consider the sum over those $Y\in\nnn$ with $Y_{31}\neq0$. By the same reasoning we may further restrict to those $Y$ with $Y_{31}\neq0\neq Y_{21}$.
 Since $Y$ is nilpotent, for every such $Y$ there exists $u\in U_0(\Q)$ such that in the matrix $\Ad u Y$ either the second or third column is identically equal to~$0$. Moreover, the $(2,1)$- and the $(3,1)$-entry in $\Ad u Y$ is the same as in $Y$ and a similar analysis as in the case of $G=\GL_2$ for \eqref{nilpgl21} shows that \eqref{nilpgl31} is bounded as asserted. 

Next suppose that $P_1=M_1U_1$ is the maximal standard parabolic subgroup with $M_1=\GL_2\times\GL_1\hookrightarrow\GL_3$ (diagonally embedded). (The other maximal standard parabolic subgroup is treated the same way.)
Then 
\[
A_GP_1(\Q)\backslash G(\A)\simeq  U_1(\Q)\backslash U_1(\A)\times  A_G M_1(\Q)\backslash M_1(\A) \times\cpt,
\]
 $F^{P_1}(u mk, T)=F^{M_1}(m, T^{M_1})$ for $u\in U_1(\A)$, $m\in M_1(\A)$, $k\in\cpt$, and $\tau_{P_1}^G(H_0(umk)-T)=\tau_{P_1}^G(H_0(m)-T)$.
Now if $Y\in \nnn\cap \tilde{\mmm}_{P_1}^G(\Q)$, then $(Y_{31}, Y_{32})\neq (0,0)$, and there exists $u\in U_0(\Q)$ such that the second or third column of $\Ad u Y$ is identically $0$. If there exists $u\in U_1(\Q)$ such that the last column of $\Ad u Y$ is $0$ (note that the $(3,1)$-and $(3,2)$-entries stay unchanged under $\Ad u$), we proceed similar as in the case of $\GL_2$ and the estimation of \eqref{nilpgl21}. Otherwise there exists $u\in U_0^{M_1}(\Q)$ such that the second column of $\Ad u Y$ is $0$ and the $(3,1)$-entry stays unchanged. This again leads to an upper bound of the asserted form by using a similar approach as for $\GL_2$ and \eqref{nilpgl22}.

Hence we are left with $P_1=G$. We estimate the corresponding integral again by an integral over a quotient of the Siegel domain $A_G P_0(\Q)\backslash \SSS_{T_1}$. Moreover, $F^G(umk,T)\leq\hat{\tau}_0^G(T-H_0(m))$ for $umk\in U_0(\A) T_0(\A) \cpt$. Hence a similar reasoning as for $\GL_2$ and the integral \eqref{nilpgl22} yields an upper bound as asserted.

Taking the estimates for all standard parabolic subgroups $P_1$ together, the assertion follows now also for $\GL_3$. \qedhere
\end{itemize}
\end{proof}

\subsection{Contribution from $\OOO_{\text{reg}}'$}
Let $\ooo\in\OOO_{\text{reg}}'$ and let $X_1\in\ooo$ be semisimple. 
Let $P_1$ be the smallest standard parabolic subgroup such that $X_1\in \mmm_{1}(\Q)$. We may assume that $X_1\in\ooo$ is chosen such that $X_1$ is not contained in any proper (not necessarily standard) parabolic subalgebra of $\mmm_1(\Q)$. 
We may further assume that if $G=\GL_2$, then $M_1=\GL_1\times\GL_1=T_0$ (diagonally embedded into $G$), or if $G=\GL_3$, then $M_1=\GL_1\times\GL_1\times\GL_1=T_0$ or $M_1=\GL_2\times\GL_1$. Then $G_{X_1}\subseteq M_{1,X_1}$ and $X_1\in\OOO_{\text{er}}^{\mmm_1}$ so that $A_{M_1}=A_{G_{X_1}}$, where $\OOO^{\mmm_1}_{\text{er}}\subseteq \OOO^{\mmm_1}$ denotes the set of regular elliptic equivalence classes in $\mmm_1(\Q)$. 
Let $\MMM=\{T_0\}$ if $G=\GL_2$, and $\MMM=\{T_0, \GL_2\times\GL_1\}$ if $G=\GL_3$.
We have a canonical bijection (given by induction of the equivalence classes along the unipotent radical of an arbitrary parabolic subgroup with Levi component $M$)
\begin{equation}\label{ss_descent_eq_classes}
 \bigcup_{M\in\MMM} \OOO^{\mmm}_{\text{reg, ell}}\longrightarrow \OOO_{\text{reg}}'.
\end{equation}
For $\ooo\in\OOO_{\text{reg}}$ the distribution $J_{\ooo}^T(\Phi)$ is a weighted orbital integral and equals for sufficiently regular~$T$
\[
 J_{\ooo}^T(\Phi)=\vol(A_{M_1}G_{X_1}(\Q)\backslash G_{X_1}(\A))\int_{G_{X_1}(\A)\backslash G(\A)} \Phi(\Ad x^{-1} X_1) v_1(x, T) ~dx
\]
(cf.\ \cite[\S 5.2]{Ch02}), where the weight function $v_1(x, T)$ is given by the volume of the convex hull (in $\aaa_1^G$) of the points $H_{P}(x)\in\aaa_1^G$ with $P$ running over parabolic subgroups in $\FFF(M_1)$ with Levi component $M_1$. In particular, $v_1(\cdot, T)$ is left $M_1(\A)$- and right $\cpt$-invariant.
 It is easily seen that this expression for $J_{\ooo}^T(\Phi)$ stays true for $\Phi\in \SSS^{\nu}(\ggG(\A))$ with $\nu$ as in Theorem \ref{main_theorem}.

\begin{proposition}
Let $T\in\aaa^+$ be sufficiently regular. 
 \begin{enumerate}[label=(\roman{*})]
  \item For $\ooo\in\OOO_{\text{reg}}'$ and $X_1\in \ooo$ as before, $v_1(x, T)$ is a polynomial in the variables $\log(q(x, X_1, T))$ with $q$ ranging over a finite collection of polynomials in the coordinate entries of $x$, $X_1$ and $T$.

\item There exists a seminorm $\mu$ on $\SSS^{\nu}(\ggG(\A))$ such that for all $\Phi\in\SSS^{\nu}(\ggG(\A))$ we have
\[
 \sum_{\ooo\in\OOO_{\text{reg}}'} \big|J_{\ooo}^T(\Phi_{\lambda}) \big|\leq \mu(\Phi) \lambda^{-(n^2-\frac{1}{2})}
\]
for all $\lambda\in(0,1]$.
 \end{enumerate}
\end{proposition}
\begin{proof}
 \begin{enumerate}[label=(\roman{*})]
  \item This is clear from the definition of the weight function.

\item Using Iwasawa decomposition, the left $M_1(\A)$-, and the right $\cpt$-invariance of $v_1(\cdot, T)$, we get for every $\ooo\in\OOO_{\text{reg}}'$
\[
 J_{\ooo}^T(\Phi_{\lambda})
= v_{X_1}^G \int_{M_{1, X_1}(\A)\backslash M_1(\A)}\int_{U_1(\A)} \Phi_{\lambda}(\Ad u^{-1}\Ad m^{-1} X_1) v_1(u, T) ~du~dm
\]
where we write $v_{X_1}^G=\vol(A_{G_{X_1}}G_{X_1}(\Q)\backslash G_{X_1}(\A))$ (note that $v_{X_1}^G=v_{X_1}^{M_1}$).
As $X_1$ and therefore also $\Ad m^{-1} X_1$ is semisimple and regular ($X_1$ is regular elliptic in $\mmm_1$), the map
$U_1(\A)\ni u\mapsto U=U(u, \Ad m^{-1} X_1):=\Ad u^{-1}\Ad m^{-1} X_1-\Ad m^{-1} X_1\in\uuu_1(\A)$ is a diffeomorphism with Jacobian $D(X_1):=\det(\ad( \Ad m^{-1}X_1);\uuu_2)=\det(\ad X_1; \uuu_2) $. We denote its inverse by $U\mapsto u(U, \Ad m^{-1} X))\in U_1(\A)$.
Hence the above integral equals
\begin{align*}
 v_{X_1}^G |D(X_1)|_{\A} & \int_{M_{1, X_1}(\A)\backslash M_1(\A)}\int_{\uuu_1(\A)} \Phi_{\lambda}(\Ad m^{-1} X_1 + U) v_1(u(U, \Ad m^{-1} X_1), T)~ dU~dm\\
=v_{X_1}^G \lambda^{-\dim \uuu_1} & \int_{M_{1, X_1}(\A)\backslash M_1(\A)}\int_{\uuu_1(\A)} \Phi(\lambda\Ad m^{-1} X_1 + U) v_1(u(\lambda^{-1}U, \Ad m^{-1} X_1), T) ~dU~dm
\end{align*}
For $Y\in \mmm_1(\A)$ define
\[
 \Psi^{M_1}(\lambda,Y)=\int_{\uuu_1(\A)} \Phi(\lambda Y + U) v_1(u(\lambda^{-1}U, Y), T) ~dU.
\]
By the first part of the proposition, we can find a finite collection of polynomials $Q_1, \ldots, Q_m$, $q_{1,1}, \ldots,q_{1,l_1}, \ldots, q_{m,l_m}$, and integers $k_1, \ldots, k_m\geq0$ such that
\[
 |v_1(u(\lambda^{-1}U, Y), T)|
\leq \sum_{i=1}^{m}|\log\lambda|^{k_i} Q_i\big(\log q_{i,1}(U, Y, T), \ldots, \log q_{i, l_i}(U,Y,T)\big)
\]
for all $\lambda\in(0,1]$ and $U$, $Y$, and $T$ as before.
Then
\[
 | \Psi^{M_1}(\lambda,Y)|
\leq \sum_{i=1}^m |\log\lambda|^{k_i} \tilde{\Psi}^{M_1}_{i,\lambda}(Y),
\]
where for $Y\in\mmm_1(\A)$,
\[
 \tilde{\Psi}^{M_1}_{i}(Y)
:=\int_{\uuu_1(\A)}\tilde{\Phi}(Y+U)  Q_i\big(\log q_{i,1}(U, Y, T), \ldots, \log q_{i, l_i}(U,Y,T)\big)~ dU
\]
and $ \tilde{\Psi}^{M_1}_{i, \lambda}(Y):= \tilde{\Psi}^{M_1}_{i}(\lambda Y)$. Here $\tilde{\Phi}\in\SSS(\ggG(\A))$ is a suitable smooth function satisfying the seminorm estimates as in Lemma \ref{std_estimates_for_SB} and such that $\tilde{\Phi}\geq |\Phi|$.
Then $\tilde{\Psi}^{M_1}_{i}\in\SSS(\mmm_1(\A))$, and 
\[
 |J_{\ooo}^T(\Phi_{\lambda})|
\leq \lambda^{-\dim\uuu_1} \sum_{i=1}^m |\log\lambda|^{k_i} J_{\ooo^{\mmm_1}}^{M_1, T^{M_1}}(\tilde{\Psi}^{M_1}_{i}),
\]
where $\ooo^{\mmm_1}\in \OOO^{\mmm_1}_{\text{reg, ell}}$ denotes the inverse image of $\ooo$ under the map \eqref{ss_descent_eq_classes}, $T^{M_1}$ is the projection of $T$ onto $\aaa^{M_1}_0$, and $J_{\ooo^{\mmm_1}}^{M_1, T^{M_1}}$ denotes the distribution associated with $\ooo^{\mmm_1}$ with respect to $M_1$.
Hence by Lemma \ref{fundamental_convergence} there are seminorms $\mu_M$ on $\SSS(\mmm(\A))$ for every $M\in\MMM$ such that for every $\lambda\in(0,1]$ we have
\begin{align*}
  \sum_{\ooo\in\OOO_{\text{reg}}} |J_{\ooo}^T(\Phi_{\lambda})|
& \leq \sum_{M\in\MMM}\lambda^{-\dim \uuu}  \sum_{i=1}^m |\log\lambda|^{k_i} \sum_{\ooo'\in\OOO^{\mmm}_{\text{reg,ell}}}J_{\ooo'}^{M_1, T^{M_1}}(\tilde{\Psi}^{M_1}_{i})\\
& \leq \sum_{M\in\MMM}\lambda^{-\dim \uuu}  \sum_{i=1}^m |\log\lambda|^{k_i} \mu_M(\tilde{\Psi}^{M}_i) \lambda^{-\dim\mmm}\\
& \leq \mu(\Phi)\sum_{M\in\MMM}\lambda^{-\dim \ppp}  \sum_{i=1}^m |\log\lambda|^{k_i} ,
\end{align*}
where $\mu(\Phi):=\max_{M, i}\mu_M(\tilde{\Psi}^{M}_i)$ which is a seminorm on $\SSS^{\nu}(\ggG(\A))$.
Since $\dim\ppp\leq \dim\ggG -1$ for every $M\in\MMM$, the assertion follows by using some trivial estimate of the form $|\log\lambda|^{k_i}\leq c_i \lambda^{-1/2}$, $c_i>0$ some constant, for the logarithmic terms.
\end{enumerate}
\end{proof}

Above results together with the fact that all distributions are polynomials in $T$ so that above results hold for every $T\in\aaa$ and not only sufficiently regular ones, implies the following:
\begin{cor}\label{holom_ct_2}
 If $n=2$, then $\Xi^T(s, \Phi)-\Xi_{\text{main}}(s, \Phi)$ can be holomorphically continued at least to $\Re s>\frac{n+1}{2}-\frac{1}{2}=1$ for every $T\in\aaa$.
\end{cor}

\subsection{Contribution of the classes of type $(2,1)$}
For $n=3$, the contribution from the classes in $\OOO_{(2,1)}$ is still missing.
Let $\ooo_{(a,b)}\in\OOO_{(2,1)}$. Then every semisimple element in $\ooo_{(a,b)}$ is over $\GL_3(\Q)$ conjugate to $X_s=\diag(a,a,b)$ so that $G_{X_s}=\GL_2\times\GL_1=:M_2$ (diagonally embedded in $\GL_3$).
Let $P_2=M_2U_2$ denote the standard parabolic subgroup with Levi component $M_2$. 
Note that $\ooo_{(a,b)}=a\One_3 +\ooo_{(0,b-a)}$.
\begin{proposition}
 Let $T\in\aaa^+$ be sufficiently regular. There exists a seminorm $\mu$ on $\SSS^{\nu}(\ggG(\A))$ such that
\[
 \bigg|\sum_{\ooo\in\OOO_{(2,1)}}\Phi_{\ooo}^T(\Phi_{\lambda})\bigg|
\leq \lambda^{-n^2+1}\mu(\Phi)
\]
for all $\Phi\in\SSS^{\nu}(\ggG(\A))$ and all $\lambda\in (0,1]$.
\end{proposition}
\begin{proof}
 We use the same idea as for the central contribution so that we need to consider the sum-integrals 
\[
 \int_{A_G P_1(\Q)\backslash G(\A)} F^{P_1}(x, T)\tau_{P_1}^G(H_0(x)-T) \sum_{\ooo\in\OOO_{(2,1)}}\sum_{X\in\tilde{\mmm}_{P_1}^G(\Q)\cap \ooo} \big|\Phi(\lambda \Ad x^{-1} X)\big|~ dx
\]
for standard parabolic subgroups $P_1\subseteq G$.

Suppose first that $P_1=P_0$ is minimal. Then $X\in\tilde{\mmm}_{P_0}^G(\Q)$ if and only if $X_{31}\neq0$ or $X_{31}=0$ and $X_{21}\neq0\neq X_{32}$. The sum-integral restricted to $X$ satisfying the second property $X_{31}=0$  gives an upper bound as asserted by the same reasons as before. Hence it suffices to consider
\[
  \int_{A_G P_0(\Q)\backslash G(\A)}\tau_{P_0}^G(H_0(x)-T)\sum_{\ooo\in\OOO_{(2,1)}} \sum_{X\in\ooo:X_{31}\neq0} \big|\Phi(\lambda \Ad x^{-1} X)\big| ~dx.
\]
As remarked before, we have $\bigcup_{\ooo\in\OOO_{(2,1)}}\ooo=\bigcup_{a\in\Q}\big(a\One_3+\bigcup_{b\in\Q\minus\{0\}} \ooo_{(0,b)}\big)$. If $Y\in \ooo_{(0,b)}$ and $Y_{31}\neq0$, then $\det Y=0$ and there exists $u\in U_2(\Q)$ such that $Z:=\Ad u Y$ satisfies $Z_{31}=Y_{31}\neq0$ and $Z_{13}=Z_{23}=Z_{33}=0$, or there exists $u\in U_0^{M_2}(\Q)$ such that $Z:=\Ad u Y$ satisfies $Z_{31}=Y_{31}\neq0$ and  $Z_{12}=Z_{22}=Z_{32}=0$.
Let $V_{3,1}^i\subseteq \ggG(\Q)$, $i=2,3$, denote the elements $Z\in \ggG(\Q)$ with $Z_{31}\neq0$ and $Z_{1i}=Z_{2i}=Z_{3i}=0$.
Then the above integral is bounded by
\begin{align*}
 & \int_{A_G U_0^{M_2}(\Q)\backslash G(\A)} \tau_0^G(H_0(x)-T) \sum_{a\in \Q}\sum_{Z\in V_{3,1}^3}\big| \Phi(\lambda \Ad x^{-1} (a\One_3 +Z))\big|~dx\\
+ & \int_{A_G U_2(\Q)\backslash G(\A)} \tau_0^G(H_0(x)-T) \sum_{a\in \Q}\sum_{Z\in V_{3,1}^2} \big|\Phi(\lambda \Ad x^{-1} (a\One_3 +Z))\big|~dx.
\end{align*}
From this it follows similarly as in the central case that the integral satisfies the asserted upper bound.

The remaining cases $P_0\subsetneq P_1\subseteq G$ are combinations of the previous case and the considerations for the central contribution. We omit the details.
\end{proof}

\begin{cor}\label{holom_ct_3}
 If $n=3$, then $\Xi^T(s, \Phi)-\Xi_{\text{main}}(s, \Phi)$ can be holomorphically continued at least to $\Re s>\frac{n+1}{2}-\frac{1}{2}=\frac{3}{2}$ for every $T\in\aaa$.
\end{cor}

\part{Density results for the cubic case}\label{part_meanval}
The purpose of this second part of the paper is to give upper and lower bounds (see Theorem \ref{asymptotic_real_fields} and Proposition \ref{lower_bound_for_residues}) for the mean value 
\begin{equation}\label{mean_sum}
X^{-\frac{5}{2}} \sum_{E:~m_1(E)\leq X} \res_{s=1}\zeta_E(s)
\end{equation}
as $X\rightarrow\infty$, where $E$ runs over all totally real cubic fields and $m_1(E)$ denotes the second successive minimum of the trace form on the ring of integers of $E$, see below.
For the upper bound we study the main part of the zeta function $\Xi_{\text{main}}(s,\Phi)$ for $\GL_3$ for suitable test functions $\Phi$. As explained above, the distributions $J_{\ooo}(\Phi)$ for $\ooo\in\OOO_{\text{er}}$ occurring in the definition of $\Xi_{\text{main}}(s,\Phi)$ are orbital integrals over orbits of regular elliptic elements. Hence in \S \ref{section_mean_val_orbint} we first study the local orbital integrals at the non-archimedean places. In \S \ref{section_asymptotic} we define suitable test functions and show an asymptotic for mean values of orbital integrals by using results from Part \ref{part_zeta}, before finally proving the asymptotic upper and lower bounds for \eqref{mean_sum} in \S \ref{bounds_for_residues}.

\section{Non-archimedean orbital integrals}\label{section_mean_val_orbint}
In this section let $G=\GL_n$ and $\ggG=\mathfrak{gl}_n$ with $n\geq 2$ arbitrary.
If $E$ is an $n$-dimensional field extension of $\Q$, let $\OOO_E$ be the ring of integers of $E$.
For $\gamma\in G(\Q)$ let $[\gamma]=\{x^{-1}\gamma x\mid x\in G(\Q)\}$ be the conjugacy class of $\gamma$ in $G(\Q)$.
As before, let $G(\Q)_{\text{er}}$ denote the set of regular elliptic elements in $G(\Q)$. Let $\FFF_n$ be the set of $n$-dimensional number fields. 
We get a surjective map from $G(\Q)_{\text{er}}$ onto the set of Galois conjugacy classes in $\FFF_n$ by attaching to $\gamma\in G(\Q)_{\text{er}}$ the conjugacy class of the field $\Q(\xi)$ for $\xi$ an (arbitrary) eigenvalue of $\gamma$. If $[E]\subseteq \FFF_n$ is such a conjugacy class and if $\Gamma_{[E]}\subseteq G(\Q)_{\text{er}}$ is the inverse image of $[E]$ under this map, then  $\Gamma_{[E]}$ is invariant under conjugation by elements of $G(\Q)$, and 
\[
 \{\xi\in E\mid \Q(\xi)=E\}\longrightarrow [\gamma_{\xi}]\in\Gamma_{[E]}/\sim
\]
is surjective. Here $\gamma_{\xi}\in G(\Q)$ denotes the companion matrix of the characteristic polynomial of $\xi$, and the map is $|\Aut(E/\Q)|$-to-$1$.

If $K$ is a finite field extension of $\Q_p$ with ring of integers $\OOO_K\subseteq K$, we normalise the measures on $K$ and $K^{\times}$ such that $\vol(\OOO_K)=1=\vol(\OOO_K^{\times})$. 
If $\theta\in\OOO_K$ is such that $\{1,\theta, \ldots, \theta^{n-1}\}$ is a basis of $K$ over $\Q_p$, let $\gamma_{\theta}\in\GL_n(\Q_p)$ denote the companion matrix of $\theta$. Then $G_{\gamma_{\theta}}(\Q_p)$ is isomorphic to $K^{\times}$ via the isomorphism induced by $\{1,\theta, \ldots, \theta^{n-1}\}$ and we define the measure on $G_{\gamma_{\theta}}(\Q_p)$ via this isomorphism. 
If $\Phi_p\in\SSS(\ggG(\Q_p))$, we define the $p$-adic orbital integrals 
\[
I_p(\Phi_p,\theta)=I_p(\Phi_p,\gamma_{\theta})=
\int_{G_{\gamma_{\theta}}(\Q_p)\backslash G(\Q_p)}\Phi_p(g^{-1}\gamma_{\theta} g)dg.
\] 
If $\gamma\in G(\Q)_{\text{er}}$, then $G_{\gamma}(\Q_p)$ is isomorphic to a direct product of $K_1^{\times}\times\ldots\times K_r^{\times}$ for suitable finite field extensions $K_1, \ldots, K_r/\Q_p$ and we choose the measure on $G_{\gamma}(\Q_p)$ such that it is compatible with our choice of measures on $K_1^{\times}\times\ldots\times K_r^{\times}$, and put $I_f(\Phi_f, \gamma)=\prod_{p<\infty}I_p(\Phi_p, \gamma)$.
Similarly, we define $I_{\infty}(\Phi_{\infty}, \gamma)$ (resp., $I(\Phi, \gamma)$) if $\Phi_{\infty}\in\SSS^{\nu}(\ggG(\R)$ (resp.,  $\Phi\in\SSS^{\nu}(\ggG(\A))$).

Our aim in this section is to understand the quantities
\[
c(\Phi_p, \gamma)=\frac{I_p(\Phi_p,\gamma)}{[\OOO_{\Q_p[\gamma]}:\Z_p[\gamma]]},
~~~~
\text{ and }
~~~~
c(\Phi_f, \gamma)=\frac{I_f(\Phi_f,\gamma)}{[\OOO_{\Q[\gamma]}:\Z[\gamma]]},
\]
where we denote for a $\Q$- or $\Q_p$-algebra $A$ the ring of integers of $A$ by $\OOO_A$.
If  $\xi\in E$ generates $E$ over $\Q$, we set $I_p(\Phi_p,\xi)=I_p(\Phi_p, \gamma_{\xi})$, and define $I_f(\Phi_p, \xi)$, $I(\Phi, \xi)$, $c(\Phi_p, \xi)$, $c(\Phi_f, \xi)$ analogously.

For a prime $p$ and $E\in\FFF_n$ let $E_p=E\otimes_{\Q}\Q_p$. If $\Phi_p$ (resp., $\Phi_f$)  is supported in $\ggG(\Z_p)$ (resp., $\ggG(\hat{\Z})$), then the orbital integral $I_p(\Phi_p, \xi)$ (resp., $I_f(\Phi_f, \xi)$) vanishes unless $\xi\in\OOO_{E_p}$ (resp., $\xi\in\OOO_E$).
We denote by $\Phi_p^0\in\SSS(\ggG(\Q_p))$ the characteristic function of $\ggG(\Z_p)$, and $\Phi_f^0=\prod_{p<\infty}\Phi_p^0\in\SSS(\ggG(\A_f))$.

\begin{proposition}\label{orb_int_greaterequal1}
 Let $E\in\FFF_n$ and  $\xi\in \OOO_E$ be such that $\Q(\xi)=E$. Then
\begin{enumerate}[label=(\roman{*})]
\item\label{partiproposition} 
If $E_p\simeq K_1\oplus\ldots\oplus K_r$ with $K_i/\Q_p$ field extensions, and if under this isomorphism $\xi$ corresponds to $(\xi_1, \ldots, \xi_r)\in K_1\oplus\ldots\oplus K_r$, we have
\[
c(\Phi_p^0, \xi)
=\prod_{i=1}^r c(\Phi_p^0, \xi_i),
\]
where $\Phi_p^0$ also denotes the characteristic function of $\mathfrak{g}_{n_i}(\Z_p)$, $n_i:=[K_i:\Q_p]$, and $c_p$ is defined on the smaller groups similar as before.
\item $c(\Phi_p^0, \xi)\geq 1$.
\item $c(\Phi_p^0, \xi+a)=c(\Phi_p^0, \xi)$
for every $a\in \Z$. Hence $c(\Phi_f^0, \cdot)$ is a well-defined function on $\OOO_E/\Z$.
\end{enumerate}
\end{proposition}
Before proving this proposition we need a few auxiliary results and fix some further notation.
If $\xi$ as is in the proposition, denote by $P_{p, \xi}$ the standard parabolic subgroup of type $(n_1, \ldots, n_r)$. Then
\begin{equation}\label{p-adic_ssdesc}
I_p(\Phi_p^0, \xi)
=\delta_{P_{p, \xi}}(\diag(\xi_1, \ldots, \xi_r))^{-1/2}\prod_{i=1}^r I_{p}(\Phi_p^0, \xi_i).
\end{equation}
Let $\Delta$ denote the discriminant map for $E\longrightarrow \Q$ as well as for $\ggG(\Q)\longrightarrow \Q$ and $F\longrightarrow \Q_p$ for $F/\Q_p$ a finite field extensions of arbitrary degree.
 If $F$ is either $\Q$ or $\Q_p$ for some  prime $p<\infty$, let $A$ be a finite-dimensional semisimple $F$-algebra, and $R\subseteq A$ an $\OOO_F$-order. 
We denote by $\Frac(R)$ the set of fractional ideals  of $R$ in $A$, i.e. the set of all full-rank $\OOO_F$-lattices $\aaa\subseteq A$ such that $R\aaa\subseteq \aaa$. 
 If $\aaa\subseteq A$ is a lattice of full rank, let $\mathcal{M}(\aaa)=\{a\in A\mid a\aaa\subseteq \aaa\}$ be the multiplier of $\aaa$. This is an $\OOO_F$-order in $A$, 
in particular $\mathcal{M}(\aaa)\subseteq\OOO_K$ and $\aaa\in \Frac(\mathcal{M}(\aaa))$. Let 
\[
\Frac^0(R)=\{\aaa\in \Frac(R)\mid \mathcal{M}(\aaa)=R\}.
\]
Let $P(R)=\{aR\mid a\in A^{\times}\}$ be the set of all $R$-principal ideals in $A$.
In general, neither $\Frac(R)$ nor $\Frac^0(R)$ are groups, but they are acted on by $P(R)$ so that we may build the quotients $\Frac(R)/P(R)$ and $\Frac^0(R)/P(R)$, which are both finite. 

\begin{lemma}\label{computing_local_non_archim_integral}
Suppose $K$ is a finite field extension of $\Q_p$, and $\theta\in \OOO_K$ generates $K$ over $\Q_p$, i.e. $K=\Q_p(\theta)$. 
Then
\[
I_{p}(\Phi_p^0, \theta)
=\sum_{\ooo\subseteq \OOO_{K}:~\theta\in\ooo}
\big|\Frac^0(\ooo)/P(\ooo)\big| \big[\OOO_{K}^{\times}:\ooo^{\times}\big]
\]
where $\ooo$ runs over all $\Z_p$-orders in $\OOO_{K}$ containing $\theta$.
\end{lemma}
\begin{rem}
 If $\Z_p[\theta]=\OOO_K$, then $I_{p}(\Phi_p^0, \theta)=1$.
\end{rem}
\begin{proof}
We first show
\begin{equation}\label{temp1}
\int_{\Q_p^{\times}\backslash G(\Q_p)}\Phi_p^0(g^{-1}\gamma_{\theta}g)dg
=[K^{\times}:(\OOO_K^{\times}\Q_p^{\times})]\sum_{\ooo\subseteq \OOO_{K}:~\theta\in\ooo}\big|\Frac^0(\ooo)/P(\ooo) \big| \big[\OOO_K^{\times}:\ooo^{\times}\big].
\end{equation}
 The set $\{1,\theta,\ldots,\theta^{n-1}\}$ forms a basis of $K$ relative to which the matrix $\gamma_{\theta}$ corresponds to the endomorphism $K\longrightarrow K$ given by multiplication with $\theta$. Moreover, this basis defines a map 
\[
\Psi:G(\Q_p)=\GL_n(\Q_p)\longrightarrow \mathcal{L}_p=\{L\subseteq K\mid L\text{ is }\Z_p\text{-lattice of full rank}\}.
\]
Hence $\Phi_p^0(g^{-1}\gamma_{\theta}g)\neq 0$ if and only if $\theta$ maps the lattice $L_g=g\OOO_K\subseteq K$ defined by $g$ into itself, i.e. $\theta L_g\subseteq L_g$, or equivalently $\theta\in\mathcal{M}(L_g)\subseteq \OOO_K$.
Hence the integral equals
\[
\sum_{\ooo\subseteq \OOO_{K}:~\theta\in\ooo}\sum_{\aaa\in\Frac^0(\ooo)/\Q_p^{\times}}\vol\big(\Psi^{-1}(\aaa)\big).
\]
Hence we have to compute the volume of $\Psi^{-1}(\aaa)$ as a subset of $G(\Q_p)$. 
Now two elements $g_1,g_2\in G(\Q_p)$ define the same $\Z_p$-lattice if and only if there exists $k\in G(\Z_p)=\cpt_p$ with $g_2=g_1k$. 
Hence with our normalisation of measures we get $\vol\big(\Psi^{-1}(\aaa)\big)=1$. 
Since
\begin{align*}
\big|\Frac^0(\ooo)/\Q_p^{\times}\big|=\big|\Frac^0(\ooo)/(\ooo^{\times}\Q_p^{\times})\big|
& =\big|\Frac^0(\ooo)/(\OOO_K^{\times}\Q_p^{\times})\big| \big[\OOO_K^{\times}:\ooo^{\times}\big]\\
& =\big|\Frac^0(\ooo)/P(\ooo)\big|\big|K^{\times}/(\OOO_K^{\times}\Q_p^{\times})\big| \big[\OOO_K^{\times}:\ooo^{\times}\big],
\end{align*}
 the assertion \eqref{temp1} follows.
If the extension $K/\Q_p$ is unramified, 
$[K^{\times}:(\OOO_K^{\times}\Q_p^{\times})]=1$.
In general, 
$[K^{\times}:(\OOO_K^{\times}\Q_p^{\times})]=[K:\OOO_K^{\times}\Q_p]$
so that this index equals the ramification index, and we  therefore have
$[K^{\times}:(\OOO_K^{\times}\Q_p^{\times})]
=\vol(\Q_p^{\times}\backslash K^{\times})
=\vol(\Q_p^{\times}\backslash G_{\theta}(\Q_p))$.
Hence the assertion of the lemma follows.
\end{proof}

\begin{proof}[Proof of Proposition \ref{orb_int_greaterequal1}]
 \begin{enumerate}[label=(\roman{*})]
\item  This follows from \eqref{p-adic_ssdesc} and the identity
\[
[\OOO_E:\Z_p[\xi]]^2 \delta_{p, \xi}(\diag(\xi_1, \ldots, \xi_r))
=\prod_{i=1}^r[\OOO_{K_i}:\Z_p[\xi_i]]^2.
\]
 \item By \ref{partiproposition} the quotient $c(\Phi_p^0, \xi)$ equals a finite product of terms of the form
\[
\frac{[\OOO_{E}^{\times}:\Z_p[\theta]^{\times}]}{[\OOO_{E}:\Z_p[\theta]]}\big|\Frac^0(\Z_p[\theta])/P(\Z_p[\theta])\big|
+\frac{1}{[\OOO_{E}:\Z_p[\theta]]}\sum_{\substack{\ooo: \\ \Z_p[\theta]\subsetneq\ooo\subseteq \OOO_{E}}}\big|\Frac^0(\ooo)/P(\ooo)\big|[\OOO_{E}^{\times}:\ooo^{\times}]
\]
for $E/\Q_p$ a finite extension generated by $\theta\in E$ with maximal ideal $\ppp\subseteq \OOO_E$ of norm $q$. Hence it certainly suffices to show 
$\frac{[\OOO_{E}^{\times}:\Z_p[\theta]^{\times}]}{[\OOO_{E}:\Z_p[\theta]]}\geq 1$,
since
$\big|\Frac^0(\Z_p[\theta])/P(\Z_p[\theta])\big|\geq 1$
and the rest of the sum is non-negative.

To show this, let $\mathfrak{f}\subseteq \Z_p[\theta]$ denote the conductor of $\Z_p[\theta]$.
Then
$\ppp/\mathfrak{f}\subseteq \OOO_E/\mathfrak{f}$
 is the unique maximal ideal so that 
$(\ppp\cap \Z_p[\theta])/\mathfrak{f}$
is the unique maximal ideal in $\Z_p[\theta]/\mathfrak{f}$. Hence
\begin{align*}
\#(\OOO_E/\mathfrak{f})^{\times}
& =\#(\OOO_E/\mathfrak{f})-\#(\ppp/\mathfrak{f})
=\#(\OOO_E/\mathfrak{f})(1-q^{-1}),\;\;\;\text{ and }\\
\#(\Z_p[\theta]/\mathfrak{f})^{\times}
& =\#(\Z_p[\theta]/\mathfrak{f})(1-(\#(\Z_p[\theta]/(\Z_p[\theta]\cap\ppp)))^{-1}).
\end{align*}
But since
$\Z_p[\theta]/(\Z_p[\theta]\cap\ppp)\hookrightarrow \OOO_E/\ppp$
is injective, we altogether get
\[
\frac{[\OOO_{E}^{\times}:\Z_p[\theta]^{\times}]}{[\OOO_{E}:\Z_p[\theta]]}
=\frac{1-q^{-1}}{1-(\#(\Z_p[\theta]/(\Z_p[\theta]\cap\ppp)))^{-1}}
\geq1.
\]

\item This is is a direct consequence of the explicit form of the orbital integral from  Lemma \ref{computing_local_non_archim_integral}.
 \end{enumerate}
\end{proof}

\section{An asymptotic for orbital integrals}\label{section_asymptotic}
From now let $G=\GL_3$ and $\ggG=\mathfrak{gl}_3$. 
The aim of this section is to prove a density result for orbital integrals, namely Proposition \ref{asymptotics_orbital_integrals} below.
If $\gamma\in G(\Q)_{\text{er}}$, we take the product measure on $G_{\gamma}(\A)=\prod_{p\leq \infty} G_{\gamma}(\Q_p)$ with local measures as in the previous section. Let $|\cdot|_E:\A_E^{\times} \longrightarrow\R_{>0}$ denote the adelic norm. Using the exact sequence $1\longrightarrow\A_E^1\hookrightarrow\A_E^{\times}\xrightarrow{|\cdot|_E} \R_{>0}\longrightarrow 1$, we also fix a  measure on $\A_E^1$.
With this choice of normalisation of measures we get
\[
 \vol(\R_{>0}G_{\gamma_{\xi}}(\Q)\backslash G_{\gamma_{\xi}}(\A))=\vol( E^{\times}\backslash \A_E^1)
=\rho_E |D_E|^{\frac{1}{2}}
\]
 for every $\xi\in E$ with $\Q(\xi)=E$, where
\[
\rho_E=\res_{s=1}\zeta_E(s).
\]
For a cubic field $E$ the set of $\xi\in E$ generating $E$ over $\Q$ is exactly $E\minus\Q$, as $E$ does not have non-trivial subfields.
For $\Phi\in\SSS^{\nu}(\ggG(\A))$, we therefore have
\begin{equation}\label{elliptic_elements_ordered_by_fields}
\Xi_{\text{main}}(s, \Phi)=\sum_{E\in\FFF_3}\frac{\rho_E}{|\Aut(E/\Q)|}
|D_E|^{\frac{1}{2}}\sum_{\xi\in E\minus\Q}\int_{0}^{\infty} \int_{G_{\gamma_{\xi}}(\A)\backslash G(\A)}
\lambda^{3s+3}\Phi(\lambda g^{-1}\gamma_{\xi} g)d^{\times}\lambda ~dg.
\end{equation}

Let $\FFF_3^+\subseteq\FFF_3$ be the set of all totally real cubic number fields, and $E\in\FFF_3^+$.
Let $Q_E:\OOO_E/\Z\longrightarrow\R$ be the positive definite quadratic form $Q_E(\xi)=\tr_{E/\Q}\xi^2-\frac{1}{3}(\tr_{E/\Q}\xi)^2$. We denote its successive minima by $m_1(E)\leq m_2(E)$, and its discriminant by $\Delta(Q_E)$. Similarly, $Q:\ggG(\A)\longrightarrow\A$ denotes the quadratic form on the matrices given by $Q(x)=\tr x^2-\frac{1}{3}(\tr x)^2$.
\begin{rem}
\begin{enumerate}[label=(\roman{*})] 
\item 
$m_1(E)$ is the second successive minimum of  $\xi\mapsto \tr_{E/\Q} \xi^2$ on $\xi\in\OOO_E$.
\item
$3\Delta(Q_E)=D_E$.
\end{enumerate}
\end{rem}
 
\begin{proposition}\label{asymptotics_orbital_integrals}
Let $\Phi_f\in\mathcal{S}(\ggG(\A_f))$ be supported in $\ggG(\widehat{\Z})$,  and suppose that $c(\Phi_f, \gamma+a)=c(\Phi_f, \gamma)$ for all $\gamma\in G(\Q)$ and $a\in\Z$. 
Then
\begin{equation}\label{asymptotic_tot_real}
\sum_{E\in \FFF_3^+ }\frac{\rho_E}{|\Aut(E/\Q)|}
\sum_{\substack{ \xi\in\OOO_E/\Z, \xi\neq0 \\ Q_E(\xi)\leq X }}c(\Phi_f,\xi)
=\beta(\Phi_f) X^{\frac{5}{2}}+o(X^{\frac{5}{2}})
\end{equation}
for $X\rightarrow \infty$, and $\beta(\Phi_f)$ is a certain constant depending on $\Phi_f$ with $\beta(\Phi_f^0)\neq0$.
\end{proposition}
The proof of this proposition will occupy the rest of this section.
\begin{rem}
 \begin{enumerate}[label=(\roman{*})]
\item The constraint on the support of $\Phi_f$ is not essential, it only changes the lattices in $E$ one has to sum over.

\item It is possible to find an analogue of the asymptotic \eqref{asymptotic_tot_real} also for fields with a complex place. However, one has to replace $Q_E$, since $Q_E$ is no longer positive definite if $E$ has a complex place.
  \end{enumerate}
\end{rem}

\subsection{Test functions}
We want to use the analytic properties of $\Xi_{\text{main}}(s,\Phi)$ to prove the proposition, hence our first task is to find test functions which separate the totally real fields from the rest.
To this end, we first construct two sequences of test functions at the archimedean places.
Let 
$\psi_{\eps}^{\pm}:\R\rightarrow \R_{\geq 0}$
be  smooth non-negative functions satisfying
\begin{align*}
\psi_{\eps}^+(x)=0
~~~~
&\text{ if }
~~~~
x<\frac{\eps}{2},
\;\;\;\;
0\leq \psi_{\eps}^+(x)\leq 1
~~~~
\text{ if }
~~~~
\frac{\eps}{2}\leq x\leq\eps,
\;\;\;\text{ and }
\psi_{\eps}^+(x)=1
~~~~
\text{ if }
~~~~
x>\eps,\\
\psi_{\eps}^-(x)=0
~~~~
&\text{ if }
~~~~
|x|>\eps,
~~~~
\text{ and }
~~~
0\leq \psi_{\eps}^-(x)\leq1,
~~~~
\text{ if }
~~~~
|x|\leq \eps,
\end{align*}
and
\[
1\leq\psi_{\eps}^+(x)+\psi_{\eps}^-(x)\leq 2
~~~~~~
\text{ if }
x>0.
\]
Define functions
$\Psi_{\eps}^{\pm}:\ggG(\R)\longrightarrow \R$
by
\[
\Psi_{\eps}^{\pm}(x)
=\psi_{\eps}^{\pm}\bigg(\frac{\Delta(x-\frac{1}{3}\tr x \One_3)}{|\tr x^2-\frac{1}{3}(\tr x)^2|^3}\bigg)
=\psi_{\eps}^{\pm}\bigg(\frac{\Delta(x-\frac{1}{3}\tr x \One_3)}{|Q(x)|^3}\bigg).
\]
These functions are well-defined and continuous, since $\psi_{\eps}^{\pm}$ is compactly supported. Moreover, away from the set of $x$ with $Q(x)=\tr x^2-\frac{1}{3}(\tr x)^2=0$ they are smooth.

For $x\in\ggG(\R)$ and large $N\in \N$ put
\[
\Phi_{\infty}^{\eps,\pm}(x)
=\psi_{\eps}^{\pm}\Big(\frac{\Delta(x-\frac{1}{3}\tr x\One_3)}{|Q(x)|^3}\Big)Q(x)^Ne^{-\pi \tr x^tx}
=\Psi_{\eps}^{\pm}(x)Q(x)^N e^{-\pi \tr x^tx}.
\]
For given $\nu\in\N$, we can choose  $N$ large enough such that $\Phi_{\infty}^{\eps, \pm}\in \SSS^{\nu}(\ggG(\R))$.
The properties of $\Phi_{\eps}^{\pm}$ can be summarised as follows.
\begin{lemma}\label{properties_special_test_functions}
For all $x\in\ggG(\R)$, $g\in G(\R)$, and $\lambda\in\R_{>0}$, we have 
\begin{enumerate}[label=(\roman{*})]
 \item 
$\Phi_{\infty}^{\eps,\pm}(\Ad g^{-1} x)=\Phi_{\eps}^{\pm}(x)$. In particular, we may write $\Phi_{\infty}^{\eps,\pm}(\xi)=\Phi_{\infty}^{\eps,\pm}(\gamma_{\xi})$ for every $\xi\in E$ and $E\in\FFF_3$.
 \item
$\Phi_{\infty}^{\eps,\pm}(\lambda x)=\Phi_{\infty}^{\eps,\pm}(x)$.
\item
$\Phi_{\infty}^{\eps,\pm}(x+\lambda\One_3)=\Phi_{\infty}^{\eps,\pm}(x)$.
\item\label{cpl:vanishing}
$\Phi_{\infty}^{\eps,+}(\lambda \Ad g^{-1} x)= 0$
if $x$ has a non-real eigenvalue.
\end{enumerate}
\end{lemma}

If we fix $\Phi_f\in\SSS(\ggG(\A_f))$ as in Proposition \ref{asymptotics_orbital_integrals}, we define test functions  $\Phi^{\eps, +}=\Phi_{\infty}^{\eps, +}\Phi_f$ and $\Phi^{\eps, -}=\Phi_{\infty}^{\eps, -}\Phi_f$. They implicitly depend on the integer $N$, and $\Phi^{\eps, \pm}\in\SSS^{\nu}(\ggG(\A))$ with $\nu$ depending on $N$.

By Lemma \ref{properties_special_test_functions} \ref{cpl:vanishing} we have 
$I_{\infty}(\Phi^{\eps, +}, \gamma)=0$ if $\gamma\in G(\Q)_{\text{er}}$ is not diagonisable over $G(\R)$.
Hence for the test function $\Phi^{\eps, +}$ only totally real fields contribute to $\Xi_{\text{main}}(s, \Phi^{\eps,+})$, i.e. we get
\begin{multline}\label{sum_of_positive_elliptic_terms}
\EEE^{+}_{\eps}(s):=\Xi_{\text{main}}(s, \Phi^{\eps, +})
=\sum_{E\in \FFF_3^+}\frac{\vol(E^{\times}\backslash \A_E^1)}{|\Aut(E/\Q)|}
\sum_{\xi\in\OOO_E\minus\Z}[\OOO_E:\Z[\xi]]c(\Phi_f, \xi)\Psi^{+}_{\eps}(\xi)\cdot\\
\bigg(\int_{0}^{\infty}\int_{G_{\gamma_{\xi}}(\R)\backslash G(\R)}
\lambda^{s}(\lambda^2Q_E(\xi))^N e^{-\pi\lambda^2\tr (x^{-1}\gamma_{\xi}x)^t(x^{-1}\gamma_{\xi} x)}~d^{\times}\lambda~ dx\bigg).
\end{multline}
Similarly, we set $\EEE^{-}_{\eps}(s)=\Xi_{\text{main}}(s, \Phi^{\eps, -}) $.

\begin{rem}
Separating the totally real fields from the rest is more complicated in the cubic than in the quadratic case. This is due to the absence of a prehomogoneous vector space structure so that there are infinitely many orbits under the action of $\GL_1\times \GL_3$ on $\ggG(\A)$. 
\end{rem}

\begin{lemma}\label{properties_of_real_elements_dirichlet_series}
There exists $N>0$ such that the following holds. Let $\Phi_f$ is as in Proposition \ref{asymptotics_orbital_integrals}. Then $\mathcal{E}^{+}_{\eps}(s)$ is holomorphic for $\Re s>2$, and has a meromorphic continuation at least in $\Re s>3/2$ with only singularity at $s=2$, which is a simple pole.
Moreover, for $\Re s>2$ the function $\mathcal{E}^+_{\eps}(s)$ equals up to an entire function the series
\begin{equation}\label{dirichlet_series_over_totally_real_fields}
I_N(s)\sum_{E\in\FFF_3^+}\frac{\rho_E}{|\Aut(E/\Q)|}
\sum_{\substack{\xi\in\OOO_E/\Z : \\ \xi\neq0 }}c(\Phi_f, \xi)\Psi_{\eps}^+(\xi)Q_E(\xi)^{-\frac{3s-1}{2}},
\end{equation}
where for $\Re s>0$ 
\[
I_N(s)
=\frac{1}{\sqrt{3\pi}}\int_{0}^{\infty}\lambda^{3s-1+2N} e^{-\pi\lambda^2}d^{\times}\lambda
=\frac{1}{\sqrt{3\pi}}\frac{\Gamma(\frac{3s+2N-1}{2})}{2\pi^{\frac{3s+N-1}{2}}}.
\]
\end{lemma}
\begin{proof}
The first assertion follows from Theorem \ref{main_thmintr} and Proposition \ref{holom_ct_intr}.
Let $E\in\FFF_3^+$, and consider the map $\OOO_E\longrightarrow \Z\oplus \OOO_E/\Z$, $\xi\mapsto (\tr \xi,\xi+\Z)$, which is a group isomorphism. 
As the coefficients $c(\Phi_f, \cdot)$ and the function $\Psi_{\eps}^{+}$ are well-defined maps on $\OOO_E/\Z$, the inner sum for $E$ in \eqref{sum_of_positive_elliptic_terms} equals
\begin{multline*}
\sum_{\substack{\xi_0\in\OOO_E/\Z: \\ \xi_0\neq0 }}[\OOO_E:\Z[\xi_0]]c(\Phi_f, \xi_0)\Psi_{\eps}^+(\xi_0)\cdot\\
\bigg(\int_{0}^{\infty}\int_{G_{\gamma_{\xi_0}}(\R)\backslash G(\R)}\lambda^{3s+3+2N}
Q_E(\xi_0)^N \sum_{a\in\Z}e^{-\pi\lambda^2\tr (x^{-1}\gamma_{\xi_0}x)^t(x^{-1}\gamma_{\xi_0} x)-3\pi\lambda^2a^2}~d^{\times}\lambda~ dx\bigg).
\end{multline*}
We split the integral over $\lambda$ in one integral over $\lambda\in[0, 1]$ and one over $\lambda\in[1, \infty)$. 
The sum over all $E$ of the second integral defines an entire function on all of $\C$ so that we may ignore it. 
For the sum over the first one we apply Poisson summation  to  the sum over $a\in\Z$, to obtain
\[
\sum_{a\in\Z}e^{-3\pi\lambda^2a^2}
=\sum_{b\in \Z}(3\pi)^{-\frac{1}{2}}\lambda^{-1}e^{-3\pi^{-1}\lambda^{-2}b^2}.
\]
Changing variables $\lambda^{-1}\in[0, 1]\leftrightarrow \lambda\in [1, \infty)$, the sum over $b\neq 0$ yields again an entire function which we can ignore. 
Hence we are left with the term belonging to $b=0$. We may add the integral over $\lambda\in[1, \infty)$ without changing its analytic behaviour.
Thus up to an entire function, $\mathcal{E}_{\eps}^{+}(s)$ equals
\begin{multline*}
\frac{1}{\sqrt{3\pi}}\sum_{E\in\FFF_3^+}\frac{\vol(E^{\times}\backslash \A_E^1)}{|\Aut(E/ \Q)|}
\sum_{\substack{\xi_0\in\OOO_E/\Z: \\ \xi_0\neq0}}[\OOO_E:\Z[\xi]]c(\Phi_f, \xi)\Psi_{\eps}^{+}(\xi)\cdot\\
\bigg(\int_{0}^{\infty}\int_{G_{\gamma_{\xi_0}}(\R)\backslash G(\R)}\lambda^{3s+2+2N}
Q_E(\xi_0)^N e^{-\pi\lambda^2\tr (x^{-1}\gamma_{\xi_0}x)^t(x^{-1}\gamma_{\xi_0} x)}~d^{\times}\lambda ~dx\bigg).
\end{multline*}
As $E$ is totally real, for every $\xi_0\in \OOO_E/\Z$, the matrix $\gamma_{\xi_0}$ is over $G(\R)$ conjugate to a diagonal matrix (with pairwise distinct eigenvalues) so that  
\begin{multline*}
 \int_{G_{\gamma_{\xi_0}}(\R)\backslash G(\R)} e^{-\pi\lambda^2\tr (x^{-1}\gamma_{\xi_0}x)^t(x^{-1}\gamma_{\xi_0} x)} ~dx\\
=\Delta(\xi_0)^{-\frac{1}{2}}e^{-\pi\lambda^2Q_E(\xi_0)} \int_{U_0(\R)}e^{-\pi \lambda^2(u_1^2+u_2^2+u_3^2)}~du
=\Delta(\xi_0)^{-\frac{1}{2}}e^{-\pi\lambda^2Q_E(\xi_0)}\lambda^{-3}.
\end{multline*}
Notice that 
$\Delta(\xi_0)^{-\frac{1}{2}}=[\OOO_{E}:\Z[\xi_0]]^{-1}D_E^{-\frac{1}{2}}$
and
$\vol(E^{\times}\backslash \A_E^1)D_E^{-\frac{1}{2}}=\res_{s=1}\zeta_E(s)=\rho_E$.
Hence changing $\lambda$ to $Q_E(\xi_0)^{\frac{1}{2}}\lambda$, the assertion follows upon defining $I_N$ as described.
\end{proof}

\begin{lemma}\label{properties_of_complex_elements}
There exists $N>0$ such that the following holds.
Let $\Phi_f$ is as in Proposition \ref{asymptotics_orbital_integrals}. 
Then  $\mathcal{E}^{-}_{\eps}(s)$ is holomorphic for $\Re s>2$ and continues to a meromorphic function at least in $\Re s>3/2$ with only pole at $s=2$ which is simple. 
Up to an entire function (defined on all of $\C$), $\mathcal{E}^{-}_{\eps}(s)$ equals for $\Re s>2$ the sum of
\[
 I_N(s)\sum_{E\in\FFF_3^+}\frac{\rho_E}{|\Aut(E/\Q)|}
\sum_{\substack{\xi\in\OOO_E/\Z: \\ \xi\neq0}} c(\Phi_f, \xi)\Psi_{\eps}^{-}(\xi)Q_E(\xi)^{-\frac{3s-1}{2}}
\]
and
\[
4\sqrt{\frac{\pi}{3}}\frac{\Gamma(\frac{3s+2l}{2})}{\pi^{\frac{3s+2l}{2}}}\sum_{E\in\FFF_3\minus\FFF_3^+}\rho_E
\sum_{\substack{\xi\in\OOO_E/\Z: \\ \xi\neq0}}c(\Phi_f, \xi)\Psi_{\eps}^{-}(\xi)J_N(\xi, s)Q_E(\xi)^N,
\]
where
\[
J_N(\xi, s)=\int_{1}^{\infty}(Q_E(\xi)+4(\Im \tilde{\xi})^2\rho^2)^{-\frac{3s+2N}{2}}d\rho,
\]
and $\tilde{\xi}$ denotes one of the two non-real conjugates of $\xi\in E\minus\Q$ if $E\in \FFF_3\minus\FFF_3^+$.
\end{lemma}
\begin{proof}
Again, the first assertion is given by Theorem \ref{main_thmintr} and Proposition \ref{holom_ct_intr}. 
 Similarly as in the proof of Lemma \ref{properties_of_real_elements_dirichlet_series}, $\mathcal{E}^{-}_{\eps}(s)$ can be written as the sum over all cubic fields $E\in\FFF$ (now of any signature) of
\begin{multline*}
 \frac{\vol(E^{\times}\backslash \A_E^1)}{|\Aut(E/\Q)|}\sum_{\substack{\xi_0\in\OOO_E/\Z: \\ \xi_0\neq0}}[\OOO_E:\Z[\xi_0]]c(\Phi_f, \xi_0)\Psi_{\eps}^-(\xi_0)\cdot\\
\bigg(\int_{0}^{\infty}\int_{G_{\gamma_{\xi_0}}(\R)\backslash G(\R)}\lambda^{3s+3+2N}Q_E(\xi_0)^N
\sum_{a\in\Z}e^{-\pi\lambda^2\tr (x^{-1}\gamma_{\xi_0}x)^t(x^{-1}\gamma_{\xi_0} x)-\frac{\pi}{3}\lambda^2a^2}~d^{\times}\lambda~ dx\bigg).
\end{multline*}
For totally real extensions, the proof of the last lemma tells us that the respective sum essentially (i.e., up to an entire function) equals
\[
I_N(s)\sum_{E\in\FFF_3^+}\frac{\rho_E}{|\Aut(E/\Q)|} \sum_{\substack{\xi\in\OOO_E/\Z: \\ \xi\neq0}}c(\Phi_f, \xi)\Psi_{\eps}^-(\xi)Q_E(\xi)^{-\frac{3s-1}{2}},
\]
with $I_N(s)$ defined as before.

For $E\in\FFF_3\minus\FFF_3^+$ and $\xi_0\in\OOO_E/\Z$, $\xi_0\neq0$, we can follow along the same lines. However, the integral
$\int_{G_{\gamma_{\xi_0}}(\R)\backslash G(\R)}e^{-\pi\lambda^2\tr (x^{-1}\gamma_{\xi_0} x)^t(x^{-1}\gamma_{\xi_0} x)}~dx$
now equals 
\[
8\pi\lambda^{-2} |\Delta(\xi)|^{-\frac{1}{2}}
\int_{2|\Im \tilde{\xi}|}^{\infty}e^{-\pi\lambda^2(Q_E(\xi)+\rho^2)}~d\rho,
\]
where $\tilde{\xi}\in\C$ denotes one of the two non-real conjugates of $\xi$.
Changing $(Q_E(\xi)+\rho^2)^{\frac{1}{2}}\lambda$ to $\lambda$, we obtain for the double integral
\[
8\pi|\Delta(\xi)|^{-\frac{1}{2}}Q_E(\xi_0)^N\int_{0}^{\infty}\lambda^{3s+2N}e^{-\pi\lambda^2}d^{\times}\lambda
\int_{2|\Im\tilde{\xi}|}^{\infty}(Q_E(\xi)+\rho^2)^{-\frac{3s+2N}{2}}~d\rho
\]
from which the assertion follows.
\end{proof}

\subsection{Dirichlet series}
To study the Dirichlet series obtained in the last section and to finish the proof of Proposition \ref{asymptotics_orbital_integrals}, we need to define a few more auxiliary functions.
$N>0$ denotes a sufficiently large integer such that Lemmas \ref{properties_of_real_elements_dirichlet_series} and \ref{properties_of_complex_elements} hold. For $t\in\C$ with $\Re t>5/2$ set
\begin{align*}
\alpha_{\eps}^{\pm}(t)
& =\sum_{E\in\FFF_3^+}\frac{\rho_E}{|\Aut(E/\Q)|}
\sum_{\substack{\xi\in\OOO_E/\Z: \\ \xi\neq0}}c(\Phi_f, \xi)\Psi_{\eps}^{\pm}(\xi)Q_E(\xi)^{-t}, \text{ and }\\
A_{\eps}^{\pm}(X)
& =\sum_{E\in\FFF_3^+}\frac{\rho_E}{|\Aut(E/\Q)|}
\sum_{\substack{\xi\in\OOO_E/\Z, \xi\neq0 \\ Q_E(\xi)\leq X }}c(\Phi_f, \xi)\Psi_{\eps}^{\pm}(\xi)
\end{align*}
(these are both independent of $N$).
Then by Lemmas \ref{properties_of_real_elements_dirichlet_series} and \ref{properties_of_complex_elements} (as $I_N(\frac{2t+1}{3})$ is holomorphic and non-vanishing in all of $\Re t>7/4$), the series defining $\alpha_{\eps}^{\pm}$ converge absolutely in $\Re t>5/2$, can be meromorphically continued up to $\Re t>7/4$, and each has in this half plane only one pole which is located at $t=5/2$, and is simple with residue
\[
 \rho_{\eps}(\Phi_f):=
\frac{3}{2}I_N(2)^{-1}\res_{s=2}\EEE^{\pm}_{\eps}(s).
\]
The functions are related by the Mellin transformation and its inverse (cf.\ \cite[\S 5]{MoVa07}): We have for $\sigma_0\gg 0$
\begin{align*}
A_{\eps}^{\pm}(X) & =\frac{1 }{2\pi i}\int_{\sigma_0-i\infty}^{\sigma_0+i\infty}\alpha_{\eps}^{\pm}(t)\frac{X^{t}}{t}dt, \text{    and }\\
\alpha_{\eps}^{\pm}(t) &   =\int_{1}^{\infty}X^{-t}dA_{\eps}^{\pm}(X).
\end{align*}
Further define
\begin{align*}
\gamma_{\eps}(t)
& =\sum_{E\in\FFF_3\minus\FFF_3^+}\rho_E
\sum_{\substack{\xi\in\OOO_E/\Z: \\  \xi\neq0}}c(\Phi_f, \xi)\Psi_{\eps}^{-}(\xi)J(\xi, \frac{2t+1}{3})Q_E(\xi)^N,\;\;\text{   and }\\
C_{\eps}(X)
& =\frac{1 }{2\pi i}\int_{\sigma_0-i\infty}^{\sigma_0+i\infty}\gamma_{\eps}(t)\frac{X^{t}}{t}dt\\
& =\sum_{E\in\FFF_3\minus\FFF_3^+}\rho_E
\sum_{\substack{\xi\in\OOO_E/\Z: \\ \xi\neq0}} c(\Phi_f, \xi)\Psi_{\eps}^{-}(\xi)Q_E(\xi)^N
\int_{1}^{b(\xi, X)}(Q_E(\xi)+4(\Im \tilde{\xi})^2\rho^2)^{-N-\frac{1}{2}}~d\rho,
\end{align*}
where
\[
b(\xi, X)=\begin{cases}
           \max\{1,\frac{\sqrt{X-Q_E(\xi)}}{2|\Im \tilde{\xi}|}\}				&\text{ if } Q_E(\xi)\leq X,\\
	    1											&\text{ if } Q_E(\xi)>X.
          \end{cases}
\]
This definition together with the definition of $\Psi_{\eps}^-(\xi)$ ensures that for every $X$, the sum over $E$ and $\xi$ is in fact finite.
From the last expression of $C_{\eps}(X)$, it is clear that if $N$ is even, $C_{\eps}(X)$ is a non-negative, monotonically increasing function in $X$.

\begin{proof}[Proof of Proposition \ref{asymptotics_orbital_integrals}]
We assume that $N$ is even and  sufficiently large such that Lemmas \ref{properties_of_real_elements_dirichlet_series} and \ref{properties_of_complex_elements} hold. By definition of $\Psi_{\eps}^+$ and $\Psi_{\eps}^-$ we have 
$\Psi_{\eps}^+(\xi)\leq 1\leq \Psi_{\eps}^+(\xi)+\Psi_{\eps}^-(\xi)$
 for all $\xi\in E$ if $E$ is totally real. Hence for every $X>0$, we get
\begin{equation}\label{chain_of_inequalties_for_asymptotic}
A_{\eps}^+(X)
\leq \sum_{E\in\mathcal{F}_3^+}\frac{\rho_E}{|\Aut(E/\Q)|}\sum_{\substack{\xi\in\OOO_E/\Z, \xi\neq0 \\ Q_E(\xi)\leq X}}c(\Phi_f,\xi)
=:\Sigma(X)
\leq A_{\eps}^+(X)+A_{\eps}^-(X).
\end{equation}
 The coefficients $\frac{\rho_E}{|\Aut(E/\Q)|}c(\Phi_f, \xi)\Psi_{\eps}^+(\xi)$ in the Dirichlet series $\alpha_{\eps}^+(t)$ are non-negative.
Hence the properties of $\alpha_{\eps}^+(t)$ stated above allow us to apply the Wiener-Ikehara Tauberian Theorem \cite[Corollary 8.7]{MoVa07}. This yields the asymptotic 
\[
A_{\eps}^+(X)
\sim\rho_{\eps}(\Phi_f)X^{\frac{5}{2}}+o(X^{\frac{5}{2}})
\]
as $X\rightarrow\infty$.
Therefore,
\[
\liminf_{X\rightarrow\infty}X^{-\frac{5}{2}}\Sigma(X)\geq \rho_{\eps}(\Phi_f)
\]
for every $\eps>0$ so that
\[
\liminf_{X\rightarrow\infty}X^{-\frac{5}{2}}\Sigma(X)\geq \rho_0(\Phi_f),
\]
where 
\[
\rho_0(\Phi_f)
=\frac{2\pi^{9/2}\zeta(3)}{\sqrt{3}}\int_{x\in \ggG(\R):~\Delta(x)>0}e^{-\pi\tr x^tx} dx\int_{\ggG(\A_f)} \Phi_f(x_f) dx_f,
\]
since $\rho_{\eps}(\Phi_f)\rightarrow \rho_0(\Phi_f)$ for $\eps\searrow0$. 

To show the reverse inequality, we have to work harder.
Consider now the function $\mathcal{E}^{-}_{\eps}(\frac{2t+1}{3})$. It has a simple pole at $t=5/2$, and is holomorphic elsewhere in some half plane $\Re s>7/4$.
As
$4\sqrt{3\pi}\frac{\Gamma(t+N+\frac{1}{2})}{\pi^{t+N+\frac{1}{2}}}$
is holomorphic and non-zero in that half plane, the function
\[
\frac{\pi^{t+N+\frac{1}{2}}}{4\sqrt{3\pi}\Gamma(t+N+\frac{1}{2})}\mathcal{E}_{\eps}^{-}(\frac{2t+1}{3})
=\frac{1}{8\sqrt{\pi}}\frac{\Gamma(t+N)}{\Gamma(t+N+\frac{1}{2})}\alpha_{\eps}^{-}(t)
+\gamma_{\eps}(t)
=\frac{1}{8\pi}\beta_N(t)\alpha_{\eps}^{-}(t)
+\gamma_{\eps}(t)
\]
has the same properties as $\mathcal{E}_{\eps}^{-}$ where
\[
\beta_N(t)
=\int_{\R}(1+x^2)^{-(t+N+\frac{1}{2})}dx
=2\int_{1}^{\infty}y^{-(t+N+\frac{1}{2})}d\sqrt{y-1}.
\]
The residue $\rho^-_{\eps}(\Phi_f)$ at $t=5/2$ is given by a constant multiple of 
\[
 \int_{\ggG(\R)}\Phi_{\infty}^{\eps, -}(x)~dx \int_{\ggG(\A_f)} \Phi_f(x_f)~dx_f,
\]
which tends to $0$ as $\eps\searrow0$.

For $X>0$ and $\sigma_0\gg0$ sufficiently large, let
\begin{align*}
B_N(X) & =\frac{1}{2\pi i}\int_{\sigma_0-i\infty}^{\sigma_0+i\infty}\beta_N(t)\frac{X^{t}}{t}dt, \text{    and }\\
AB_{N, \eps}(X) & =\frac{1}{2\pi i}\int_{\sigma_0-i\infty}^{\sigma_0+i\infty}\beta_N(t)\alpha_{\eps}^{-}(t)\frac{X^{t}}{t}dt.
\end{align*}
In particular,
\[B_N(X)=2\int_{1}^{X}y^{-(N-1)}d\sqrt{y-1}.\]
From the definitions it is clear that $C_{\eps}(X)\geq0$, $B_N(X)\geq0$, and $AB_{N,\eps}(X)\geq0$, and the functions are monotonically increasing.
Hence an application of the Wiener-Ikehara Theorem gives
$\lim_{X\rightarrow\infty}X^{-\frac{5}{2}}(AB_{N, \eps}(X)+C_{\eps}(X))=\rho_{\eps}^-(\Phi_f)$, 
and, as everything is non-negative, 
$AB_{N, \eps}(X)\leq \rho_{\eps}^-(\Phi_f)X^{\frac{5}{2}}+R_{\eps}(X)$,
where $R_{\eps}(X)$ is a suitable error function with $R_{\eps}(X)\rightarrow 0$ as $X\rightarrow\infty$.
Therefore, 
\[
X^{\frac{5}{2}}\rho_{\eps}^-(\Phi_f)+R_{\eps}(X)
\geq \frac{1}{2\pi i}\int_{\sigma_0-i\infty}^{\sigma+i\infty}\beta_N(t)\alpha_{\eps}(t)\frac{X^{t}}{t}dt
\]
and the right hand side can be written as
\begin{multline*}
\frac{1}{2\pi i}\int_{\sigma_0-i\infty}^{\sigma+i\infty}\alpha_{\eps}^-(t)
\left(\int_{1}^{\infty}v^{-t}dB_l(v)\right)\frac{X^{t}}{t}dt\\
=\int_{1}^{\infty}\left(\frac{1}{2\pi i}\int_{\sigma_0-i\infty}^{\sigma+i\infty}\alpha_{\eps}^-(t)\left(\frac{X}{v}\right)^{t}\frac{dt}{t}\right) dB_N(X)
=\int_{1}^{\infty}A_{\eps}^-(\frac{X}{v})dB_N(v).
\end{multline*}
As $A_{\eps}^-$ is monotonically increasing, the last integral is bounded from below by
\[
\geq \int_{2}^{3}A_{\eps}^-(\frac{X}{v})dB_N(v)
\geq A_{\eps}^-(\frac{X}{3})\int_{2}^{3}dB_N(v)>0
\]
for all $X> 0$. Hence there exists a constant $c>0$ such that for every $\eps>0$, we have 
$\limsup_{X\rightarrow \infty}X^{-\frac{5}{2}}A_{\eps}^-(X)\leq c\rho_{\eps}^-(\Phi_f)$,
and thus
\[
\limsup_{X\rightarrow\infty}X^{-\frac{5}{2}}A_{\eps}^-(X)\longrightarrow 0
~~~~~~
\text{ for }
\eps\searrow0.
\]
Hence
\[
\limsup_{X\rightarrow\infty}X^{-\frac{5}{2}}\Sigma(X)
= \limsup_{X\rightarrow\infty}X^{-\frac{5}{2}} A_{\eps}^+(X)
+\limsup_{X\rightarrow\infty}X^{-\frac{5}{2}} A_{\eps}^-(X)
\leq \rho_{\eps}^+(\Phi_f) +c\rho_{\eps}^-(\Phi_f)\longrightarrow \rho_0(\Phi_f)
\]
for $\eps\searrow0$, which finishes the proof of the asymptotic. 
\end{proof}

\section{Bounds for mean values of residues of Dedekind zeta functions}\label{bounds_for_residues}
We want to use the result from the last section to obtain information on the mean value of residues of Dedekind zeta functions.
As $c(\Phi_f^0,\xi)\geq 1$ for all $\xi\in E\minus\Q$ and all $E\in\FFF_3$ by Proposition \ref{orb_int_greaterequal1}, an immediate consequence of Proposition \ref{asymptotics_orbital_integrals} is the following upper bound.
\begin{theorem}\label{asymptotic_real_fields}
There exists $\alpha<\infty$ such that
\begin{equation}\label{residues_ordered_by_minima}
\limsup_{X\rightarrow \infty}X^{-\frac{5}{2}}\sum_{\substack{ E\in\mathcal{F}_3^+: \\ m_1(E)\leq X }}\res_{s=1}\zeta_E(s)
\leq \alpha.
\end{equation}
\end{theorem}
To complement this upper bound we show the following lower bound.
\begin{proposition}\label{lower_bound_for_residues}
We have for every $\eps>0$, we have 
\[
\liminf_{X\rightarrow\infty} X^{-\frac{5}{2}+\eps}\sum_{\substack{E\in\mathcal{F}_3^+: \\ m_1(E)\leq X}}\res_{s=1}\zeta_E(s) 
=\infty.
\]
\end{proposition}
In fact, Conjecture \ref{conj_intr} is expected to be true.
The proof of this proposition is of a complete different nature than the proof of Theorem \ref{asymptotic_real_fields}: Basically we will show that there are sufficiently many irreucible cubic polynomials, cf.\ also the introduction where a relation to \cite[Remark 3.3]{ElVe06} is explained. Ultimately, one hopes that Proposition \ref{lower_bound_for_residues} (and even Conjecture \ref{conj_intr}) can also be deduced from Propositon \ref{asymptotics_orbital_integrals}, cf.\ Appendix \ref{appendixb} for a sequence of test functions that might be useful.
We need the following auxiliary result to prove Proposition  \ref{lower_bound_for_residues}:
\begin{lemma}\label{estimates_part_of_cubic_fields}
 \begin{enumerate}[label=(\roman{*})]
  \item\label{estimates_cubic1} Let $Q:\R^2\longrightarrow\R$ be a positive definite quadratic form with discriminant $\Delta(Q)$ and first successive minimum $m_1(Q)\geq1$. Then, as $X\rightarrow\infty$, we have
\[
\sum_{\substack{\gamma\in\Z^2: \\ Q(\gamma)\leq X}}1=\frac{2\pi X}{\sqrt{\Delta(Q)}}+ O\bigg(\sqrt{\frac{m_1(Q)}{\Delta(Q)}}X^{\frac{1}{2}}\bigg)
\]
with implied constant independent of $Q$.

\item\label{estimates_part_of_cubic_fields2} For all $\eps>0$, we have as $X\rightarrow\infty$
\[
\sum_{\substack{E\in\mathcal{F}_3^+: \\ m_2(E)\leq X}}\rho_E \sum_{\substack{\xi\in \OOO_E/\Z, \xi\neq0, \\ Q_E(\xi)\leq X }}1=O(X^{2+\eps}).
\]
 \end{enumerate}
\end{lemma}
\begin{proof}
 \begin{enumerate}[label=(\roman{*})]
  \item We need to count all points in $\Z^2$ which are contained in the ellipse $E_X:=\{x\in\R^2\mid Q(x)\leq X\}$.
 By a  theorem of Gauss \cite[p.161]{Co62}, the number of such points  is equal to the area $\frac{2\pi X}{\sqrt{\Delta(Q)}}$ of the ellipse $E_X$ plus some small error term of order $R X^{\frac{1}{2}}$ for $R$ the length of the major axis of the ellipse $E_1$ and all implicit constants independent of $Q$. Since $m_1(Q)\geq1$, it is easily verified that $R\leq \sqrt{\frac{m_1(Q)}{\Delta(Q)}}$ finishing the proof of the assertion.

\item 
By Minkowski's second theorem (see, e.g. \cite[VIII.4.3]{Ca97}),  there are $a_1, a_2>0$ such that  for all cubic fields $E$,
$a_1m_1(E)m_2(E)\leq D_E \leq a_2 m_1(E)m_2(E)$ 
so that $m_1(E)\leq m_2(E)\leq X$ implies
$c_0 D_E\leq m_1(E)m_2(E)\leq 16X^2$
for some $c_0>0$, and moreover, $m_1(E)/\Delta(Q_E)$ is bounded from above by an absolute constant.
Hence  there is by \ref{estimates_cubic1} some constant $C>0$ such that 
\[
\sum_{\substack{\xi\in\OOO_E/\Z, ~\xi\neq0, \\ Q_E(\xi)\leq X}}1\leq C \frac{X}{\sqrt{\Delta(Q_E)}}
\]
for all $E$ with $m_1(E)\leq m_2(E)\leq X$.
By the Brauer-Siegel Theorem \cite[XVI, \S 4 Theorem 4]{La86}, there exists for all $\eps>0$ some number $C_{\eps}>0$ such that
$\rho_E=\res_{s=1}\zeta_E(s)=4D_E^{-\frac{1}{2}}h_ER_E\leq C_{\eps}D_E^{\eps}$
for all totally real cubic fields $E$.
 Hence the left hand side of \ref{estimates_part_of_cubic_fields2} equals
\[
\sum_{\substack{E\in\FFF_3^+: \\ m_2(E)\leq X}}\rho_E \sum_{\substack{\xi\in \OOO_E \minus \Z: \\ Q_E(\xi)\leq X}} 1
\leq C C_{\eps}\sqrt{3}\sum_{E:~m_2(E)\leq X}X D_E^{\eps-\frac{1}{2}}.
\]
This can bounded by
\[
 C C_{\eps}\sqrt{3}X\sum_{E:~D_E\leq 16X^2} D_E^{\eps-\frac{1}{2}}
\leq C C_{\eps}\sqrt{3}X^{1+\eps}\sum_{E:~D_E\leq 16X^2}D_E^{-\frac{1}{2}}.
\]
By \cite[Theorem 1]{DaHe71} or \cite[Theorem I.1]{DaWr88}, $\sum_{E:~D_E\leq X}1=c_0 X+o(X)$ for some $c_0>0$ so that
\[
 C C_{\eps}\sqrt{3}X^{1+\eps}\sum_{E:~D_E\leq 16X^2}D_E^{-\frac{1}{2}}
\leq  16 c_0C C_{\eps}\sqrt{3}X^{2+\eps}+o(X^{2+\eps})
\]
which is the assertion.
 \end{enumerate}
\end{proof}

\begin{proof}[Proof of Proposition \ref{lower_bound_for_residues}]
It suffices to assume that $\eps\in (0,1/2)$.
 We first show that 
\begin{equation}\label{limes_inferior_all_vectors}
\liminf_{X\rightarrow\infty} X^{-\frac{5}{2}+\eps}
\sum_{E\in\FFF_3^+}
\sum_{\substack{\xi\in\OOO_E/\Z, \xi\neq0, \\ Q_{E}(\xi)\leq X}}\rho_E=\infty
\end{equation}
for every $\eps>0$. 
Let $\eps>0$. By the Brauer-Siegel Theorem  there exists $A_{\eps}>0$ such that
$\rho_E\geq A_{\eps}D_E^{-\frac{\eps}{2}}$
for all $E$.
Thus this sum is bounded from below by 
$A_{\eps}X^{-\frac{\eps}{2}}\sum_{E\in\FFF_3^+}N_E(X)$,
where $N_E(X):=\big|\{\xi\in\OOO_E/\Z: \xi\neq0,~Q_{E}(\xi)\leq X \}\big|$.
Hence it will certainly suffice to show that there exists $C>0$ such that
\[
\sum_{E\in\FFF_3^+}N_E(X)\sim C X^{\frac{5}{2}}
\]
as $X\rightarrow\infty$. 
The map associating with the pair
$E\in\FFF_3^+$,  $\xi\in\OOO_E/\Z$, $\xi\neq0$,
the characteristic polynomial $T^3+a_1T+a_0$ of $\xi-\frac{1}{3}\tr\xi \One_3$ is $3-1$ or $1-1$ depending on whether $E$ is Galois or not.
 As $E$ is totally real, we have 
$\Delta(\xi-\frac{1}{3}\tr\xi\One_3)=-4 a_1^3-27 a_0^2>0$,
or equivalently $a_0^2\leq -\frac{4}{27}a_1^3$. Since $X\geq Q_{E}(\xi)=-2a_1>0$, this implies
\begin{equation}\label{constraints_coeff}
-\frac{X}{2}\leq a_1<0
\;\;\text{   and   } \;\;
0<a_0
\leq \sqrt{-\frac{4}{27}a_1^3}
\leq \frac{1}{3\sqrt{6}}X^{\frac{3}{2}}.
\end{equation}
 Hence, ignoring constants, there are $a_1^{\frac{3}{2}}$ many $a_0$ and
\[
\int_{1}^{X/2}a_1^{\frac{3}{2}}da_1
=\frac{1}{10\sqrt{2}}X^{\frac{5}{2}}-\frac{2}{5}
\]
many $a_1$ satisfying all the conditions. 
On the other hand, any irreducible polynomial with integral coefficients satisfying the inequalities in \eqref{constraints_coeff} defines (a conjugacy class of) a cubic field $E$ and $\xi$ as before. Thus we only need to show that the reducible polynomials with coefficients satisfying above constraints do not contribute to $CX^{\frac{5}{2}}$. If
$T^3+a_1T+a_0$
 is reducible over $\Q$, we can write it as a product
$(T^2+b_1T+b_0)(T+c)$
with $b_1, b_0, c\in\Z$. Hence $c=-b_1$, $cb_0=a_0$ and $b_0-c^2=a_1$. Hence if we fix $a_0$ (for which there are at most $O(X^{\frac{3}{2}})$ possibilities), there are at most $O(a_0^{\delta})\leq O(X^{\delta})$ possibilities for $c$ and $b_0$ for any $ \delta>0$. Thus there are only $O(X^{\frac{3}{2}+\delta})$ reducible polynomials satisfying above constraints.
This finishes the proof of \eqref{limes_inferior_all_vectors}.

Now split the sum over $E$ in the following parts: One belonging to $E\in\FFF_3^+$ such that $m_1(E)>X$, one over $E$ such that $m_1(E)\leq X<m_2(E)$,  and the last one over $E$ such that $m_1(E)<m_2(E)\leq X$. For $E$ with $m_1(E)>X$, there are no $\xi$ contributing to the sum in \eqref{limes_inferior_all_vectors} so that the sum on the left hand side of \eqref{limes_inferior_all_vectors} equals
\begin{equation}\label{sum_fields}
X^{-\frac{5}{2}+\eps}\sum_{\substack{ E\in \FFF_3^+: \\ m_1(E)\leq X<m_2(E)}}\rho_E
N_E(X)
+X^{-\frac{5}{2}+\eps}\sum_{\substack{ E\in\FFF_3^+:\\ m_1(E)\leq m_2(E)\leq X }}\rho_E
N_E(X).
\end{equation}
By Lemma \ref{estimates_part_of_cubic_fields}\ref{estimates_part_of_cubic_fields2}, the second sum tends to $0$ for $X\rightarrow\infty$ provided $\eps<\frac{1}{2}$. 
Hence the limes inferior of the first part of the sum is not bounded from below as $X\rightarrow \infty$ for any $\eps\in(0,1/2)$.
As $m_1(E)\leq X<m_2(E)$, every $\xi\in\OOO_E/\Z$, $\xi\neq0$, with $Q_E(x)\leq X$ is of the form $\xi=n\xi_0$ for some $n\in\N$, and $\xi_0$ one of the two non-zero primitive vectors in $\OOO_E/\Z$. Note that $Q_E(\pm\xi_0)=m_1(E)$. Thus
\begin{align*}
\sum_{\substack{ E\in\FFF_3^+: \\ m_1(E)\leq X<m_2(E)}}\rho_EN_E(X)
&=\sum_{n\in\N}
\sum_{\substack{ E\in\FFF_3^+: \\ m_1(E)\leq X<m_2(E)}}\rho_E
\sum_{\substack{\xi_0\in(\OOO_E/\Z)_{\text{prim}},~¸xi_0\neq0 \\ Q_E(\xi_0)\leq \frac{X}{n^2}}}1\\
&=2\sum_{n\in\N}\sum_{\substack{ E\in\FFF_3^+:\\ m_1(E)\leq \frac{X}{n^2}<m_2(E)}}\rho_E,
\end{align*}
where $(\OOO_E/\Z)_{\text{prim}}$ denotes the set of primitive vectors in $\OOO_E/\Z$.
Suppose there are $\kappa\in(0,1/2)$ and $0<c_0<\infty$ such that
\[
\liminf_{X\rightarrow\infty}X^{-\frac{5}{2}+\kappa}\sum_{\substack{E\in\FFF_3^+: \\ m_1(E)\leq X<m_2(E)}}\rho_E=c_0<\infty.
\]
Then
\[
X^{-\frac{5}{2}+\kappa}\sum_{\substack{ E\in\FFF_3^+: \\ m_1(E)\leq X<m_2(E)}}\rho_E
N_E(X)
=2\sum_{n\in\N}n^{-5+2\kappa} (\frac{X}{n^2})^{-\frac{5}{2}+\kappa}
\sum_{\substack{E\in\FFF_3^+: \\ m_1(E)\leq \frac{X}{n^2}<m_2(E)}}\rho_E
\]
and, for every $n$, 
$\lim\inf_{X\rightarrow\infty}(\frac{X}{n^2})^{-\frac{5}{2}+\kappa}\sum_{E\in\FFF_3^+,~m_1(E)\leq \frac{X}{n^2}<m_2(E)}\rho_E=c_0$
so that the limit inferior of the above is $2c_0\zeta(5-2\kappa)$ in contradiction to the unboundedness of the limit inferior of the first sum in \eqref{sum_fields} as $X\rightarrow\infty$. This finishes the proof of the proposition.
\end{proof}

\appendix
\section{Asymptotic approximation of truncation functions}\label{appendix}
The purpose of this appendix is to prove Proposition \ref{main_approx_res} in the case of a nilpotent orbit $\NNN\subseteq \nnn\subseteq$ for $G=\GL_n$ and $G=\SL_n$ and $n\leq 3$. 
\subsection{The case $n=2$}
There are two nilpotent orbits, namely $\NNN_{\text{triv}}=\{0\}\subseteq \ggG$ and $\NNN_{\text{reg}}$ which is generated by $X_0=\left(\begin{smallmatrix}0&1\\0&0\end{smallmatrix}\right)$. For $\NNN_{\text{triv}}$ there is nothing to show so that we only consider $\NNN=\NNN_{\text{reg}}$. As noted earlier, the associated Jacobson-Morozov parabolic subgroup for $X_0$ is $P=P_0=T_0U_0$, and $C_{U_0}(X_0)=U_0$. 

We first show the following:
\begin{lemma}\label{trunc_fct_first_app}
There exist $c_1, c_2>0$ such that
\begin{equation}\label{approximationgl2}
 \big| \tilde{F}^{T_0}(t, T)-\int_{U_0(\Q)\backslash U_0(\A)} F(vt, T)~dv\big|
\leq c_1 e^{-c_2 \|T\|}
\end{equation}
for all $t\in T_0(\A)$ and all sufficiently regular $T\in\aaa^+$ with $d(T)> \delta\|T\|$. 
\end{lemma}
\begin{proof}
 Let $t\in T_0(\A)$ and let $T\in\aaa^+$ be sufficiently regular. If $\tilde{F}^{T_0}(t, T)=0$, then by definition also $F(vt, T)=0$ for all $v\in U_0(\A)$ so that
$\int_{U_0(\Q)\backslash U_0(\A)} F(vt, T)dv=0$. Hence we assume that $t$ is such that $\tilde{F}^{T_0}(t, T)=1$. 
Then the left hand side of \eqref{approximationgl2} equals
\begin{align*}
& \vol\{v\in U_0(\Q)\backslash U_0(\A)\mid \exists \gamma\in G(\Q):~ \varpi(H_0(\gamma vt)-T)>0\}\\
& \leq\vol\{v\in U_0(\Q)\backslash U_0(\A)\mid \exists u\in U_0(\Q):~ \varpi(H_0(u vt)-T)>0\}\\
& +\vol\{v\in U_0(\Q)\backslash U_0(\A)\mid \exists u\in U_0(\Q):~ \varpi(H_0(wu vt)-T)>0\}
\end{align*}
for $\varpi $ the unique element in $\widehat{\Delta}_0$, and $w=\left(\begin{smallmatrix} 0&1\\1&0\end{smallmatrix}\right)$ a representative for the non-trivial Weyl group element. Here we used the left $P_0(\Q)$-invariance of $H_0$. Using again the left $U_0(\Q)$-invariance and that $\tilde{F}^{T_0}(t, T)=1$, the volume of the first set is $0$ so that we only need to estimate
\[
\vol\{v\in U_0(\Q)\backslash U_0(\A)\mid \exists u\in U_0(\Q):~ \varpi(H_0(wu vt)-T)>0\}.
\]
For that write  $u=\left(\begin{smallmatrix} 1&x\\0&1\end{smallmatrix}\right)\in U_0(\Q)$ and  $v=\left(\begin{smallmatrix}1&y\\0&1\end{smallmatrix}\right)\in U_0(\Q)\backslash U_0(\A)$.
Then
\[
 \varpi(H_0(wu vt))
=\varpi(wH_0(t)) + \varpi(H_0(w t^{-1} uv t))
= -\varpi(H_0(t)) - \log\|(1, e^{-\alpha(t)}(x+y))\|_{\A}
\]
for $\alpha$ the unique simple root in $\Delta_0$ and $\|\cdot\|_{\A}$ the adelic vector norm. As $[0,1]\subseteq \R$ is a fundamental domain for $\Q\hookrightarrow\A$, it therefore suffices to estimate the volume of the set
\[
 \{y\in [0,1]\mid \exists x\in\Q:~ \|(1, e^{-\alpha(t)}(x+y))\|_{\A}<e^{-\varpi(T)-\varpi(H_0(t))}\}.
\]
Now 
\[
 \|(1, e^{-\alpha(t)}(x+y))\|_{\A}
= (1+ e^{-2\alpha(t)}(x+y)^2)^{1/2}\prod_{p<\infty}\max\{1, |x|_p\}
= \big(1+ e^{-2\alpha(t)}(\frac{p}{q}+y)^2\big)^{1/2}q
\]
if we write $x=\frac{p}{q}$ for $p,q$ coprime integers, $q>0$.
For the inequality
\begin{align*}
 (1+ e^{-2\alpha(t)}(\frac{p}{q}+y)^2)^{1/2}q & <e^{-\varpi(T)-\varpi(H_0(t))}\\
\Leftrightarrow
(\frac{p}{q}+y)^2< q^{-2}e^{2\alpha(t)-2\varpi(T)-2\varpi(H_0(t))}- e^{2\alpha(t)}
& = q^{-2} e^{2\varpi(H_0(t))-2\varpi(T)}-e^{2\alpha(t)}
\end{align*}
to have a solution in $y\in[0,1]$, we must necessarily have $0<q<e^{-\varpi(T)-\varpi(H_0(t))}$ and $-q<p<q$. Hence the volume of the above set is bounded by a constant multiple of 
\begin{align*}
 \sum_{\substack{q:\\ 0<q<e^{-\varpi(T)-\varpi(H_0(t))}}}\sum_{\substack{p: \\ -q<p<q}} \sqrt{q^{-2} e^{2\varpi(H_0(t))-2\varpi(T)}-e^{2\alpha(t)}}
& \leq  2\sum_{\substack{q: \\ 0<q<e^{-\varpi(T)+\varpi(H_0(t))}}} e^{\varpi(H_0(t))-\varpi(T)}\\
\leq 2e^{-\varpi(T)-\varpi(H_0(t))}e^{\varpi(H_0(t))-\varpi(T)} & =2e^{-2\varpi(T)}=2e^{-\alpha(T)}.
\end{align*}
Hence there is $\eps>0$ such that for all $t\in T_0(\A)$ we have
\[
  \big| \tilde{F}^{T_0}(t, T)-\int_{U_0(\Q)\backslash U_0(\A)} F(vt, T)~dv\big|
\leq 2e^{-\alpha(T)}\leq 2 e^{-\eps\|T\|}
\]
for all sufficiently regular $T$ with $d(T)>\delta\|T\|$ so that the lemma is proved.
\end{proof}

\begin{cor}
Let $\nu$ be as in Lemma \ref{fundamental_convergence}.
Then there exists a seminorm $\mu$ on $\SSS^{\nu}(\ggG(\A))$ such that for every $\Phi\in \SSS^{\nu}(\ggG(\A))$ and nilpotent orbit $\NNN\subseteq\ggG(\Q)$, we have 
\[
 \big|j_{\NNN}^T(\Phi)-\tilde{j}_{\NNN}^T(\Phi)\big|
\leq \mu(\Phi) e^{-c_2\|T\|}
\]
for all sufficiently regular $T\in\aaa^+$ with $d(T)>\delta \|T\|$.
\end{cor}
\begin{proof}
 As before, we only need to consider $\NNN=\NNN_{\text{reg}}$. Let $X_0=\left(\begin{smallmatrix}0&1\\0&0\end{smallmatrix}\right)\in\NNN_{\text{reg}}$. Let $\Phi\in\SSS^{\nu}(\ggG(\A))$. We may assume that $\Phi$ is $\cpt$-conjugation invariant.
Then
\[
 j_{\NNN_{\text{reg}}}^T(\Phi)
= \int_{A_0^G} \delta_0(a)^{-1} \bigg(\int_{U_0(\Q)\backslash U_0(\A)} F(ua, T)du\bigg) \sum_{X\in\uuu_0(\Q)\cap \NNN_{\text{reg}}}\Phi(\Ad a^{-1} X) ~da.
\]
Note that $\tilde{F}^{T_0}(a, T)=\hat{\tau}_0^G(T-H_0(a))=0$ implies $F(ua, T)=0$ for all $u\in U_0(\A)$, i.e., $F(ua, T)\leq \tilde{F}^{T_0}(a, T)$ for all $u$ and $a$. Using Lemma \ref{trunc_fct_first_app} (and the notation introduced there), we therefore get
\[
 \big| j_{\NNN_{\text{reg}}}^T(\Phi)- \tilde{j}_{\NNN_{\text{reg}}}^T(\Phi)\big|
\leq c_1e^{-c_2\|T\|}\int_{A_0^G} \hat{\tau}_0^G(T-H_0(a)) \sum_{X\in\uuu_0(\Q)\cap \NNN_{\text{reg}}}\big|\Phi(\Ad a^{-1} X)\big| ~da
\]
for all sufficiently regular $T$ with $d(T)>\delta\|T\|$. Since $\Ad a^{-1} X=\left(\begin{smallmatrix} 0&a^{-2} x\\0&0\end{smallmatrix}\right)$ for  $x\in\Q\minus\{0\}$ with $X=\left(\begin{smallmatrix} 0&x\\0&0\end{smallmatrix}\right)$, and $a=\diag(a, a^{-1})\in A_0^G$, it suffices to consider the case that $\big|\Phi(\Ad a^{-1} X)\big|\leq \varphi(a^{-2} x)$ for a suitable $\varphi\in \SSS(\ggG(\A))$, $\varphi\geq0$, which satisfies the seminorm estimates with respect to $\Phi$ from Lemma \ref{std_estimates_for_SB}. 
 Now
\begin{align*}
& \int_{A_0^G} \hat{\tau}_0^G(T-H_0(a)) \sum_{X\in\uuu_0(\Q)\cap \NNN_{\text{reg}}}\big|\Phi(\Ad a^{-1} X)\big| ~da
\leq \int_0^{e^{\alpha(T)}/2} a^{-2} \sum_{x\in\Q\minus\{0\}} \varphi(a^{-2}x)~d^{\times} a\\
& =\int_0^1a^{-3} \sum_{x\in\Q\minus\{0\}} \varphi(a^{-2}x)~da
+ \int_1^{\alpha(T)/2}a^{-3} \sum_{x\in\Q\minus\{0\}} \varphi(a^{-2}x)~da\\
& \leq
\mu_1(\Phi)\big( \int_0^1a ~da
+ \int_1^{\alpha(T)/2}a^{-1}~da\big)
=\mu_2(\Phi)(1+ T)
\end{align*}
for suitable seminorms $\mu_1, \mu_2$, where we used the standard estimates  for Schwartz-Bruhat functions. This proves the corollary.
\end{proof}

\subsection{The case $n=3$}
There are now three different nilpotent orbits in $\ggG$: The trivial orbit $\NNN_{\text{triv}}=\{0\}$, the minimal orbit $\NNN_{\text{min}}$ generated by $X_{\text{min}}=\left(\begin{smallmatrix} 0&0&1\\0&0&0\\0&0&0\end{smallmatrix}\right)$, and the regular orbit $\NNN_{\text{reg}}$ generated by $X_{\text{reg}}=\left(\begin{smallmatrix} 0&1&0\\0&0&1\\0&0&0\end{smallmatrix}\right)$. The first case again is trivial so that we only need to consider the other two. In the last two cases the associated Jacobson-Morozov parabolic is the minimal parabolic. 
\begin{lemma}\label{trunc_app2}
 There are $c_1, c_2>0$ such that for every $X_0\in \{X_{\text{min}}, X_{\text{reg}}\}$ and $v'\in C_{U_0}(X_0, \A)\backslash U_0(\A)$ we have
\begin{equation}\label{approximation2}
 \big| \tilde{F}^{T_0}(t, T)-\int_{C_{U_0}(X_0,\Q)\backslash C_{U_0}(X_0,\A)} F(vv't, T)dv\big|
\leq c_1 e^{-c_2 \|T\|}
\end{equation}
for all $t\in T_0(\A)$ and all sufficiently regular $T\in\aaa^+$ with $d(T)>\delta\|T\|$.
\end{lemma}
\begin{proof}
 We split the proof of the lemma according to the two orbits. Write $\Delta_0=\{\alpha_1, \alpha_2\}$ such that $\alpha(\diag(t_1,t_2,t_3))=|t_1/t_2|$ and $\alpha_2(\diag(t_1, t_2, t_3))=|t_2/t_3|$.

\paragraph*{\underline{{\bf $\NNN=\NNN_{\text{min}}$:}}}
 Write $X_0=X_{\text{min}}$. Then $C_{U_0}(X_0)=U_0$ so that $v'=1$. Let $t\in T_0(\A)$. It is clear that $\tilde{F}^{T_0}(t, T)=0$ again implies that $F(vt, T)=0$ for all $v\in U_0(\A)\backslash U_0(\A)$. Hence we again assume that $t$ is such that $\tilde{F}^{T_0}(t, T)=1$.
To estimate the left hand side of \eqref{approximation2} it will therefore suffice to bound the volume of the set
\[
 \{v\in U_0(\Q)\backslash U_0(\A)\mid \exists \gamma\in G(\Q)~\exists \varpi\in\widehat{\Delta}_0:~\varpi(H_0(\gamma vt)-T)> 0\}. 
\]
Using Bruhat decomposition for $G(\Q)$ and the left $P_0(\Q)$-invariance of $H_0$, it suffices to bound for each $w\in W$ and $\varpi\in\widehat{\Delta}_0$ the volume of the set
\[
 V_T(w, \varpi, T)= \{v\in U_0(\Q)\backslash U_0(\A)\mid \exists u\in U_0(\Q):~\varpi(H_0(wu vt)-T)> 0\}.
\]
Now for $v\in U_0(\Q)\backslash U_0(\A)$ and $u\in U_0(\Q)$  we have 
$
H_0(wu vt)
=H_0((wtw^{-1})(wt^{-1} u v t))
=wH_0(t) +H_0(wt^{-1} u v t)
$
so that
\[
 \varpi(H_0(wu vt)-T)>0
 \Leftrightarrow
 \varpi(H_0(wt^{-1}u vt))>\varpi(T -w H_0(t))
 \Leftrightarrow
 e^{-\varpi(H_0(wt^{-1}u vt))}< e^{\varpi(wH_0(t)-T)}.
\]
Hence $ \vol V_T(w, \varpi, t)$ equals
\[
\vol\big(\{x_1, x_2, x_3\in[0,1]\mid \exists y_1, y_2, y_3\in\Q:~ - \varpi(H_0(wt^{-1}\left(\begin{smallmatrix} 1& x_1+y_1&x_2+y_2\\ \;&1&x_3+y_3 \\\;&\;&1\end{smallmatrix}\right) t))< -\varpi(T -w H_0(t))\}\big)
\]
Suppose $u=\left(\begin{smallmatrix} 1&u_1&u_2\\\;&1&u_3\\\;&\;&1\end{smallmatrix}\right)\in U_0(\A)$. We first want to compute the last two rows of $wuw^{-1}$, as they can be used to compute $\varpi(H_0(w u w^{-1}))$.
\begin{itemize}
 \item $w=w_1=\id$, then the last two columns equal
 \[
  \begin{pmatrix} 0&1&u_3\\ 0&0&1\end{pmatrix}.
 \]
 \item $w=w_2$ is the simple reflexion about the root $\alpha_1$. Then the last two rows equal
 \[
    \begin{pmatrix} u_1&1&u_2\\ 0&0&1\end{pmatrix}
 \]
 
\item $w=w_3$ is the simple reflexion about the root $\alpha_2$. Then the last two rows equal
\[
   \begin{pmatrix} 0&1&0\\ 0&u_3&1\end{pmatrix}
\]

\item $w=w_4$ is the longest Weyl element. Then the last two rows equal
\[
   \begin{pmatrix} u_3&1&0\\ u_2&u_1&1\end{pmatrix}
\]

\item $w=w_5$ is represented by $\left(\begin{smallmatrix} 0&0&1\\1&0&0\\0&1&0\end{smallmatrix}\right)$. Then the last two rows equal
\[
   \begin{pmatrix} u_2&1&u_1\\ u_3&0&1\end{pmatrix}
\]

\item $w=w_6=w_5^{-1}$. Then the last two rows equal
\[
   \begin{pmatrix} 0&1&0\\ u_1&u_2&1\end{pmatrix}.
\]
\end{itemize}

{\bf The case $\varpi=\varpi_2$:}
Using the above computations, we have for $u:=\left(\begin{smallmatrix} 1& x_1+y_1&x_2+y_2\\ \;&1&x_3+y_3 \\\;&\;&1\end{smallmatrix}\right)$
\[
  e^{-\varpi_2(H_0(w_it^{-1}ut))}=
  \begin{cases}
   \|(0,0,1)\|_{\A}=1												&\text{if }i\in\{1,2\},\\
   \|(0,e^{-\alpha_2(t)}(x_3+y_3), 1)\|_{\A}									&\text{if }i\in\{3,5\},\\
   \|(e^{-( \alpha_1+\alpha_2)(t)}(x_2+y_2), e^{-\alpha_1(t)}(x_1+y_1),1)\|_{\A}				&\text{if }i\in\{4,6\}.
  \end{cases}
\]
Since $\tilde{F}^{T_0}(t, T)\neq0$, we have $\varpi_2(T -w H_0(t))=\varpi_2(T - H_0(t))\leq 0$ so that $\vol V_T(w_1, \varpi_2,t)=\vol V_T(w_2, \varpi_2, t)=0$.

Now if $w\in\{w_3, w_5\}$ we have $\varpi_2(wH_0(t))=(\varpi_1-\varpi_2)(H_0(t))$, and therefore
\[
   e^{-\varpi_2(H_0(wt^{-1}ut))}< e^{\varpi_2(wH_0(t)-T)}
   \Leftrightarrow
      \|(0,e^{-\alpha_2(t)}(x_3+y_3), 1)\|_{\A}<   e^{(\varpi_1-\varpi_2)(H_0(t))-\varpi_2(T)}.
\]
Writing out the adelic norm on the left hand side, this is equivalent to (recall that $x_3\in[0,1]$)
\[
 (1+e^{-2\alpha_2(t)}(x_3+y_3)^2)^{1/2}\prod_{p<\infty} \max\{1,|y_3|_p\}<e^{(\varpi_1-\varpi_2)(H_0(t))-\varpi_2(T)}.
\]
We can write $y_3=a/b$ with $a, b$ coprime integers. Then $\prod_{p<\infty} \max\{1,|y_3|_p\}= |b|$ so that the above is equivalent to
\begin{align*}
 1+e^{-2\alpha_2(t)}(x_3+y_3)^2 & <b^{-2}e^{2(\varpi_1-\varpi_2)(H_0(t))-2\varpi_2(T)}\\
 \Leftrightarrow\;\;\;
 (x_3+\frac{a}{b})^2 & <\left[b^{-2}e^{2(\varpi_1-\varpi_2)(H_0(t))-2\varpi_2(T)}-1\right]e^{2\alpha_2(t)}.
\end{align*}
If there exists $x_3$ satisfying this inequality we must necessarily have $e^{(\varpi_1-\varpi_2)(H_0(t))-\varpi_2(T)}>1$ and $|b|<e^{(\varpi_1-\varpi_2)(H_0(t))-\varpi_2(T)}$. It moreover suffices to consider $0\leq a\leq b $, since if for $a> b$ there still exists $x_3$ as before, then the volume of $V_T(w, \varpi_2, t)$ equals $1$.
Hence the volume of all $x_3\in[0,1]$ for which there exists $y_3\in \Q$ as above is bounded by
\[
 \sum_{0<b<e^{(\varpi_1-\varpi_2)(H_0(t))-\varpi_2(T)}} \sum_{0\leq a<b} b^{-1}e^{(\varpi_1-\varpi_2)(H_0(t))-\varpi_2(T)}e^{\alpha_2(t)}
 \leq e^{2(\varpi_1-\varpi_2)(H_0(t))-2\varpi_2(T)}e^{\alpha_2(t)}.
\]
Note that $ 2( \varpi_1-\varpi_2)+\alpha_2=\varpi_1$
so that, since $\varpi_1(H_0(t))\leq \varpi_1(T)$ by assumption, we get
\[
 \vol V_T(w, \varpi_2, t)
 \leq e^{-\alpha_2(T)}
\]
for $w\in\{w_3, w_5\}$.

Now if $w\in\{w_4, w_6\}$, we have $\varpi_2(wH_0(t))=-\varpi_1(H_0(t))$. Therefore,
\begin{align*}
   & e^{-\varpi_2(H_0(wt^{-1}ut))}  < e^{\varpi_2(wH_0(t)-T)}\\
   \Leftrightarrow\;\;\;\;
 & \|(e^{-( \alpha_1+\alpha_2)(t)}(x_2+y_2), e^{-\alpha_1(t)}(x_1+y_1),1)\|_{\A}  <   e^{-\varpi_1(H_0(t))-\varpi_2(T)}.
\end{align*}
This is equivalent to 
\[
 (1+e^{-2\alpha_1(t)}(x_1+y_1)^2+ e^{-2( \alpha_1+\alpha_2)(t)}(x_2+y_2)^2)^{1/2}\prod_{p<\infty} \max\{1, |y_1|_p, |y_2|_p\}
 < e^{-\varpi_1(H_0(t))-\varpi_2(T)}.
\]
Write $y_i=a_i/b_i$ with $a_i, b_i$ coprime integers. Then $\prod_{p<\infty} \max\{1, |y_1|_p, |y_2|_p\}=\lcm(b_1, b_2)=:b$, and as above it suffices to consider $0\leq a_1, a_2<b< e^{-\varpi_1(H_0(t))-\varpi_2(T)}$.
Hence the volume of $V_T(w, \varpi_2, t)$ is bounded by the sum over all such $a_1, a_2, b$ of the volume of all $x_1, x_2\in [0, 1]$ satisfying
\[
 e^{-2\alpha_1(t)}(x_1+\frac{a_1}{b})^2+ e^{-2( \alpha_1+\alpha_2)(t)}(x_2+\frac{a_2}{b})^2
 <b^{-2} e^{-2\varpi_1(H_0(t))-2\varpi_2(T)}-1
\]
so that for $w\in\{w_4, w_6\}$ we have
\begin{align*}
 \vol V_T(w, \varpi_2, t)
&\leq 
 \sum_{0<b<e^{-\varpi_1(H_0(t))-\varpi_2(T)}} b e^{\alpha_1(t)} e^{( \alpha_1+\alpha_2)(t)} e^{-\varpi_1(H_0(t))-\varpi_2(T)}\\
 &\leq  e^{\alpha_1(t)} e^{( \alpha_1+\alpha_2)(t)} e^{-3\varpi_1(H_0(t))-3\varpi_2(T)}
 =e^{-3\varpi_2(T)}.
\end{align*}

{\bf The case $\varpi=\varpi_1$:}
Using the same notation as before, we can compute 
\begin{multline*}
   e^{-\varpi_1(H_0(w_it^{-1}ut))}=\\
\begin{cases}
        \|(0,0,1)\|_{\A}=1																						&\text{if } i\in\{1,3\},\\
	\|(0, 1, e^{-\alpha_1(t)}(x_1+y_1))\|_{\A}																		&\text{if }i\in\{2,6\},\\
	\|(1, e^{-\alpha_2(t)}(x_3+y_3), e^{-(\alpha_1+\alpha_2)(t)} \big((x_1+y_1)(x_3+y_3)-(x_2+y_2)\big))\|_{\A}					&\text{if }i\in\{4,5\}.
    \end{cases}
\end{multline*}
If $w\in\{w_1, w_3\}$ it follows as before that $V_T(w, \varpi_1, t)=0$. 
If $w\in \{w_2, w_6\}$, then $\varpi_1(wH_0(t))=(\varpi_2-\varpi_1)(H_0(t))$, and  it follows as before that $\vol V_T(w, \varpi_1, t)$ is bounded from above by
\[
 e^{2(\varpi_2-\varpi_1)(H_0(t))-2\varpi_1(T)}e^{\alpha_1(t)}
 \leq e^{-\alpha_1(T)}
\]
by our assumption on $t$.

For the last case $w\in\{w_4, w_5\}$ we have $\varpi_1(wH_0(t))=-\varpi_2(H_0(t))$ so that
\[
    e^{-\varpi_1(H_0(wt^{-1}ut))}<e^{\varpi_1(wH_0(t)-T)}
\]
is equivalent to 
\[
 \|(1, e^{-\alpha_2(t)}(x_3+y_3), e^{-(\alpha_1+\alpha_2)(t)} \big((x_1+y_1)(x_3+y_3)-(x_2+y_2)\big))\|_{\A}
 <e^{-\varpi_2(H_0(t))-\varpi_1(T)}.
\]
It follows similarly as before (we may replace  $(x_1+y_1)(x_3+y_3)-(x_2+y_2)$ by $x_2+y_2$ for our purposes) that the volume $\vol V_T(w, \varpi_1, t)$ is bounded by
\[
 e^{\alpha_2(t)} e^{( \alpha_1+\alpha_2)(t)} e^{-3\varpi_2(H_0(t))-3\varpi_1(T)}
 =e^{-3\varpi_1(T)}
\]
finishing the case $\varpi=\varpi_1$.

Taking all computations for $\varpi=\varpi_1, \varpi_2$ together, we obtain
 \[
  \bigg|\tilde{F}^{T_0}(t, T)-\int_{C_U(X_0,\Q)\backslash C_U(X_0, \A)} F(vt, T)dv\bigg|
  \leq 2\big(e^{-\alpha_1(T)}+e^{-\alpha_2(T)}+e^{-3\varpi_1(T)}+e^{-3\varpi_2(T)})
  \leq  8e^{-d(T)}
 \]
 for all $t\in T_0(\Q)\backslash T_0(\A)$. For $d(T)>\delta \|T\|$ the assertion follows.

\paragraph*{\underline{{\bf $\NNN=\NNN_{\text{reg}}$:}}}
Let $t\in T_0(\A)$ be again such that $\tilde{F}^{T_0}(t, T)=1$. For the representative $X_0=\left(\begin{smallmatrix} 0&1&0\\0&0&1\\0&0&0\end{smallmatrix}\right)$ of $\OOO_{\text{reg}}$, the Jacobson-Morozov parabolic subgroup is again $P=P_0$, and
\[
 C_{U_0}(X_0)=\{\left(\begin{smallmatrix} 1&a&b\\0&1&a\\0&0&1\end{smallmatrix}\right)\}.
\]
As a complement of $C_{U_0}(X_0)\subseteq U_0$ we choose the subspace $V:=\{\left(\begin{smallmatrix}1&c&0\\0&1&-c\\0&0&1\end{smallmatrix}\right)\}\subseteq U_0$.
Let $v'=v'(c)=\left(\begin{smallmatrix}1&c&0\\0&1&-c\\0&0&1\end{smallmatrix}\right)\\in V(\A)$ be fixed.
We want to approximate the sets
\[
 V_T(\varpi, t, v')
 =\{v\in C_{U_0}(X_0, \Q)\backslash C_{U_0}(X_0, \A)\mid \exists \gamma\in G(\Q):~ \varpi(H_0(\gamma vv't)-T)>0\}
\]
for each $\varpi\in\{\varpi_1, \varpi_2\}$.  We split this set into disjoint sets $V_T(w, \varpi, t, v')$ for $w\in W$ according to the Bruhat decomposition as before.

{\bf The case $\varpi=\varpi_2$:}
If applicable, we use the same notation as in the case of the minimal orbit, but now write $x_1=a+c$, $x_3=a-c$, and $x_2=b-ac$ with $c$ fixed and $a, b\in\Q\backslash \A$.
Hence
\[
  e^{-\varpi_2(H_0(wt^{-1}yvv't))}=
  \begin{cases}
   \|(0,0,1)\|_{\A}=1												&\text{if }w\in\{w_1, w_2\},\\
   \|(0,e^{-\alpha_2(t)}(a-c+y_3), 1)\|_{\A}									&\text{if }w\in\{w_3, w_5\},\\
   \|(e^{-( \alpha_1+\alpha_2)(t)}(b-ac+y_2), e^{-\alpha_1(t)}(a+c+y_1),1)\|_{\A}				&\text{if }w\in\{w_4, w_6\}.
  \end{cases}
\]
The first case $w\in\{w_1, w_2\}$ again leads to $\vol V_T(w, \varpi_2, t, v')=0$ for every $t$ with $\tilde{F}^{T_0}(t, T)=1$.

If $w\in \{w_3, w_5\}$, we now choose a fundamental domain for $a$ as $[c, 1+c]$ so that this case can in fact be treated similar to the minimal orbit. Hence
\[
 \vol V_T(w,\varpi_2,  t, v')\leq e^{-\alpha_2(T)}.
\]

Similarly, if $w\in\{w_4, w_6\}$ we can choose the fundamental domains for $a$ and $b$ in such a way that we are left with the same type of estimates as in the case of the minimal orbit. Hence
\[
 \vol V_T(w, \varpi_2, t, v')\leq e^{-3\varpi_2(T)}.
\]

{\bf The case $\varpi=\varpi_1$:}
As for the minimal orbit, we obtain
\begin{multline*}
   e^{-\varpi_1(H_0(w_it^{-1}yvv't))}=\\
\begin{cases}
        \|(0,0,1)\|_{\A}=1																						&\text{if } i\in\{1,3\},\\
	\|(0, 1, e^{-\alpha_1(t)}(a+c+y_1))\|_{\A}																		&\text{if }i\in\{2,6\},\\
	\|(1, e^{-\alpha_2(t)}(a-c+y_3), e^{-(\alpha_1+\alpha_2)(t)} \big((a+c+y_1)(a-c+y_3)-(b-ac+y_2)\big))\|_{\A}					&\text{if }i\in\{4,5\}.
    \end{cases}
\end{multline*}
Choosing for each $w$ appropriate fundamental domains for $a$ and $b$, we are left with the same computations and estimates as in the minimal orbit case.

Taking everything together, we again obtain:
 For the regular unipotent orbit with Jacobson-Morozov parabolic $P=P_0$ we can approximate $\int_{C_U(u_0,\Q)\backslash C_U(u_0, \A)} F(vt, T)~dv$ by $\tilde{F}^{T_0}(t, T)$ asymptotically in $T$, in fact,
 \[
  \bigg|\tilde{F}^{T_0}(t, T)-\int_{C_U(X_0,\Q)\backslash C_U(X_0, \A)} F(vt, T)dv\bigg|
  \leq  8e^{-d(T)}
 \]
 for all $t\in T_0(\Q)\backslash T_0(\A)$. For $d(T)>\delta \|T\|$ the assertion follows.
\end{proof}

\begin{cor}
Let  $\nu>0$ be as in Lemma \ref{fundamental_convergence}. There exists a seminorm $\mu$ on $\SSS^{\nu}(\ggG(\A))$ such that for every $\Phi\in\SSS^{\nu}(\ggG(\A))$ and every nilpotent orbit $\NNN\subseteq\ggG(\Q)$ we have
\[
 \big|j_{\NNN}^T(\Phi)-\tilde{j}_{\NNN}^T(\Phi)\big|
\leq \mu(\Phi)e^{-c_2\|T\|}
\]
for every sufficiently regular $T\in\aaa^+$ with $d(T)>\delta\|T\|$.
\end{cor}
\begin{proof}
 Again, we only need to consider the non-trivial orbits, and we moreover may assume that $\Phi$ is $\cpt$-conjugation invariant. 
First consider the regular orbit $\NNN=\NNN_{\text{reg}}$ and $X_0=X_{\text{reg}}$. Using the results and notation of Lemma \ref{trunc_app2} and proceeding similar as in the $n=2$-case, we can bound $ \big|j_{\NNN}^T(\Phi)-\tilde{j}_{\NNN}^T(\Phi)\big|$ from above by
\begin{multline*}
\leq c_1 e^{-c_2\|T\|}\int_{A_0^G} \delta_0(a)^{-1} \hat{\tau}_0^G(T-H_0(a)) \int_{\uuu^{>2}(\A)} \sum_{X\in\uuu^2(\Q)\cap\NNN} \big|\Phi(\Ad a^{-1} (X+U))\big|~dU~ da\\
= c_1 e^{-c_2\|T\|}\int_{A_0^G} \delta_{\uuu^2}(a)^{-1} \hat{\tau}_0^G(T-H_0(a)) \int_{\uuu^{>2}(\A)} \sum_{X\in\uuu^2(\Q)\cap\NNN} \big|\Phi_{>2}(\Ad a^{-1} X)\big| ~da,
\end{multline*}
where $\Phi_{>2}(X):=\int_{\uuu^2(\A)} \Phi(X+U)dU$.
Again, we may assume that 
\[|\Phi_{>2}(\Ad a^{-1} X)|\leq \varphi(a_1^{-2}a_2 x_1)\varphi(a_1a_2^{-2} x_2)\]
 for a sufficiently rapidly decaying function $\varphi$ with seminorms bounded in terms of $\Phi$. Here we write $X=\left(\begin{smallmatrix} 0&x_1&0\\0&0&x_2\\0&0&0\end{smallmatrix} \right)$, $x_1, x_2\in\Q\minus\{0\}$, and $a_i=\varpi_i(H_0(a))$. Hence the above is bounded by
\[
c_1 e^{-c_2\|T\|}\int_{0}^{e^{\varpi_1(T)}/2}\int_0^{e^{\varpi_2(T)}/2} a_1^{-1}a_2^{-1}  \sum_{\substack{x_1\in\Q\minus\{0\}, \\ x_2\in\Q\minus\{0\}}} \varphi(a_1^{-2}a_2 x_1) \varphi(a_1 a_2^{-2} x_2) ~d^{\times}a_2~d^{\times}a_1 .
\]
Considering the cases $a_1^{-2} a_2\gtrless1$ and $a_1 a_2^{-2}\gtrless 1$ separately, we see that the integral is again bounded by a seminorm $\mu(\Phi)$ and a (quadratic) polynomial in $T$ so that this case is finished.

Now consider the case $\NNN=\NNN_{\text{min}}$ and $X_0=X_{\text{min}}$. Similar as before, we are left to estimate
\begin{multline*}
 c_1 e^{-c_2\|T\|}\int_{A_0^G} \delta_{U^{\leq 2}}(a)^{-1} \hat{\tau}_0^G(T-H_0(a)) \sum_{X\in\uuu^2(\Q)\cap\NNN} \big|\Phi(\Ad a^{-1} X)\big| ~da\\
\leq c_1 e^{-c_2\|T\|}\int_{0}^{e^{\varpi_1(T)}/2}\int_0^{e^{\varpi_2(T)}/2}
 a_1^{-2}a_2^{-2} \sum_{x\in\Q\minus\{0\}} \varphi(a_1^{-1}a_2^{-1} x) ~d^{\times} a_1 ~d^{\times} a_2,
\end{multline*}
for $\varphi$ a suitable function.
If we change one of the variables to $a_1a_2$, we can analyse the integral similar as before to obtain the assertion.
\end{proof}

\section{A sequence of test functions}\label{appendixb}
In this appendix, we give a sequence of test functions at the non-archimedean places which might be useful to deduce Conjecture \ref{conj_intr} from Proposition \ref{asymptotics_orbital_integrals}. 

For a prime $p$ define $\tilde{\Phi}_p:\ggG(\Q_p)\longrightarrow \C$  by
\[
\tilde{\Phi}_p(x)=\begin{cases}
		  \frac{[\OOO_{\Q_p[x]}:\Z_p[x]]}{I_p(\Phi_p^0, x)}=c(\Phi_p^0, x)^{-1}		&\text{if }\Delta(x)\neq0,\text{ and }x\in\ggG(\Z_p),\\
		  0										&\text{else}.
                  \end{cases}
\]
Then $\tilde{\Phi}_p$ is locally constant in $\ggG(\Q_p)\minus\{x\in\ggG(\Q_p)\mid \Delta(x)=0\}$, but not on all of $\ggG(\Q_p)$.
For $x\in \ggG(\Q_p)$ with $\tilde{\Phi}_p(x)\neq0$, we have
\[
c(\tilde{\Phi}_p, x)=
\frac{1}{[\OOO_{\Q_p[x]}:\Z_p[x]]}\int_{G_x(\Q_p)\backslash G(\Q_p)}\tilde{\Phi}_p(g^{-1}xg)~dg=1
\]
so that in fact one would actually like to use $\tilde{\Phi}_f:=\prod_{p<\infty}\tilde{\Phi}_p$ as a test function at the archimedean places, which we are not allowed to do because of $\tilde{\Phi}_f\not\in\SSS(\ggG(\A_f))$.  

However, we can construct a sequence of functions in $\SSS(\ggG(\A_f))$ converging to $\tilde{\Phi}_f$:
Let $\Sigma\subseteq \ggG(\Z_p)$ denote the set of all $x\in\ggG(\Z_p)$ such that $\Delta(x)=0$.
For $m\in\N_0$ define a function $\Phi_p^m:\ggG(\Q_p)\longrightarrow \C$ by
\[
\Phi_p^m(x)=\begin{cases}
             1						&\text{if }x\in \Sigma+p^m\ggG(\Z_p),\\	
	    \tilde{\Phi}_p(x)				&\text{if }x\not\in \Sigma+p^m\ggG(\Z_p).
            \end{cases}
\]
In particular, $\Phi_p^0$ coincides with the characteristic function of $\ggG(\Z_p)$. By construction $\Phi_p^m\in\SSS(\ggG(\Q_p))$ and $\Phi_p^m$ is $\cpt_p$-invariant.
 Let $\mmm=(m_p)_{p<\infty}$ be a sequence of integers $m_p\in\N_0$ of which almost all are zero. Let $\DIV^+(\Q)$ denote the set of all such sequences. It has a partial order given by $\mmm\geq\mmm'$ if and only if $m_p\geq m_p'$ for all primes $p$.
Define the function  $\Phi_f^{\mmm}:\ggG(\A_f)\longrightarrow \C$ by
$\Phi_f^{\mmm}=\prod_{p<\infty}\Phi_p^{m_p}$.
Then $\Phi_f\in\SSS(\ggG(\A_f))$ and it is $\cpt_p$-invariant.

By definition we have for all $\mmm,\mmm'\in \DIV^+(\Q)$ with $\mmm\geq\mmm'$ and all $x\in \ggG(\A_f)$  we have
\[
0\leq\tilde{\Phi}_f(x)
\leq \Phi_f^{\mmm}(x)
\leq \Phi_f^{\mmm'}(x)
\leq \Phi^0_f(x)
\leq 1.
\]
Moreover, $\lim_{\mmm}\Phi_f^{\mmm}(x)=\tilde{\Phi}_f(x)$ for every $x$. Similarly, the functions $\Phi^{m_p}_p$ are monotonically decreasing with limit function $\tilde{\Phi}_p$ so that 
$\lim_{m_p\rightarrow\infty}\int_{\ggG(\Q_p)} \Phi_{p}^{m_p}(x)~dx
=\int_{\ggG(\Q_p)}\tilde{\Phi}_p(x)~dx$
and
\[\lim_{m_p\rightarrow\infty}\int_{G_{\gamma}(\Q_p)\backslash G(\Q_p)}\Phi_p^{m_p}(g^{-1}\gamma g)~dx
=\int_{G_{\gamma}(\Q_p)\backslash G(\Q_p)}\tilde{\Phi}_p(g^{-1}\gamma g)~dx
=1\]
 for all regular elliptic $\gamma$. 
The existence of these limits does not suffice to pass from $c(\xi,\Phi_f)$ to $1$ in the asymptotic \ref{asymptotics_orbital_integrals} which would prove Conjecture \ref{conj_intr}. It would be necessary to show uniformity of the convergence in  $Q(\gamma)=\tr \gamma^2-\frac{1}{3}(\tr \gamma)^2 $ and the number of primes for which $m_p\neq0$.

\bibliographystyle{amsalpha}
\bibliography{bibzetafcts}
\end{document}